\documentclass{amsart}

\usepackage[utf8]{inputenc}
\usepackage[top=1.25in, bottom=1.25in, left=1.1in, right=1.1in]{geometry, mathtools}
\usepackage[all]{xy}
\usepackage[toc,page]{appendix}
\usepackage{amsthm, amsmath, amssymb, amscd, latexsym, multicol, verbatim, enumerate, graphicx,xy, color}
\usepackage{mathrsfs}
\usepackage{tikz,stackrel}
\usepackage{mathdots}
\usepackage[all]{xy}
\usetikzlibrary{calc}
\usetikzlibrary{matrix,arrows,decorations.pathmorphing}
\usetikzlibrary{intersections}
\usetikzlibrary{decorations.markings}
\usetikzlibrary{external}
\usepackage{setspace}

\usepackage{tikz-cd}
\usetikzlibrary{decorations.markings}
\makeatletter
\tikzcdset{
open/.code={\tikzcdset{hook, circled};},
closed/.code={\tikzcdset{hook, slashed};},
circled/.code={\tikzcdset{markwith={\draw (0,0) circle (.375ex);}};},
slashed/.code={\tikzcdset{markwith={\draw[-] (-.4ex,-.4ex) -- (.4ex,.4ex);}};},
markwith/.code={
\pgfutil@ifundefined{tikz@library@decorations.markings@loaded}%
{\pgfutil@packageerror{tikz-cd}{You need to say %
\string\usetikzlibrary{decorations.markings} to use arrow with markings}{}}{}%
\pgfkeysalso{/tikz/postaction={/tikz/decorate,
/tikz/decoration={
markings,
mark = at position 0.5 with
{#1}}}}},
}
\makeatother
\makeatletter
\tikzset{
circled/.code={\tikzset{markwith={\draw (0,0) circle (.375ex);}};},
slashed/.code={\tikzset{markwith={\draw[-] (-.4ex,-.4ex) -- (.4ex,.4ex);}};},
markwith/.code={
\pgfutil@ifundefined{tikz@library@decorations.markings@loaded}%
{\pgfutil@packageerror{tikz-cd}{You need to say %
\string\usetikzlibrary{decorations.markings} to use arrow with markings}{}}{}%
\pgfkeysalso{/tikz/postaction={/tikz/decorate,
/tikz/decoration={
markings,
mark = at position 0.5 with
{#1}}}}},
}
\makeatother

\makeatletter
\providecommand{\leftsquigarrow}{%
  \mathrel{\mathpalette\reflect@squig\relax}%
}
\newcommand{\reflect@squig}[2]{%
  \reflectbox{$\m@th#1\rightsquigarrow$}%
}
\makeatother

\usepackage[english]{babel}
\usepackage{stmaryrd}
\usepackage{mathbbol}
\usepackage{float}
\usepackage{comment}
\usepackage[colorlinks=true, pdfstartview=FitV, linkcolor=blue, citecolor=blue, urlcolor=blue, breaklinks=true]{hyperref}
\usepackage{theoremref}
\usepackage{caption}
\usepackage{youngtab}
\usepackage{ragged2e}
\usepackage{relsize}
\usepackage{anyfontsize}
\usepackage[utf8]{inputenc}
\usepackage{bm}
\usepackage{scalerel}
\usepackage{marvosym}
\DeclareFontFamily{U}{mathx}{\hyphenchar\font45}
\DeclareFontShape{U}{mathx}{m}{n}{
      <5> <6> <7> <8> <9> <10>
      <10.95> <12> <14.4> <17.28> <20.74> <24.88>
      mathx10
      }{}
\DeclareSymbolFont{mathx}{U}{mathx}{m}{n}
\DeclareFontSubstitution{U}{mathx}{m}{n}
\DeclareMathAccent{\widecheck}{0}{mathx}{"71}
\DeclareMathAlphabet{\mathbbm}{U}{bbm}{m}{n}
\usepackage{verbatim}
\usepackage{pdflscape}

\usepackage{array} 
\newcolumntype{C}{>{$}c<{$}}

\makeatother
\theoremstyle{remark}
\numberwithin{equation}{section}

\theoremstyle{definition}
\newtheorem*{theorem*}{Theorem}
\newtheorem*{definition*}{Definition}
\newtheorem{theorem}{Theorem}[section]
\newtheorem{definition}[theorem]{Definition}
\newtheorem{proposition}[theorem]{Proposition}
\newtheorem*{proposition*}{Proposition}
\newtheorem{lemma}[theorem]{Lemma}
\newtheorem{corollary}[theorem]{Corollary}
\newtheorem{remark}[theorem]{Remark}
\newtheorem{notation}[theorem]{Notation}
\newtheorem{example}[theorem]{Example}
\newtheorem*{example*}{Example}

\newtheorem{condition}{Condition}

\DeclareMathOperator{\Sing}{Sing}
\newcommand{\scat}{\mathfrak{D}}
\newcommand{\mono}{m}
\DeclareMathOperator{\Supp}{Supp}

\newcommand{\lrp}[1]{\left(#1\right)}
\newcommand{\lrb}[1]{\left[#1\right]}

\newcommand{\lrc}[1]{\left\{#1\right\}}
\newcommand{\lra}[1]{\left\langle{#1}\right\rangle}

\newcommand{\Q}{\mathbb{Q} }
\newcommand{\R}{\mathbb{R} }
\newcommand{\C}{\mathbb{C} }

\newcommand{\Z}{\mathbb{Z} }

\newcommand{\wall}{\mathfrak{d}}

\newcommand{\Lp}{p^\vee}
\newcommand{\LA}{{}^{L}\cA}

\newcommand{\LX}{{}^{L}\cX}

\newcommand{\LGam}{{}^{L}\Gamma}

\newcommand{\cXe}{\cX_{\bf 1}}
\newcommand{\cXeH}{\cX_{{\bf 1}_{T_{H^*_{\cX}}}}}

\newcommand{\cAm}{\cA^{\vee}}
\newcommand{\cXm}{\cX^{\vee}}

\newcommand{\cAHAm}{(\cA/T_{H_{\cA}})^{\vee}}
\newcommand{\cXeHm}{\lrp{\cXeH}^{\vee}}

\newcommand{\cXem}{\lrp{\cX_{\bf 1}}^{\vee}}

\newcommand{\Xnet}{\cX}

\makeatletter
\newcommand{\oversetcustom}[3][0ex]{%
  \mathrel{\mathop{#3}\limits^{
    \vbox to#1{\kern-0.5\ex@
    \hbox{$\scriptstyle#2$}\vss}}}}
\makeatother

\newcommand{\T}{\mathbb{T}}
\newcommand{\orT}{\oversetcustom{\longrightarrow}{\mathbb{T}_r}}

\newcommand{\gv}{\mathbf{g} }
\newcommand{\cv}{\mathbf{c} }
\newcommand{\cval}{\nu^{\Phi}_{\seed}}
\newcommand{\vb}[1]{\mathbf{#1}}
\newcommand{\cA}{\mathcal{A} }
\newcommand{\cAp}{\mathcal{A}_{\mathrm{prin}}}
\newcommand{\cXp}{\mathcal{X}_{\mathrm{prin}}}

\newcommand{\prin}{{\mathrm{prin}} }
\newcommand{\cX}{\mathcal{X} }
\newcommand{\cV}{\mathcal{V} }

\newcommand{\lb}{\mathcal{L} }
\newcommand{\trop}{\mathrm{trop} }
\newcommand{\Trop}{\mathrm{Trop} }
\newcommand{\tf}{\vartheta }

\newcommand{\eq}[2]{\begin{equation}\label{#2} \begin{split} #1  \end{split} \end{equation}}
\newcommand{\eqn}[1]{\begin{equation*} \begin{split} #1 \end{split} \end{equation*}}

\newcommand{\Nuf}{N_{\text{uf}}}
\newcommand{\Iuf}{I_{\text{uf}}}

\newcommand{\seed}{\textbf{s}}

\DeclareMathOperator{\val}{val}

\newsavebox{\mybox}
\newcommand{\back}[1]{%
  \ThisStyle{\ifmmode%
    \savebox{\mybox}{$\SavedStyle#1$}%
    \reflectbox{\usebox{\mybox}}%
  \else%
    \savebox{\mybox}{#1}%
    \reflectbox{\usebox{\mybox}}%
  \fi%
}}

\DeclareMathOperator{\supp}{Supp}

\DeclareMathOperator{\Hom}{Hom}

\DeclareMathOperator{\Spec}{Spec}

\DeclareMathOperator{\Grass}{Gr}

\DeclareMathOperator{\ord}{ord}
\DeclareMathOperator{\cmid}{mid}
\DeclareMathOperator{\up}{up}

\DeclareMathOperator{\Cox}{Cox}
\DeclareMathOperator{\Pic}{Pic}
\newcommand{\UT}{\text{UT} }

\newcommand{\Yng}{\Yboxdim4pt \yng}
\newcommand{\Yngs}{\Yboxdim2pt \yng}

\DeclareMathOperator{\op}{op}

\newcommand{\uf}{\operatorname{uf}}

\DeclareMathOperator{\conv}{conv}
\DeclareMathOperator{\bconv}{conv_{BL}}
\DeclareMathOperator{\NewtT}{Newt_{\tf}}

\newcommand{\ie}{{\em i.e. }}
\newcommand{\cf}{{\em cf.}\ }

\DeclareFontFamily{U}{musix}{}
\DeclareFontShape{U}{musix}{m}{n}{<-> s*[1.01] musix11}{}

\newcommand{\Np}{N_{\prin}}

\newcommand{\Mpc}{M^\circ_{\prin}}

\title{Newton--Okounkov bodies and minimal models for cluster varieties}

\author{Lara Bossinger, Man-Wai Cheung, Timothy Magee and Alfredo N\'ajera Ch\'avez}
\date{\today}

\address{
Instituto de Matem\'aticas Unidad Oaxaca, 
Universidad Nacional Aut\'onoma de M\'exico,
Le\'on 2, altos, 
Centro Hist\'orico,
68000 Oaxaca,
Mexico}
\email{lara@im.unam.mx}

\address{
School of Mathematics, Kavli IPMU (WPI), UTIAS, The University of Tokyo, Kashiwa, Japan, 277-8583}
\email{manwai.cheung@ipmu.jp}

\address{Department of Mathematics, King's College London, Strand, London WC2R 2LS, UK}
\email{timothy.magee@kcl.ac.uk}

\address{
Consejo Nacional de Ciencia y Tecnología - Instituto de Matem\'aticas Unidad Oaxaca, 
Universidad Nacional Aut\'onoma de M\'exico,
Le\'on 2, altos, 
Centro Hist\'orico,
68000 Oaxaca,
Mexico}
\email{najera@im.unam.mx}

\begin{document}

\maketitle

\begin{abstract}
Let $Y$ be a (partial) minimal model of a scheme $V$ with a cluster structure (of type $\cA$, $\cX$ or of a quotient of $\cA$ or a fibre of $\cX$). Under natural assumptions, for every choice of seed  we associate a Newton--Okounkov body to every divisor on $Y$ supported on $Y \setminus V$ and show that these Newton--Okounkov bodies are positive sets in the sense of Gross, Hacking, Keel and Kontsevich \cite{GHKK}. This construction essentially reverses the procedure in loc. cit. that generalizes the polytope construction of a toric variety to the framework of cluster varieties.

In a closely related setting, we consider cases where $Y$ is a projective variety whose universal torsor $\UT_Y$ is a partial minimal model of a scheme with a cluster structure of type $\cA$. If the theta functions parametrized by the integral points of the associated superpotential cone form a basis of the ring of algebraic functions on $\UT_Y$ and the action of the torus $T_{\text{Pic}(Y)^*}$ on $\UT_Y$ is compatible with the cluster structure, then for every choice of seed we associate a Newton--Okounkov body to every line bundle on $Y$. We prove that any such Newton--Okounkov body is a positive set and that $Y$ is a minimal model of a quotient of a cluster $\cA$-variety by the action of a torus. 

Our constructions lead to the notion of the intrinsic Newton--Okounkov body associated to a boundary divisor in a partial minimal model of a scheme with a cluster structure. This notion is intrinsic as it relies only on the geometric input, making no reference to the auxiliary data of a valuation or a choice of seed.
The intrinsic Newton--Okounkov body lives in a real tropical space rather than a real vector space. 
A choice of seed gives an identification of this tropical space with a vector space, and in turn of the intrinsic Newton--Okounkov body
with a usual Newton--Okounkov body associated to the choice of seed.
In particular, the Newton--Okounkov bodies associated to seeds are related to each other by tropicalized cluster transformations providing a wide class of examples of Newton-Okoukov bodies exhibiting a wall-crossing phenomenon in the sense of Escobar--Harada \cite{EH20}.

This approach includes the partial flag varieties that arise as minimal models of cluster varieties (for example full flag varieties and Grassmannians). For the case of Grassmannians, our approach recovers, up to interesting unimodular equivalences, the Newton--Okounkov bodies constructed by Rietsch--Williams in \cite{RW}.

\end{abstract}

\tableofcontents

\section{Introduction}

\subsection{Overview} Cluster varieties are certain schemes constructed by gluing a (possibly infinite) collection of algebraic tori using distinguished birational maps called cluster transformations.
These schemes were introduced in \cite{FG_Teich,FG_cluster_ensembles} and can be studied from many different points of view. 
They are closely related to cluster algebras and $Y$-patterns defined by Fomin and Zelevinsky in \cite{FZ_clustersI, FZ_clustersIV}. 
In this paper we approach them from the perspectives of birational and toric geometry, mainly following \cite{GHK_birational,GHKK}.
In \cite{GHKK}, the authors show that certain sets called \emph{positive polytopes} can be used to produce compactifications of cluster varieties and toric degenerations of such compactifications. 
In the trivial case where the cluster variety in question is just a torus, a positive lattice polytope is simply a usual convex lattice polytope and this construction produces the toric variety associated to such a polytope.
One of the main goals of this paper is to reverse this construction in a systematic way and understand this process from the view-point of Newton--Okounkov bodies. 
We also study the wall-crossing phenomenon for Newton--Okounkov bodies arising from cluster structures.
We treat independently the case of the Grassmannians as, in this context, we compare the Newton--Okounkov bodies we construct with those constructed in \cite{RW} and explore some consequences.
Moreover, throughout the text we systematically consider not only cluster varieties but also quotients and fibres associated to them (see \S\ref{sec:quotients-fibres} for the precise definitions of these quotients and fibres). For simplicity, in this introduction our main focus is on cluster varieties. 
We fix once and for all an algebraically closed field $\Bbbk$ of characteristic zero. 
Unless otherwise stated, all the schemes we consider are over $\Bbbk$.

\subsection{The tropical spaces} Let $\cV $ be a cluster variety. By definition, $\cV$ is endowed with an atlas of algebraic tori of the form 
\[
\cV= \bigcup_{\seed} T_{L;\seed},
\]
where $L$ is a fixed lattice, $T_{L; \seed} $ is a copy of the algebraic torus $T_L= \Spec(\Bbbk[L^*])$ associated to $L$ (so $L^*=\text{Hom}(L, \Z)$) and the tori in the atlas are parametrized by \emph{seeds $\seed$ for} $\cV$. 
We will exploit the fact that $\cV$ is a log-Calabi--Yau variety. 
This property implies that $\cV$ is endowed with a canonical up-to-scaling volume form $\Omega$. 
Moreover, recall that a cluster variety is of one of the types: $\cA$ or $\cX$.

Just like in toric geometry where one can consider the dual torus $T_L^{\vee}:=T_{L^*}$, the \emph{dual} of $\cV$ is a cluster variety $\cV^{\vee}$ whose defining atlas consists of tori of the form $T^\vee_L$.
It is well known that the ring $H^{0} (T_L,\mathcal{O}_{T_{L}})$ of algebraic functions on $T_L$ has a distinguished basis --the set of characters of $T_L$-- parametrized by $L^*$. 
For nearly 10 years it was conjectured that this fact can be generalized for $\cV$ using this notion of duality.
In order to state such a generalization, we consider the integral tropicalization of $\cV^{\vee}$, which we denote by $\Trop_{\Z}(\cV^{\vee})$.
The precise definition of $\Trop_{\Z}(\cV^{\vee})$
 can be found in \S\ref{ss:tropicalization}.
 For this introduction the key fact that we need is that a prime divisor $D$ on a variety birational to $\cV^\vee $ determines a point of $\Trop_{\Z}(\cV^{\vee})$ if $\Omega $ has a pole along $D$.
In \cite{FG_cluster_ensembles} Fock--Goncharov conjectured that $H^{0}(\cV, \mathcal{O}_\cV)$ has a canonical vector space basis parametrized by $\Trop_{\Z}(\cV^{\vee})$.
Although false in general, this conjecture does hold in many of the cases of wide interest.
In \cite{GHKK} the authors linked this conjecture to the log Calabi--Yau mirror symmetry conjecture \cite[Cnjecture 0.6]{GHK_logCY}, suggesting that the canonical basis proposed by Fock--Goncharov is the \emph{theta basis}.
As we would like to be as close to toric geometry as possible we systematically assume that the full Fock--Goncharov conjecture holds for the cluster variety $ \cV$ under consideration.
So, under under the assumption that the full Fock--Goncharov conjecture holds for $\cV$, one may consider $\Trop_{\Z}(\cV^{\vee})$ as replacing $L^*$ and the characters of $T_L$ are replaced by the theta functions on $\cV$.
Moreover, the real vector space $L^*\otimes \R$ is replaced by the real tropicalization $ \Trop_{\R}(\cV^{\vee})$ and convex polyhedra inside $L^*\otimes \R$
are replaced by positive sets in the real tropical space $\Trop_{\R}(\cV^{\vee})\supset \Trop_{\Z}(\cV^{\vee})$ (see \S\ref{ss:tropicalization} and Definition \ref{def:positive_set} for the definitions of $\Trop_{\R}(\cV^{\vee})$ and of positive set, respectively).

Besides the trivial case where $\cV$ is just a torus (and hence $\cV^{\vee}$ is just the dual torus), the tropical spaces $\Trop_{\Z}(\cV^{\vee})$ and $\Trop_{\R}(\cV^{\vee})$ do not possess a linear structure (there is no natural notion of addition in these spaces and only multiplication by positive scalars makes sense).
However, in certain situations these tropical spaces do contain subsets where addition and scalar multiplication make sense, which we call \emph{linear subsets}. 
In any case, every choice of seed $\seed^\vee$ for $\cV^{\vee}$ gives rise to a bijection $\mathfrak{r}_{\seed^\vee}:\Trop_{\R}(\cV^{\vee}) \longrightarrow \R^d$ that restricts to a bijection $\Trop_{\Z}(\cV^{\vee}) \overset{\sim}{\longrightarrow} \Z^d$, where $d$ is the dimension of both $\cV$ and $\cV^\vee$. 
In general, different seeds lead to different bijections. 
When we fix one such identification $\mathfrak{r}_{\seed^\vee}$ and talk about linear subsets of $\Z^d$ and positive subsets of $\R^d$, what we mean is that the inverse image of such a set under $\mathfrak{r}_{\seed^\vee}$ has the given property.

\subsection{Positive Newton--Okounkov bodies and minimal models} Newton--Okounkov bodies are convex closed sets in real vector spaces. Their systematic study was developed by Lazarsfeld--Musta\c{t}\u{a} \cite{LM09} and Kaveh--Khovanskii \cite{KK12} based on the work of Okounkov \cite{Oko96,Oko03}. 
This concept is a far reaching generalization of both the Newton polytope of a Laurent polynomial and the polytope of a polarized projective toric variety. 
In \cite{KK12} the authors introduced Newton--Okounkov bodies for Cartier divisors on irreducible varieties.
In this paper we consider Newton--Okounkov bodies associated to Weil divisors in the setting of minimal models for cluster varieties.
More precisely, let $D$ be a Weil divisor on a $d$-dimensional normal variety $Y$ admitting a non-zero global section, that is, the space $H^0(Y, \mathcal{O}(D))$ is non-zero, where $\mathcal{O}(D)$ is the coherent sheaf associated to $D$. 
The section ring of $D$ is a graded ring 
\[
R(D)=\bigoplus_{k\in \Z_{\geq 0}}{R}_k(D) 
\] 
whose $k$-th homogeneous component is the vector space $R_k(D)=H^0(Y, \mathcal{O}(kD)) \subset \Bbbk (Y)$. 
Fix a non-zero element $\tau\in R_1(D)$, and suppose we are given a total order on $\Z^d$ and a valuation $\nu: \Bbbk(Y)^* \to \Z^d$. 
Then the Newton--Okounkov body associated to this data is:
\eqn{
\Delta_\nu(D,\tau) := \overline{\conv\Bigg( \bigcup_{k\geq 1}  \lrc{\frac{\nu\lrp{f/\tau^k}}{k} \mid f\in R_k(D)\setminus \{0\} } \Bigg) }\subseteq \R^d.
}

Given a cluster variety $\cV$, our first goal is to use its cluster structure to construct Newton--Okounkov bodies associated to divisors in compactifications of $\cV$, generalizing the construction of the polytope of a torus invariant divisors on a toric variety.
Hence, we need to establish the class of compactifications of $\cV $, the divisors therein and the valuations we consider.

We begin discussing valuations obtained from the cluster structure. 
In case $\cV$ is a cluster $\cA$-variety, this is closely related to the work of Fujita and Oya \cite{FO20}. 
However, our approach includes the cases where $\cV$ is a cluster $\cX$-variety, a quotient of a cluster $\cA$-variety, or a fibre of a cluster $\cX$-variety.
In order to be able to use the cluster structure of $\cV$ to construct a valuation on $ \Bbbk(\cV)$ certain conditions (depending on whether $\cV$ is of type $\cA$ or of type $\cX$) need to be fulfilled. 
For instances, if $\cV$ is of type $\cA$, a sufficient condition is that the rectangular matrix $\widetilde{B}$ determining the cluster structure of $\cV$ has full rank\footnote{A weaker condition is enough to construct a valuation for a given seed, see Remark \ref{rem:dom_order}.}; if $\cV$ is of type $\cX$ we need that the full Fock--Goncharov conjecture holds for $\cX$ (as we are assuming), see \S\ref{sec:cluster_valuations} for more details, including the cases of quotients of $\cA$ and fibres of $\cX$.
In case the necessary conditions are satisfied then for every $\seed$ for $\cV$ we have a cluster valuation
\[
\nu_\seed: \Bbbk(\cV) \setminus\{0\} \to (\Z^d, <_{\seed}).
\]
The total order $<_{\seed}$ on  $\Z^d$ depends also on the type of $\cV$.
Moreover, in case $\cV$ is of type $\cA$ in the literature this valuation is generally denoted by $\gv_{\seed} $ and called a $\gv$-\emph{vector valuation} as it is closely related to the $\gv$-vectors associated to cluster monomials introduced in \cite{FZ_clustersIV}.
In case $\cV$ is of type $\cX$ the associated cluster valuation has not been systematically defined yet in the literature to the best of our knowledge. 
In this case we also denote $\nu_{\seed}$ by $\cv_{\seed}$ and call it a $\cv$-\emph{vector valuation} since this valuation is closely related to the $\cv$-vectors associated to $Y$-variables introduced in \cite{NZ} and more generally to {\bf c}-vectors of theta functions on $\cX$ defined in \cite{BFMNC}, and currently investigated in \cite{ML23}. 
In any case, for every seed $\seed$ the theta basis of $H^0(\cV, \mathcal{O}_{\cV})$ is adapted for the cluster valuation $\nu_{\seed}$.
In particular, if $ Y$ is a variety birational to $\cV $ and $D$ is a divisor in $Y$, then, upon a choice of non-zero section $\tau \in R_1(D)$ and a seed $\seed$, we can construct a Newton--Okounkov body $\Delta_{\nu_\seed}(D,\tau) $. 
We are primarily interested in conditions ensuring that such a Newton--Okounkov body is a positive set.
On the one hand this is a condition that needs to be satisfied if one seeks to reverse Gross--Hacking--Keel--Kontsevich's construction of a compactification of a cluster variety from a positive set.
On the other hand, we are further interested in describing how the change of seed affects the Newton--Okounkov body and positivity plays the key role in understanding this. 
 If $\Delta_{\nu_\seed}(D,\tau)$ is positive then any other $\Delta_{\nu_{\seed'}}(D,\tau)$ is obtained from $\Delta_{\nu_\seed}(D,\tau)$ by a composition of tropicalized cluster transformations. 
 This will be discussed in more detail in the next subsection of the introduction.
In order to be able to show that $\Delta_{\nu_\seed}(D,\tau) $ is positive we restrict the class of compactifications of $\cV$, the divisors we consider, and the sections we choose.

One can define a partial minimal model for $\cV$\footnote{Throughout the text we consider more generally (partial) minimal models for schemes $V$ with a cluster structure given by a birational map $\cV\dashrightarrow V$.} is an inclusion $\cV \subset Y$ such that $Y$ is normal and $\Omega$ has a simple pole along every irreducible divisorial component of the boundary $D=Y \setminus \cV$, see \cite[Remark~1.3]{GHK_birational}. It is a minimal model if $Y$ is projective over $\Bbbk$.
These are the kind of (partial) compactifications of $\cV $ we consider.
The main reason for this is that any prime divisor supported on $D$ determines a primitive point of $\Trop_{\Z}(\cV)$.  
Let $D'$ be a divisor supported on $D$. 
We say that $R(D')$ has a {\it graded theta basis} if for each $k$ 
the set of theta functions on $\cV$ contained in $H^0(Y,\mathcal{O}(kD'))$ forms a basis (see Definition~\ref{def:graded_theta_basis}).
Then we can prove the following result.

\begin{theorem*}
(Theorem \ref{NO_bodies_are_positive})
Let $D'$ be a Weil divisor supported on the boundary $D$ of the minimal model $\cV \subset Y $ such that $R(D')$ has a graded theta basis. Let $\tau\in R_1(D')$ be such that $\nu_{\seed}(\tau) $ belongs to a linear subset of $ \Z^d$. Then the Newton--Okounkov body  $\Delta_{\nu_{\seed}}(D',\tau)$ is a positive polytope.
\end{theorem*}

In Lemma~\ref{lem:graded_theta_basis} we provide sufficient conditions ensuring that $R(D)$ has a graded theta basis. 
Moreover, the work of Mandel \cite{Man19} provides conditions ensuring that a line bundle on a cluster $\cX$-variety has a graded theta basis.

We further study another setting where we can use cluster structures to construct Newton--Okounkov bodies and show that they are positive polytopes:
suppose that $Y$ is a normal projective variety such that its Picard group is free and finitely generated.
The universal torsor of $Y$ is a scheme $\UT_Y$ whose ring of algebraic  functions is isomorphic to the direct sum of all the spaces of sections associated to all (isomorphism classes of) line bundles over $Y$.
We assume that $\UT_Y$ is a partial minimal model of a cluster $\cA$-variety, which we denote by $\cA \subset \UT_Y$.
For example, we encounter this situation frequently in the study of homogeneous spaces, where moreover the ring of global functions on $\UT_Y$ has a representation theoretic interpretation due to the Borel--Weil--Bott Theorem (Remark~\ref{rmk:borel weil bott}).
This fact is commonly used when constructing Newton--Okounkov bodies in Lie theory, see e.g. \cite{FFL15} and the references therein.

Let $D_1, \dots, D_s$ be the irreducible divisorial components of $D= \UT_Y \setminus \cV$ and let $\tf^{\cA^{\vee}}_{i}$ be the theta function on $\cA^{\vee}$ parametrized by the point in $\Trop_{\Z}(\cA)$ associated to $D_i$.
The (theta) superpotential\footnote{If $Y$ is Fano, then this should be considered as the superpotential used for mirror symmetry purposes.} associated to the inclusion $\cA \subset \UT_Y$ is
\[
W_{\UT_Y} = \sum_{i=1}^s \tf^{\cA^{\vee}}_i.
\]
The associated superpotential cone is the subset $\Xi_{\UT_Y}$ of $\Trop_{\R}(\cV^{\vee})$ where the tropicalized superpotential takes non-negative values.
Given a choice of seed $ \seed^\vee$ for $\cA^\vee$, $\Xi_{\UT_Y}$ is identified with a polyhedral cone $\Xi_{\UT_Y, \seed^\vee}\subset \R^d$.
As discussed in \cite{GHKK}, in many cases the integral points of $\Xi_{\UT_Y}$ parametrize the set of theta functions on $\cA$ that extend to $\UT_Y$. 
This happens for example if $\cA$ has \emph{theta reciprocity} (see Definition \ref{def:theta_reciprocity}), a condition that is conjectured to be true in situations more general than ours.
Even stronger, in many of the examples arising in nature the integral points of $\Xi_{\UT_Y}$ parametrize a basis of $H^0(\UT, \mathcal{O}_{\UT_Y})$.
In \cite{GHKK}, Gross--Hacking--Keel--Kontsevich give criteria ensuring that this is satisfied.
These conditions hold true in many cases of interest in representation theory, as was proven in several papers including \cite{BF, FO20,  GKS_polyhedral,GKS_typeA, GKS_string, Mag20, SW18} and \cite[\S9]{GHKK}.
Moreover, for special choices of seeds $\seed^{\vee} $, in these cases the cone $\Xi_{\UT_Y, \seed^\vee}$ agrees with known polyhedral cones such as the Gelfand--Tsetlin cone, string cones or the Knudson--Tao hive cone.
Much of the inspiration of this paper is due to the representation theoretic results that precede it.
In the case where the integral points of $\Xi_{\UT_Y}$ parametrize the set of theta functions on $\cA$ that extend to $\UT_Y$, we can restrict a {\bf g}-vector valuation $\gv_\seed $ from $\Bbbk(\cA)$ to $H^0(\UT_Y, \mathcal{O}_{\UT_Y})$. Therefore, given a line bundle $\lb$ on $Y$  we can construct a Newton--Okounkov body $\Delta_{\gv_\seed}(\lb)$ in a similar way as before.
In order to show that $\Delta_{\gv_{\seed}}(\lb)$ is a positve polytope we need to consider torus actions on $ \cA$ and fibrations of $\cA^{\vee}$ over a torus as we now explain. 

The universal torsor $\UT_Y$ is endowed with the action of the torus $T_{\text{Pic}(Y)^*}$ associated to the dual of the Picard group of $Y$. 
We first need this torus action to preserve $\cA$ 
and that the induced action on $\cA$ is cluster in the sense of \thref{k_and_pic} (roughly speaking this means that the restricted action can be identified with the action induced by the choice of a sublattice of the kernel of $\widetilde{B}$).
In such situations we have a cluster fibration 
\[
w:\cA^{\vee}\to T_{\text{Pic}(Y)}.
\]
Recall that the choice of seed gives rise to the identification  $\mathfrak{r}_{\seed^\vee}:\Trop_{\R}(\cA^{\vee}) \to \R^d$. 
The tropicalization of $w$ expressed using such an identification is a linear map $w^T:\R^d \to \text{Pic}(Y)\otimes \R$. 
Under the conditions above the Newton--Okounkov body $\Delta_{\gv_{\seed}}(\lb)$ can be described as a slicing of the superpotential cone. More precisely, we have the following result (see Definition \ref{def:quotient_fibre}).
\begin{theorem*}
(\thref{thm:k_and_pic})
Assume that the theta functions on $\cV$  parametrized  by the integral points of $\Xi_{\UT_Y}$ form a basis of $H^0(\UT_Y, \mathcal{O}_{\UT_Y})$. 
If the action of $ T_{\text{Pic}(Y)^*}$ restricts to a cluster action of $ T_{\text{Pic}(Y)^*}$ on $ \cA$ then for any class $[\lb]\in \text{Pic}(Y) $ the Newton--Okounkov body $\Delta_{{\bf g}_{\seed}}(\lb)$ can be describe as
\[
\Delta_{{\bf g}_{\seed}}(\lb)=\Trop_{\R}(w)^{-1}([ \lb ])\cap \Xi_{\UT_Y, \seed}.
\]
In particular, $\Delta_{{\bf g}_{\seed}}(\lb)$ is a positive subset of $\Trop_{\R}(\cV^{\vee})$ and $Y$ is a minimal model of the quotient of $\cA$ by the action of $T_{\text{Pic}(Y)^*}$.
\end{theorem*}

The case where $Y$ is the Grassmannian $\text{Gr}_{n-k}(\C^n)$ fits the framework above so it is possible to use the cluster $\cA$ structure to construct Newton--Okounkov bodies associated to arbitrary line bundles over $\text{Gr}_{n-k}(\C^n)$. 
We show that the Newton--Okounkov bodies we construct are unimodular to the Newton--Okounkov bodies constructed for $\text{Gr}_{n-k}(\C^n)$ by Rietsch and Williams in \cite{RW} using the cluster $\cX$ structure on Grassmannians (see Theorem~\ref{thm: val and gv}).
Moreover, the flow valuations of \cite{RW} are instances of $\bf c$-vector valuations.

This comparison result already has interesting consequences related to toric degenerations:
\begin{enumerate}
    \item Given a rational polytopal Newton--Okounkov body $\Delta$ for a (very ample) line bundle $\lb$ over $Y$ Anderson's main result in \cite{An13} applies and it yields a toric degeneration of $Y$ to a toric variety (whose normalization is) defined by $\Delta$. As the semigroup algebras of the {\bf g}-vector valuations are saturated, no normalization is necessary.
    \item The construction of Gross--Hacking--Keel--Kontsevich in \cite[\S8]{GHKK} associates to a positive polytope $P$ a minimal model $\cV\subset Y$ and moreover, using Fomin--Zelevinsky's principal coefficients, a toric degneration of $Y$ to the toric variety defined by $P$.
    As our Newton--Okounkov bodies are positive polytopes, this construction applies in our setting.
\end{enumerate}
The identification of the Newton--Okounkov bodies constructed by Rietsch--Williams and our Newton--Okounkov bodies constructed from {\bf g}-vectors implies the following result.

\begin{theorem*}(Theorem~\ref{thm: val and gv} and Remark~\ref{rmk:toric degen})
The toric degenerations of $\text{Gr}_{n-k}(\C^n)$ determined by the Newton--Okounkov polytopes constructed by Rietsch--Williams using Anderson's result coincide with the toric degenerations of $\text{Gr}_{n-k}(\C^n)$ given by Gross--Hacking--Keel--Kontsevich construction using principal coefficients.
\end{theorem*}

\subsection{The intrinsic Newton--Okounkov body} 
Understanding how Newton--Okounkov bodies change upon changing the valuation is an interesting problem that has attracted the attention of several authors, see for example \cite{EH20, BMNC, FH21,CHM22,HN23}.
So let us return to the discussion on how the Newton--Okounkov bodies constructed above transform if we change the choice of seed.
Given any two seeds $\seed $ and $\seed'$ for $ \cV^\vee$ there is a piecewise linear bijection $\Trop_{\R}(\mu^{\cV^\vee}_{\seed,\seed'}):\R^d \to \R^d$ relating the identifications of $\Trop_{\R}(\cV^{\vee})$ with $\R^d$.
More precisely, we have a commutative diagram
\[
\xymatrix{
&\Trop_{\R}(\cV^{\vee})
\ar_{\mathfrak{r}_{\seed}}[dl] \ar^{\mathfrak{r}_{\seed'}}[dr] & \\
\R^d \ar^{\Trop_{\R}(\mu^{\cV^\vee}_{\seed,\seed'})}[rr]&   & \R^d.
}
\]
Every map $\Trop_{\R}(\mu^{\cV^\vee}_{\seed,\seed'})$ restricts to a piecewise linear bijection of $\Z^d$ and, 
by construction, the maps $\Trop_{\R}(\mu^{\cV^\vee}_{\seed,\seed'})$ are composition of tropicalized cluster transformations for $\cV^\vee$ (see \S\ref{sec:intrinsic_NOB} for a more concise description).
For a subset $P\subseteq \Trop_{\R}(\cV^{\vee})$ we let $P_{\seed}=\mathfrak{r}_{\seed}(P)$.
One of the main properties behind our interest in showing that the Newton--Okounkov bodies we have constructed are positive sets is the following:
if $P\subseteq \Trop_{\R}(\cV^{\vee})$ is a positive set then $\Trop_{\R}(\mu^{\cV^\vee}_{\seed,\seed'})(P_{\seed})=P_{\seed'}$ for any two seeds, $\seed$ and $\seed'$.
In particular, in this situation the entire collection of sets $\{P_\seed\}_{\seed}$ parametrized by the seed for $\cV^{\vee}$ may be replaced by $P$, a single intrinsic object that can be used to recover any $P_\seed$ in the family.

In the case where a Newton--Okounkov body $\Delta_{\nu_{\seed}}$ (associated to a line bundle $\lb$ or a pair $(D',\tau)$ as in the previous subsection) is positive, any other Newton--Okounkov body $\Delta_{\nu_{\seed'}}$ associated to the same data is also positive. 
In this situation there is a single intrinsic object $\Delta_{\mathrm{BL}} \subset\Trop_{\R}(\cV^{\vee})$ representing the entire collection $\{ \Delta_{\nu_{\seed}}\}_{\seed}$. 
We call $\Delta_{\mathrm{BL}}$ the \emph{intrinsic Newton--Okounkov body} (associate to the data we begin with).
The subindex $\mathrm{BL}$ in $\Delta_{\mathrm{BL}}$ stands for \emph{broken line}, the choice of this notation goes back to \cite{CMNcpt} where the last three authors of this paper introduce \emph{broken line convexity}-- a notion of convexity defined in a tropical space that ensures positivity. 
Broken lines are pieces of tropical curves in $\Trop_{\R}(\cV^{\vee})$ used to define theta functions on $\cV$ and describe their multiplication (see \S\ref{sec:tf_and_parametrizations}).
Straight line segments defining convexity in a linear space are replaced by broken line segments in the tropical space to define broken line convexity.
The main result of \cite{CMNcpt} is that a closed set is broken line convex if and  only if it is positive.

In the situations where we are able to show that $\Delta_{\nu_{\seed}}\subset \R^d$ is positive, it turns out that it is moreover polyhedral, a property that fails in general, see e.g. \cite{LM09,KLM_NObodies_spherical}. 
Since $\Delta_{\nu_{\seed'}}=\Trop_{\R}(\mu^{\cV^\vee}_{\seed,\seed'})(\Delta_{\nu_{\seed}})$ any other $\Delta_{\nu_{\seed'}}$ is also polyhedral.
The integral points of the convex bodies we consider are naturally associated to theta functions,
which suggests is the following  question: does there exist a finite set of theta functions such that $\Delta_{\nu_{\seed'}}$ is the convex hull of their images under $ \nu_{\seed'}$ for any seed $\seed'$?
Such a collection of points might vary as we change seeds as exhibited in the case of the Grassmannians in an example in \cite[\S9]{RW} and generalized to an infinite family of examples in \cite[Theorem~3]{bossinger2019full}.
Given the notion of broken line convexity, a slight reformulation of the question becomes more natural: 
does there exist a finite set of theta functions such that the broken line convex hull of their images under $\nu_{\seed'}$ is $\Delta_{\nu_{\seed'}}$ for some (and hence any) seed $\seed'$?
In fact, from the intrinsic Newton--Okounkov body perspective, the valuation is replaced by integral tropical points parametrizing theta functions and there is no reference to a seed at all.
Using this perspective, $\Delta_{\mathrm{BL}}$ becomes a broken line convex subset of $\Trop_{\R}(\cV^{\vee})$ whose integral points parametrize the theta basis of the first graded piece $R_1$ of the corresponding graded ring.
In \thref{taut} we give sufficient conditions ensuring that $\Delta_{\mathrm{BL}}$ can be described as the broken line convex hull of a finite collection of points and describe this collection.
Applying this result to the setting of Grassmannians we obtain that if $\lb_e$ is line bundle over $\Grass_{n-k}(\C^n)$ obtained by pullback of $\mathcal{O}(1)$ under the Plücker embedding $\Grass_{n-k}(\C^n)\hookrightarrow \mathbb P^{\binom{n}{k}-1}$ then the intrinsic Newton--Okounkov body $\Delta_{\mathrm{BL}}(\lb_e)$ is the broken line convex hull of the ${\bf g}$-vectors of the Pl\"ucker coordinates (Corollary~\ref{cor:intrinsicNO grassmannian}).

Broken line convexity also allows to generalize the Newton polytope of a Laurent polynomial to the the world of cluster varieties.
In particular, in \S\ref{sec:intrinsic_NOB} we introduce the \emph{theta function analog of the Newton polytope} of $f$, for any $f\in H^0(\cV, \mathcal{O}_\cV)$.
The intrinsic Newton--Okounkov bodies $\Delta_{\mathrm{BL}}$ can be described using this notion.
The key idea is exploiting the bijection between the theta basis (a special case of an \emph{adapted basis}) and integral tropical points parametrizing them.
This idea is explained for full rank valuations with finitely generated value semigroup in the survey \cite{B-toric}.
It is therefore interesting to continue studying this new class of objects.

\subsection{Organization of the paper}
In \S\ref{sec:background} we review background material on cluster varieties their quotients and their fibres (\S\ref{sec:back_ghkk}), and on tropicalization (\S\ref{ss:tropicalization}).
In \S\ref{sec:tf_and_parametrizations} we recall the construction of cluster scattering diagrams and the theta functions on (quotients and fibres of) cluster varieties.
In \S\ref{sec:minimal_models} we elaborate on the existence of a theta basis on the ring of regular functions on a partial minimal model of (a quotient or a fibre of) a cluster variety. This section largely follows \cite{GHKK}.
In \S\ref{sec:cluster_valuations} we recall the {\bf g}-vector valuations for (quotients) $\cA$-varieties. We introduce {\bf c}-vector valuations for (fibres of) $\cX$-varieties.
The main results of the paper are contained in \S\ref{sec:no}. The study of Newton--Okoukov bodies associated to  Weil divisors on minimal models is treated in \S\ref{sec:NO_bodies} while the Newton--Okoukov bodies for line bundles are treated in \S\ref{sec:universal_torsors}.
 The intrinsic Newton--Okounkov body and the wall-crossing phenomenon for these are addressed in \S\ref{sec:intrinsic_NOB}.
 Finally, in \S\ref{sec:NO_Grass} we apply the results of the previous section to Grassmannians. 
 One of the main technical conditions to be satisfied is verified in \S\ref{sec:Pic_property}.
 In \S \ref{sec:GHKK_and_RW} we prove a unimodular equivalence between the Newton--Okounkov bodies we construct and those constructed by Rietsch--Williams in \cite{RW}. 
 In \S\ref{sec:Grass_intrinsic} we describe the intrinsic Newton--Okounkov bodies for Grassmannians as the broken line convex hull
 of the {\bf g}-vectors of Pl\"ucker coordinates (in arbitrary seeds).

\subsubsection*{Acknowledgements} 
The authors L. Bossinger and A. Nájera Chávez were partially supported by PAPIIT project IA100122 dgapa UNAM 2022 and by CONACyT project CF-2023-G-106.
M. Cheung was supported by World Premier International Research Center Initiative (WPI Initiative), MEXT, Japan.
T. Magee was supported by EPSRC grant EP/V002546/1.

\section{Preliminaries}\label{sec:background}

\subsection{Cluster varieties, quotients and fibres}\label{sec:back_ghkk}
We briefly recall the construction of cluster varieties, their quotients and their fibres. The reader is invited to consult \cite{GHK_birational,GHKK} for the details we shall omit in this section. 

Unless otherwise stated, all tensor products are taken with respect to $\Z$. Moreover, given a lattice $L$ we denote by $L^*:= \Hom(L,\Z)$ its $\Z$-dual and let $ \langle \cdot , \cdot \rangle: L\times L^* \to \Z$ be the canonical pairing given by evaluation. 
We further denote by $L_\R:= L \otimes \R$ the real vector space associated to $L$. We fix an algebraically closed field $\Bbbk$ of characteristic $0$ and let $T_L:= \text{Spec}(\Bbbk [L^*])$ be the algebraic torus whose character lattice is $L^*$. 

\subsubsection{Cluster varieties and their dualities}
\label{sec:cluster_var}

The {\bf fixed data} $\Gamma$ consist of the following:
\begin{itemize}
    \item a finite set $I$ of {\bf directions} and a distinguished subset $\Iuf \subseteq I$ of {\bf mutable} (or {\bf unfrozen}) {\bf directions}. Elements of $I \setminus \Iuf $ are the {\bf frozen directions};
    \item a lattice $N$ of rank $|I|$ together with a saturated sublattice $N_{\text{uf}}\subseteq N$ of rank $|I_{\text{uf}}|$; 
    \item a skew-symmetric bilinear form $\{ \cdot , \cdot \} : N \times N \rightarrow \Q$;
    \item a finite index sublattice $N^\circ \subseteq N$ such that $\{ N, \Nuf \cap N^{\circ}\}\subset \Z$ and $\{ \Nuf, N^{\circ} \}\subset \Z$;
    \item a collection of positive integers $\{d_i\}_{i \in I}$ with greatest common divisor $1$;
    \item the dual lattices $M = \Hom (N, \Z)$ and $M^{\circ}=\Hom(N^{\circ},\Z)$.
\end{itemize}
A ${\bf seed}$ for $\Gamma$ is a tuple $\seed := ( e_i )_{i \in I}$ such that $\{ e_i \}_{i\in I}$ is a basis for $N$, $\{e_i\}_{i \in \Iuf}$ is a basis for $\Nuf$ and $\{d_i e_i \}_{i \in I } $ is a basis for $N^{\circ}$. 
We let $f_i := {d_i}^{-1} e_i^*$ and observe that $\{f_i\}_{i\in I}$ is a basis of $M^{\circ}$.
For $i,j\in I$ we write $\epsilon_{ij}:= \lbrace e_{i},d_j e_{j} \rbrace$ and define the matrix $\epsilon=(\epsilon_{ij})_{i,j\in I}$. 
When we work with various seeds at the same time we introduce labels of the form $e_{i;\seed}$, $f_{i;\seed}$, $\epsilon_{\seed}=(\epsilon_{ij;\seed})$, etc. to distinguish the data associated to $\seed $. We can {\bf mutate} a seed $\seed=(e_i)_{i\in I}$ in a mutable direction $k\in \Iuf$ to obtain a new seed $\mu_k(\seed)=(e'_i)_{i\in I}$ given by
\begin{equation}
\label{e_mutation}
e_i':=\begin{cases} e_i+[\epsilon_{ik}]_+e_k & i\neq k,\\
-e_k&i=k,
\end{cases}
\end{equation}
where $[x]_+:= \text{max}(0,x)$ for $x \in \R$. 

Let $r:=|\Iuf|$ and let $\mathbb{T}_r$ denote the $r$-regular tree whose edges are labeled by the elements of $\Iuf$.
We refer to $r$ as the {\bf rank} and fix it one and for all.
By a common abuse of notation, the set of vertices of this tree is also denoted by $\mathbb T_r$. 
We fix once and for all a distinguished vertex $v_0\in \T_r$ and let $\orT$ be the unique orientation of $ \T_r$ such that the $r$ edges incident to $v_0$ are oriented in outgoing direction from $v_0$, and every vertex different from $v_0$ has one incoming edge and $r-1$ outgoing edges.
We write $v\overset{k}{\longrightarrow}v'\in \orT$ to indicate that the edge in between the vertices $v,v'$ of $\orT$ is oriented from $v$ to $v'$ and is labeled by $k$.

Fix once and for all a seed $\seed_0=(e_i\mid i \in I)$ and call it the {\bf initial seed}.
To every vertex $v\in \T_r$ we attach a seed $\seed_v$ as follows: 
we let $\seed_{v_0}=\seed_0$, if $v\overset{k}{\longrightarrow}v'\in \orT$ then $\seed_{v'}=\mu_k(\seed_{v})$.
For simplicity we write $\seed\in \orT$ if $\seed=\seed_v$ for some $v\in \orT$.

For every seed $\seed=(e_{i;\seed}\mid i\in I)\in \orT$ we introduce the {\bf seed tori} $\cA_{\seed} = T_{N^{\circ}} $ and $ \cX_{\seed} = T_{M}$ which are endowed with the {\bf cluster coordinates} $\{A_{i;\seed} := z^{f_{i;\seed}}\}_{i \in I}$ and $\{X_{i;\seed} := z^{e_{i;\seed}}\}_{i \in I}$, respectively. The {\bf $\cA$-cluster transformation} associated to $\seed$ and $k \in \Iuf$ is the birational map
$
\mu^{\cA}_{k}:\cA_{\seed} \dashrightarrow \cA_{\mu_k(\seed)}
$
specified by the pullback formula
\begin{equation}
\label{A_mut}
(\mu^{\cA}_{k})^*(z^m):=z^{m} (1+z^{v_{k;\seed}})^{-\langle d_k e_{k;\seed},m\rangle} \ \ \text{ for }m\in M^{\circ},
\end{equation}
where $v_{k;\seed}:=\{e_{k;\seed}, \cdot \}\in M^{\circ}$. Similarly, the {\bf $\cX$-cluster transformation} associated to ${\vb s}$ and $k$ is the birational map
$
\mu^{\cX}_{k}:\cX_{\seed} \dashrightarrow \cX_{\mu_k(\seed)} 
$
specified by the pull-back formula
\begin{equation}
\label{X_mut}
(\mu^{\cX}_{k})^*(z^n):=z^{n} (1+z^{e_{k;\seed}})^{-[ n,e_{k;\seed} ]}\ \ \text{ for }n\in N,
\end{equation}
where $[\cdot, \cdot]:N\times N \to \Q$ is the bilinear form determined by setting $[e_i,e_j]=\lrc{e_i, d_je_j}$.

For seeds $\seed, \seed'\in \orT$ connected by iterated mutation in a sequence of directions $k_1, \dots, k_s\in \Iuf$, we let $\mu^{\cA}_{\seed, \seed'} $ (resp. $\mu^{\cX}_{\seed, \seed'} $) be the composition of cluster transformations in the same sequence of directions and in the same order. 
A birational transformation of the form $\mu^{\cA}_{\seed, \seed'}$ (or $\mu^{\cX}_{\seed, \seed'}$) can be used to glue its domain and range by identifying the largest open subschemes where the transformation is an isomorphism. 
We use this kind of gluing to define cluster varieties. 
More precisely, the cluster $ \cA$-variety associated to $\Gamma$ and $\seed_0$ is
\[
\cA_{\Gamma,\seed_0}:=\bigcup\limits_{\seed\in \orT} \cA_{\seed}/ \left( \text{gluing by } \mu^{\cA}_{\seed', \seed''} \right)_{\seed',\seed''\in\orT}.
\]
The cluster $ \cX$-variety associated to $\Gamma$ and $\seed_0$ is
\[
\cX_{\Gamma,\seed_0}:=\bigcup\limits_{\seed\in \orT} \cX_{\seed}/ \left( \text{gluing by } \mu^{\cX}_{\seed', \seed''} \right)_{\seed',\seed''\in\orT}.
\]

From now on an element $\seed\in \orT$ will be referred to as a seed for $\cA$ (or $\cX$).
It is important to recall that declaring another $\seed \in \orT$ as an initial seed gives rise to isomorphic cluster varieties.
We fix the pair $(\Gamma,\seed) $ once and for all and denote $\cA_{\Gamma, \seed_0}$ (resp. $\cX_{\Gamma, \seed_0}$) simply by $\cA$ (resp. $\cX$).

\subsubsection{Quotients of $\cA$-varieties and fibres of $\cX $-varieties}\label{sec:quotients-fibres} 

Let $N^{\perp}_{\text{uf}}:= \{ m\in M \mid \langle n, m \rangle=0 \ \forall \  n\in N_{\text{uf}} \} $.
In particular, $ M/ N^{\perp}_{\uf}\cong (N_{\text{uf}})^*$. 
By a slight abuse of notation we also write $ M^{\circ}/ \Nuf^{\perp}$. Here $\Nuf^\perp$ is taken in $M^\circ$ rather than $M$, so $M^{\circ}/ \Nuf^{\perp}$ is torsion free.
Since $\{ N_{\text{uf}},N \}\subseteq \Z$ the following homomorphisms are well defined
\begin{align} \label{eq:p12star}
  \begin{matrix}
    p_1^*: & N_{\uf} & \rightarrow & M^\circ &\qquad \phantom{aaaaa} \qquad \qquad & p_2^* : & N & \rightarrow& M^{\circ}/ N^{\perp}_{\uf}. \\
    & n &\mapsto & \{ n, \cdot \}
  & \qquad  & &  n &\mapsto &  \{ n, \cdot \} + \Nuf^{\perp}
  \end{matrix}
\end{align}
The matrix representing $p_2^*$ with respect to a seed $\seed\in \orT$ is the \emph{extended exchange matrix} $\widetilde{B}_{\seed}$ of \cite{FZ_clustersIV}.

\begin{definition}\thlabel{def:p-star}
A {\bf cluster ensemble lattice map} for $\Gamma$ is a homomorphism $p^*: N \to M^\circ$ such that $p^*|_{N_{\text{uf}}} = p^*_1$ and the composition $N \overset{p^*}{\longrightarrow} M^\circ \twoheadrightarrow M^{\circ}/ N^{\perp}_{\uf}$ agrees with $p_2^*$, where $ M^\circ \twoheadrightarrow M^{\circ}/ N^{\perp}_{\uf}$ denotes the canonical projection. Note that different choices of $p^*$ differ by a homomorphism $N/ N_{\uf} \rightarrow  N^{\perp}_{\uf} $.
\end{definition}

In other words, given a seed $\vb s $, the $|I|\times|I| $ square matrix $B_{p^*;\vb{s}}$ associated to a cluster ensemble lattice map $p^*$ with respect to the bases $(e_i)_{i\in I}$ and $(f_i)_{i\in I}$ satisfies 
\begin{equation}\label{eq:Mp*}
B_{p^*;\vb{s}} - \epsilon^{\rm{tr}}_\seed=
\lrb{\begin{matrix}
0 & 0 \\
0 & \ast
\end{matrix}},
\end{equation}
where the $0$ entries represent the blocks $\Iuf\times\Iuf$, $\Iuf\times (I\setminus \Iuf)$, and $(I\setminus \Iuf)\times \Iuf$,
and the $\ast$ entry indicates that the $(I\setminus \Iuf)\times(I\setminus \Iuf)$ block has no constraints.
Every cluster ensemble lattice map $p^*:N\to M^{\circ}$ commutes with mutation. Therefore, $p^*$ gives rise to a {\bf cluster ensemble map}
\[
p:\cA \to \cX.
\]

The map $p:\cA \to \cX$ yields both, torus actions on $\cA$ and fibrations of $\cX$ over a torus, as we explain subsequently.
Let
\begin{equation}\label{eq:define K}
K=\ker(p_2^*)=\lrc{k\in N\mid \{k,n\}=0\,\forall\, n\in N_{\rm uf}^\circ} \quad \text{and} \quad K^{\circ}=K \cap N^{\circ}. 
\end{equation}
To obtain an action on $\cA$ we consider  a saturated sublattice 
\[H_{\cA} \subseteq K^\circ.
\]
The inclusion $ H_{\cA} \hookrightarrow N^\circ$ gives rise to an inclusion $T_{H_{\cA}}\hookrightarrow T_{N^{\circ}}$
as a subgroup.
Since $p^*$ commutes with mutation and $H_{\cA}\subseteq K $ we have a non-canonical inclusion 
\[
T_{H_{\cA}}\hookrightarrow \cA.
\]
The action of $ T_{H_{\cA}} $ on $T_{N^\circ}$ given by multiplication extends to a free action of $T_{H_{\cA}} $ on $\cA$ and gives rise to  
a geometric quotient $\cA \to \cA/T_{H_{\cA}}$.
The scheme $\cA/T_{H_{\cA}}$ is obtained by gluing tori of the form  $T_{N^{\circ}/H_{\cA}}\cong T_{N^{\circ}}/T_{H_{\cA}}$; the gluing is induced by the $\cA$-mutations used to glue the seed tori for $\cA$.
More precisely, for every seed $\seed$ for $\cA$ we let $(\cA/T_{H_{\cA}})_{\seed}$ be a copy of the torus $ T_{N^{\circ}/H_{\cA}} $. 
For $k\in \Iuf$ the mutation   $\mu^{\cA/T_{H_\cA}}_{k}: (\cA/T_{H_{\cA}})_{\seed}  \dashrightarrow  (\cA/T_{H_{\cA}})_{\mu_k(\seed)}$ is given by
\begin{equation}
\label{A/T_mut}
\lrp{\mu^{\cA/T_{H_{\cA}}}_{k}}^*(z^m):=z^{m} (1+z^{v_{k;\seed}})^{-\langle d_k e_{k;\seed},m\rangle} \ \ \text{ for }m\in H_{\cA}^{\perp}.
\end{equation}
Let $\mu^{\cA/T_{H_{\cA}}}_{\seed, \seed'}$ denote the composition of mutations determined by the path in $\orT$ connecting $\seed, \seed'\in \orT$. Then 
\[
\cA/T_{H_{\cA}}:=\bigcup\limits_{\seed\in \orT} (\cA/T_{H_{\cA}})_{\seed}/ \left( \text{gluing by } \mu^{\cA/T_{H_{\cA}}}_{\seed', \seed''} \right)_{\seed', \seed'' \in \orT}.
\]

To obtain the fibration of $\cX$ over a torus we consider a saturated sublattice 
\[
H_{\cX} \subseteq K.
\]
The inclusion $H_{\cX} \hookrightarrow N$ induces a surjection $T_M:= \Spec(\Bbbk[N]) \to \Spec(\Bbbk[H_{\cX}])=:T_{H_{\cX}^*}$. This extends to a globally defined map 
\begin{equation}
\label{eq:weight_map}
    w_{H_{\cX}}:\cX\to T_{H^*_\cX}.
\end{equation}

\begin{remark}
The subindex $\cV $ in the lattice $H_{\cV}$ stands for the cluster variety $\cV$ for which the choice of sublattice is relevant. 
When there is no risk of confusion, we drop the subindex $\cV$ from $H_{\cV}$ (see the end of \S\ref{sec:FG_dual}).
\end{remark}

We let $\cX_{\phi}$ be the fibre of the map \eqref{eq:weight_map} over a closed point $\phi\in T_{H^*_{\cX}}$.
In this work we mainly focus on the fibre $\cXeH$, where ${\bf 1}_{T_{H^*_{\cX}}}\in T_{H^*_{\cX}}$ is the identity element. When there is no risk of confusion on the fibration we are considering we will denote this scheme simply by $\cXe$.

The fibre $\cXe$ is obtained by gluing tori isomorphic to $T_{H^\perp_{\cX}}$ via the restrictions of the $\cX$-mutations used to glue the seed tori for $\cX$ (see \cite[\S4]{GHK_birational} for a detailed treatment of this construction).
As in the previous situations, we have a description of the form 
\[
\cXe:=\bigcup\limits_{\seed\in \orT} (\cXe)_{\seed}/ \left( \text{gluing by } \mu^{\cXe}_{\seed', \seed''} \right)_{\seed', \seed'' \in \orT},
\]
where $(\cXe)_{\seed}$ is a torus isomorphic to $T_{H_{\cX}^\perp}$, $\mu^{\cXe}_{k}: (\cXe)_{\seed}  \dashrightarrow  (\cXe)_{\mu_k(\seed)}$ is given by  
\begin{equation}
\label{X_phi_mut}
\lrp{\mu^{\cXe}_{k}}^*(z^{n+H_{\cX}}):=z^{n+H_{\cX}}(1+z^{e_{k;\seed}+H_{\cX}})^{-[ n,e_{k;\seed} ]}\ \ \text{ for } n+H_{\cX} \in N/H_{\cX}
\end{equation}
and $\mu^{\cXe}_{\seed,\seed'}$ is defined as for the other varieties we have introduced so far.

\begin{definition}
\label{def:quotient_fibre}
A variety of the form $\cA/T_{H_{\cA}} $ is referred to as a {\bf quotient of $\cA$}. A variety of the form $\cXe$ is referred to as a {\bf fibre of $\cX$}. 
A {\bf cluster action} on $\cA$ is the action of a torus of the form $T_{H_{\cA}}$.
\end{definition}

Let $T$ be an algebraic torus endowed with a set of coordinates $z_1, \dots , z_r$ and let $\omega_T$ be its canonical bundle. 
A {\bf volume form} on $T$ is a nowhere vanishing form in $H^0(T, \omega_T) $.
The {\bf standard volume form} on $T$ is (any non-zero scalar multiple of) 
\[
\Omega_T= \frac{dz_1 \wedge \dots \wedge dz_r}{z_1 \cdots z_r}.
\]

\begin{definition}
	A {\bf log Calabi–Yau pair} $(Y, D)$ is a smooth complex projective variety $Y$ together with a reduced normal crossing divisor\footnote{This is not the most general definition, but it is sufficient for the current discussion.  See \cite[Section~1]{GHK_birational} for a more in depth discussion.} $D\subset Y$ such that $K_X+D=0$. 
We say a scheme $V$ is log Calabi–Yau if there exists a log Calabi–Yau pair $(Y,D)$ such that $V$ is $Y \setminus D$ up to codimension 2.
\end{definition}

It follows from \cite{Iitaka} that any log Calabi--Yau variety $V$ is endowed with a unique up to scaling holomorphic volume form (\ie a nowhere vanishing holomorphic top form) $\Omega_V$ which has at worst a simple pole along each component of $D$ for any such $(Y,D)$.  See \cite{GHK_birational} for further details.

As explained in \cite[\S1]{GHK_birational} both $ \cA$ and $\cX$ are log Calabi--Yau, the key point being that these schemes are obtained by gluing tori via birational maps that preserve the standard volume form on each seed torus (endowed with cluster coordinates). 
For the same reason, the schemes of the form $\cA/T_{H_{\cA}}$ and $\cX_{\phi}$ are also log Calabi--Yau.
The canonical volume form on $\cA/T_{H_{\cA}}$ (resp. $\cX_{\phi}$) is induced by (resp. the restriction of) the canonical volume form of $\cA$ (resp. $\cX$).

\subsubsection{Principal coefficients, $\cX$ as a quotient of $\cAp $ and $\cA$ as a fibre of $\cXp$}
\label{sec:principal_coefficients}
For the fixed data $\Gamma=\lrp{I, \Iuf, N,N^{\circ}, M, M^{\circ}, \{ \cdot, \cdot \}, \{d_i\}_{i\in I} }$, we consider its principal counterpart 
\[
\Gamma_{\prin}=\lrp{I_{\prin}, (I_{\prin})_{\text{uf}}, N_{\prin}, N_{\prin}^{\circ}, M_{\prin}, M^{\circ}_{\prin}, \{ \cdot, \cdot \}_{\prin}, \{d_i\}_{i\in I_{\prin}} },
\]
where the index set $I_\prin$ is the disjoint union of two copies of $I$, its subset $(I_\prin)_{\text{uf}}$ is the set $\Iuf$ thought of as a subset of the first copy of $I$, 
\[
       N_{\prin} = N \oplus M^\circ, \quad  N_{\prin}^{\circ}= N^{\circ}\oplus M, \quad (N_{\prin})_{\text{uf}}=\Nuf \oplus 0, \quad M_{\prin} = M \oplus N^\circ, \quad M_{\prin}^{\circ}=M^{\circ}\oplus N.
\]
For $i \in I_{\prin}$ belonging to either the first or second copy of $I$, the corresponding integer in the tuple $\{d_i \mid i\in I_{\prin}\}$ is equal to integer indexed by $i$ for $\Gamma$, and 
\[
\{(n_1,m_1),(n_2,m_2)\}_{\prin}= \{n_1, n_2\} + \langle n_1,m_2 \rangle - \langle n_2,m_1 \rangle.
\]
Recall that $\seed_0=(e_i)_{i \in I} $ is the initial seed for $\Gamma$. Then the initial seed for $\Gamma_{\prin}$ is ${\seed_0}_{\prin}=\lrp{(e_i,0),(0,f_i)}_{i\in I}$.
Since $\Gamma$ and $\seed_0$ were already fixed, we denote the cluster variety $\cA_{\Gamma_{\prin},{{\seed{_0}}_{\prin}}}$ (resp. $\cX_{\Gamma_{\prin},{{\seed{_0}}_{\prin}}}$) simply by $\cAp$ (resp. $\cX_{\prin}$). It is moreover worth pointing out that $\cAp$ is in fact independent of the choice of initial seed $\seed_0$ as explained in \cite[Remark B.8]{GHKK}.

In \cite{GHK_birational} the authors show that the scheme $\cX$ can be described as a quotient of $\cAp$ in the sense of Definition \ref{def:quotient_fibre}. 
To obtain such a description we need to choose a cluster ensemble lattice map $p^*:N \to M^{\circ}$ for $\Gamma$. 
This choice determines the cluster ensemble map 
\begin{equation}
\label{eq:def_p_prin}
p_{\prin}: \cAp \to \cXp.
\end{equation}
The map $p_{\prin}$ is induced by the cluster ensemble lattice map
\begin{align*}
p_{\prin}^*:N_{\prin} &\to \Mpc\\
(n,m) &\mapsto \lrp{p^*(n)-m,n}
\end{align*}
for $\Gamma_\prin$.
Set $K_{\prin}:=\ker(p_{\prin,2}^*)$ and $ K_{\prin}^\circ:= K_{\prin}\cap N^\circ_\prin$, where $p_{\prin,2}^*$ corresponds to the map $p_2^*$ in \eqref{eq:p12star} for $\Gamma_{\prin}$.
We let
\begin{equation}
\label{eq:H_Aprin} 
  H_{\cAp}:= \lrc{\left.\lrp{n,-(p^*)^*(n)}\in N^\circ_\prin \, \right| \, n \in N^\circ}.  
\end{equation}

It is straightforward to verify that $H_{\cAp}$ is a saturated sublattice of $K^\circ_\prin$ that is isomorphic to $N^\circ$.
In particular, we have a quotient $\cAp/ T_{H_{\cAp}}$ endowed with an atlas of seed tori isomorphic to $T_M $ (indeed, $T_{N^\circ_{\prin}}/T_{H_{\cAp}}\cong T_{N^\circ \oplus M}/T_{N^\circ}\cong T_M$). 
There is an isomorphism 
\begin{equation}
    \label{eq:def_chi}
    \chi : \cAp/T_{H_{\cAp}}\overset{\sim}{\longrightarrow} \cX
\end{equation}
respecting the cluster tori of domain and range.
The restriction of $\chi $ to a seed torus is a monomial map whose pullback is given by
\eqn{
\chi^*: N &\to (H_{\cAp})^\perp 
\\
n &\mapsto (p^*(n),n).
}
There is also a surjective map 
\begin{equation}
\label{eq:def_tilde_p}
    \tilde{p}:\cAp \to \cX.
\end{equation}
respecting seed tori.
The restriction of $\tilde{p} $ to a seed torus is a monomial map whose pullback is given by
\begin{align*}
    \tilde{p}^*:  N &\to \Mpc\\
       \ \ n &\mapsto  (p^*(n),n).
\end{align*}
In particular, we have $ \tilde{p}= \chi\circ \varpi$, where 
\begin{equation}
    \label{eq:def_varpi}
\varpi: \cAp \to \cAp/T_{H_{\cAp}} 
\end{equation}
is the canonical projection. 

It is also possible to describe $\cA$ as a fibre of $\cXp$. There is an injective map
\begin{equation}
\label{eq:def_xi}
\xi:\cA \to \cXp
\end{equation}
respecting seed tori.
The restriction of $\xi $ to a seed torus is a monomial map whose pullback is given by
\begin{align*}
\xi^*: N_{\prin} &\to M^\circ
\\
(n,m) &\mapsto p^*(n)-m.
\end{align*}
Let
\begin{equation}
    \label{eq:H_Xprin} 
H_{\cXp}:= \lrc{\lrp{n,p^*(n)}\in N_\prin \mid n \in N}.
\end{equation}
It is routine to check that $H_{\cXp}$ is a valid choice to construct a fibration of $\cXp $ over the torus $T_{H^*_{\cXp}}$. 
Hence, we can consider the fibre $(\cXp)_{\bf 1}=(\cXp)_{{\bf 1}_{T_{H^*_{ \cXp}}}} $ associated to this fibration. 
There is an isomorphism 
\begin{equation}
\label{eq:def_delta}
\delta:
\cA \overset{\sim}{\longrightarrow} (\cXp)_{\bf 1}     
\end{equation}
respecting seed tori.
The restriction of $\delta $ to a seed torus is a monomial map whose pullback is given by
\begin{align*}
\delta^*: N_{\prin} /H_{\cXp} &\to M^\circ
\\
(n,m) + H_{\cXp} &\mapsto p^*(n)-m.
\end{align*}
In particular, we have that
\[
\xi=\iota \circ \delta,
\]
where $\iota: (\cXp)_{\bf 1}\hookrightarrow \cXp$ is the canonical inclusion.
For later reference we also introduce the map
\begin{equation}
\label{eq:def_rho}
\rho: \cXp \to \cX.
\end{equation}
respecting seed tori.
The restriction of $\rho $ to a seed torus is a monomial map whose pullback is given by
\begin{align*}
    \rho^*:  N &\to \Np \\
       \ \ n & \mapsto  (n,p^*(n)).
\end{align*}
In particular, $\rho \circ p_{\prin}= \tilde{p} $.
The maps we have considered so far fit into the following commutative diagram \begin{equation*}
\xymatrix{
(\cXp)_{\bf 1} \ar@{^{(}->}^{\ \iota}[r] & \cXp \ar_{\rho}[d] & \cAp \ar_{p_{\prin}}[l] \ar@{->>}^{\varpi}[d] \ar_{\tilde{p}}[dl] \\
\cA \ar^{\delta}_{\cong}[u] \ar_{p}[r] \ar^{\xi}[ru] & \cX & \cAp/T_{H_{\cAp}.} \ar_{\cong \ \ }^{\chi \ \ }[l]
}
\end{equation*}

\begin{remark}
    \label{rem:labels}
The maps introduced in this section are associated with $\Gamma$, hence, we label the maps with the subindex $\Gamma$ to stress the fixed data $\Gamma$ they are associated with.
\end{remark}

\subsection{Tropicalization}
\label{ss:tropicalization}
In this section we discuss  tropicalizations of cluster varieties. We mainly follow \cite[\S1]{GHK_birational}, \cite[\S2]{GHKK} and \cite[\S1.1]{FG_cluster_ensembles}. 

Let $T_L$ be the torus associated to a lattice $L$. A rational function $f$ on $T_L$ is called positive if it can be written as a fraction $f=f_1/f_2$, where both $f_1$ and $f_2$ are a linear combination of characters of $T_L$ with coefficients in $\Z_{>0} $.
The collection of positive rational functions on $T_L$ forms a semifield inside $ \Bbbk(T_L)$ denoted by $Q_{\rm sf}(L)$.
A rational map $f:T_L\dashrightarrow T_{L'}$ between two tori is a {\bf positive rational map} if the pullback $f^*:\Bbbk(T') \to \Bbbk(T)$ restricts to an isomorphism $f^*:Q_{\rm sf}(L') \to Q_{\rm sf}(L)$.
If $P$ is a semifield, then the $P$ valued points of $T_L$ form the set
\begin{equation}
\label{eq:FG_tropicalization}
T_L(P):=\Hom_{\rm sf}(Q_{\rm sf} (L), P)
\end{equation}
of semifield homomorphisms from $Q_{\rm sf} (L)$ to $ P$.
In particular, a positive birational isomorphism $\mu:T\dashrightarrow T'$ induces a bijection
\begin{align*}
\mu_*: T(P) & \to T'(P)\\
 h \ & \mapsto \ h \circ f^*.    
\end{align*}
By a slight but common abuse of notation the sublattice of monomials of $Q_{\rm sf}(L)$ is denoted by $L^*$. 
Considering $P$ just as an abelian group the restriction of an element of $Q_{\rm sf}(L)$ to $L^*$
determines a canonical bijection $T_L(P) \overset{\sim}{\longrightarrow} \Hom_{\rm groups} (L^*, P) $. 

\begin{remark}
\label{rem:identification}
We systematically identify $T_L(P)$ with $L\otimes P$ by composing the canonical bijection $T_L(P) \overset{\sim}{\longrightarrow} \Hom_{\rm groups} (L^*, P) $ with the canonical isomorphism $\Hom_{\rm groups}(L^*, P) \cong L \otimes P$.
\end{remark}

Let $ \cV$ be a (quotient or a fibre of a) cluster variety.
For every $\seed, \seed'\in \orT$ the gluing map $\mu^{\cV}_{\seed , \seed'}: \cV_\seed\dashrightarrow  \cV_\seed'$ is a positive rational map.
So we can glue $\cV_{\seed}(P) $ and $ \cV_{\seed'} (P)$ using $(\mu^{\cV}_{\seed, \seed'})_*$ and define 
\[
\cV(P):= \coprod_{\seed \in \orT} \cV_{\seed}(P) / \left(\text{gluing by } (\mu^{\cV}_{\seed, \seed'})_*\right)_{\seed, \seed'\in \orT}.
\]
Every point ${\bf a}\in \cV(P)$ can be represented as a tuple $(a_{\seed})_{\seed\in \orT}$ such that $(\mu^{\cV}_{\seed, \seed'})_*(a_\seed)=(a_{\seed'}) $ for all $\seed,\seed'\in \orT$.
Since all of the maps $(\mu^\cV_{\seed,\seed'})_*$ are bijections, the assignment 
\eq{
   \mathfrak{r}_{\seed}:\cV(P)&\to \cV_{\seed}(P)\quad \text{given by} \quad {\bf a}=(a_{\seed})_{\seed \in \orT} \mapsto a_{\seed}.
}{not:tropical_space}
determines an identification of $\cV(P) $ with $ \cV_{\seed}(P)$. 
If $S\subset \cV(P)$ we let 
\begin{equation}
\label{eq:identification}
S_{\seed}(P):=\mathfrak{r}_{\seed} (S) \subset \cV_\seed(P)
\end{equation}
and write $S_{\seed}$ instead of $S_{\seed}(P)$ when the semifield $P$ is clear from the context.

The semifields we consider in this note are the integers, the rationals and the real numbers with their additive structure together with the semifield operation determined by taking the maximum (respectively, minimum).
We denote these semifields by $\Z^T$, $\Q^T$ and $\R^T$ (respectively, $\Z^t$, $\Q^t$ and $\R^t$).
The canonical inclusions $\Z \hookrightarrow \Q \hookrightarrow \R$ give rise to canonical inclusions
\[
\cV(\Z^T) \hookrightarrow \cV(\Q^T) \hookrightarrow \cV(\R^T) \quad \quad \text{ and } \quad \quad  \cV(\Z^t) \hookrightarrow \cV(\Q^t) \hookrightarrow \cV(\R^t).
\]
For a set $S\subseteq \cV(\R^T)$ (resp. $S\subseteq \cV(\R^t)$) we let $
S(\Z):= S\cap \cV(\Z^T)$ (resp. $S(\Z):= S\cap \cV(\Z^t)$). Moreover, for $G=\Z, \Q$ or $\R$, there is an isomorphism of semifields $G^T\to G^t$ given by $x \mapsto -x$ induces a canonical bijection
\begin{align} \label{eq:imap}
    i: \cV(G^T) \rightarrow \cV(G^t). 
\end{align}
Since $i$ amounts to a sign change (see Remark \ref{rem:i_map} below), we think of $i$ as an involution and denote its inverse again by $i$.

\begin{remark}\label{rmk:geometric trop}
The set $\cV(\Z^t)$ can be identified with the {\bf geometric tropicalization} of $\cV$, 
defined as
\begin{equation*}
    \cV^{\trop}(\Z) 
     \coloneqq \{ \text{divisorial discrete valuations } \nu: \Bbbk(\cV) \setminus \{ 0\} \rightarrow \mathbb Z \mid \nu (\Omega_{\cV}) <0 \} \cup \{ 0\}, 
\end{equation*}
where a discrete valuation is divisorial if it is given by the order of vanishing of a $\Z_{>0}$-multiple of a prime divisor on some variety birational to $\cV$. 
\end{remark}

\begin{remark}
\label{rem:i_map}
Let $G=\Z, \Q$ or $\R$. Identifying $\cV(G^T)$ with $\cV_{\seed}(G^T)$ via the bijection $\mathfrak{r}_\seed$ the map $i$ in \eqref{eq:imap} can be thought of as the multiplication by $-1$ (\cf Remark \ref{rem:identification}).
\end{remark}

A positive rational function $g$ on $\cV $ is a rational function on $\cV$ such that the restriction of $g$ to every seed torus $\cV_{\seed}$ is a positive rational function.

\begin{definition}
\label{def:trop_functions}
The {\bf tropicalization} of a positive rational function $g: \cV \dashrightarrow \Bbbk$ with respect to $\R^T $ is the function $g^T:\cV(\R^T)\to \R$ given by 
\begin{equation}
\label{eq:restriction}
{\bf a}\mapsto a_{\seed}(g),
\end{equation}
where $ {\bf a}=(a_{\seed})_{\seed \in \orT}$. The tropicalization of $g$ with respect to $\R^t$ is the function $g^t:\cV(\R^t)\to \R$ defined as
\[
{\bf v} \mapsto -v_{\seed}(g),
\]
\end{definition}
where $ {\bf v}=(v_{\seed})_{\seed \in \orT}$. A direct computation shows that both $g^T$ and $g^t$ are well defined. Namely, one checks that for $\seed,\seed'\in \orT$
\[
a_{\seed} (g)=a_{\seed'}(g),
\]
where in the left (resp. right) side of the equality we think of $g$ as a rational function on $\cV_\seed$ (resp. $\cV_{\seed'}$).
Moreover, we have that
\begin{equation}
\label{eq:comparing_tropicalizations}
    g^T({\bf a})=g^t(i({\bf a})),
\end{equation}
for all ${\bf a} \in \cV(\R^T)$.
\begin{remark}
In order to keep notation lighter we adopt the following conventions: 
\begin{itemize}
\item given a positive  rational function $g\in \Bbbk (\cV)=\Bbbk(\cV_\seed)$ the tropicalizations of $g$ with domains $\cV(\R^T)$ and $\cV_{\seed}(\R^T)$ are denoted by the same symbol $g^T$ for all $\seed\in \orT$;

\item the restriction of $g^T$ (resp. $g^t$) to $\cV(\Z^T)$ (resp. $\cV(\Z^t)$) is also denoted by $g^T$ (resp. $g^t$);
\item when $P$ is one of $\Z^T, \Q^T$ or $\R^T$ (resp. $\Z^t, \Q^t$ or $\R^t$) the map $(\mu^{\cV}_{\seed, \seed'})_*$ is denoted by $(\mu^{\cV}_{\seed, \seed'})^T$ (resp. $(\mu^{\cV}_{\seed, \seed'})^t$).
\end{itemize}
\end{remark}

\begin{remark}
    Later we will need to systematically consider $\cV(\R^t)$ when $\cV$ is a variety of the form $\cA$ or $\cA/T_H$ and $ \cV(\R^T)$ when $\cV$ is a variety of the form $\cX$ or $\cXe$.
    In particular, from \S\ref{sec:FG_conj} on we use the notation $\Trop_{\R}(\cV)$ that takes into account the different kinds of tropicalizations that we use for different kinds of varieties, see equation \eqref{eq:unif}. 
\end{remark}

For latter use we record the following formulae associated to the mutations determined by $\Gamma$:
\begin{equation}
    \label{eq:tropical_A_mutation}
    \lrp{\mu^{\cA}_{k}}^T(n)=n+[\langle v_k,n\rangle]_+(-d_ke_k)
\end{equation}
and
\begin{equation}
    \label{eq:tropical_X_mutation}
    \lrp{\mu^{\cX}_{k}}^T(m)=m+[\langle d_ke_k,m \rangle]_+v_k.
\end{equation}
In case we tropicalize these mutations with respect to $\R^t$ we replace $[\ \cdot\ ]_+$ by $[\ \cdot\ ]_-$.

Finally, if we think of $T_L (\R^T)$ (resp. $T_L(\R^t)$) as a vector space (see Remark \ref{rem:identification}), the tropicalization of a positive Laurent polynomial $g= \sum_{\ell\in L^*}c_{\ell} z^{\ell} \in Q_{\rm sf}(L)$  with respect to $\R^T$ (resp. $\R^t$) is the function $g^T:  T_L(\R^T) \to \R$ (resp. $g^t:  T_L(\R^t) \to \R$) given by 
\begin{eqnarray*}
x &\mapsto& - \max \{ \langle \ell , x  \rangle \mid \ell\in L^* \text{ such that } c_{\ell} \neq 0 \}\\
 (\text{resp. } x &\mapsto& \min \{ \langle \ell , x  \rangle \mid \ell\in L^* \text{ such that } c_\ell \neq 0 \}).
\end{eqnarray*}

\section{Theta functions and their labeling by tropical points}
\label{sec:tf_and_parametrizations}

\subsection{Fock--Goncharov duality}\label{sec:FG_dual}
For $\Gamma=(I, \Iuf, N,N^{\circ}, M, M^{\circ}, \{ \cdot, \cdot \}, \{d_i\}_{i\in I} )$ the Langlands dual fixed data
is $\Gamma^\vee=(I, \Iuf, N^\vee, (N^\vee)^{\circ}, M^\vee, (M^\vee)^{\circ}, \{ \cdot, \cdot \}^\vee, \{d^\vee_i\}_{i\in I} )$, where $d:=\text{lcm}(d_i)_{i\in I}$,
\[
     N^\vee = N^\circ, \quad  (N^\vee)^{\circ}= d\cdot N, \quad  M^\vee = M^\circ, \quad (M^\vee)^{\circ}=d^{-1}\cdot M, \quad \{\cdot, \cdot \}^\vee= d^{-1}\{\cdot, \cdot \} \quad \text{and} \quad d^\vee_i:=d\,d_i^{-1}.
\]
If $ \seed=(e_i)_{ i\in I}$ is a seed for $\Gamma$ then the Langlands dual seed is $\seed^\vee:=(e_i^\vee)_{i\in I}$, where $e_i^\vee:=d_ie_i$. We also set $v^\vee_i:=\{e^\vee_i, \cdot \}^\vee$ These constructions give rise to {\bf Langlands dual cluster varieties} which we denote as follows
\begin{align*}
	\begin{array}{l l l l}
	{}^L(\cA_{\Gamma;\seed_0}) := \cA_{\Gamma^\vee;\seed_0^\vee} \qquad \qquad & 	 \text{and} \qquad \qquad & {}^L(\cX_{\Gamma; \seed_0}) := \cX_{\Gamma^\vee; \seed_0^\vee}.
\end{array}
\end{align*}
Since $\Gamma $ and $\seed_0$ were already fixed, we denote ${}^L(\cA_{\Gamma;\seed_0})$ (resp. ${}^L(\cX_{\Gamma;\seed_0})$) simply by $\LA$ (resp. $\LX$).

\begin{definition}
The {\bf Fock--Goncharov dual} of $\cA$ (resp. $\cX$) is the cluster variety $\cA^{\vee}$ (resp. $\cX^{\vee}$) given by
\begin{equation*}
    \cA^{\vee} := {}^L\cX \qquad \qquad  	 \text{and} \qquad \qquad  \cX^{\vee} := {}^L\cA.
\end{equation*}
\end{definition}

In particular, we have that
\[
 \cAp^\vee = {}^L(\cX_\prin)=\cX_{(\Gamma_\prin)^\vee} \qquad \qquad  	 \cXp^{\vee}= {}^L(\cAp)=\cA_{(\Gamma_\prin)^\vee}.
\]

\begin{remark}
    \label{rem:Lprin}
Notice that $\cA_{(\Gamma_{\prin})^\vee}$ (resp. $\cX_{(\Gamma_{\prin})^\vee}$) is canonically isomorphic to  $\cA_{(\Gamma^\vee)_{\prin}}$ (resp. $\cX_{(\Gamma^\vee)_{\prin}}$). Hence, we frequently identify these schemes without making reference to the canonical isomorphisms between them.
\end{remark}

It is not hard to see that the map 
\begin{eqnarray}\label{eq:L p}
    {(\Lp)^*:= -d^{-1}(p^*)^*:N^\vee \to (M^\vee)^{\circ}}
\end{eqnarray}
is well defined and is a cluster ensemble lattice map for the Langlands dual data $\LGam$, where $(p^*)^*$ is the lattice map dual to $p^*$. 
Indeed, in the bases for $ N^\vee $ and $ (M^\vee)^{\circ}$ determined by $\seed^\vee $, and in comparison with the matrix $B_{p^*;\vb{s}}$ in \eqref{eq:Mp*}, the matrix of $(\Lp)^*$ is of the form
\begin{equation*}
B_{(\Lp)^*;\seed^\vee}= -B_{p^*;\seed}^{\rm{tr}}.
\end{equation*}
In particular, we have an associated dual cluster ensemble map
\[p^\vee:\cA^\vee \to \cX^\vee.\]

We proceed to introduce the Fock--Goncharov dual for a quotient of $\cA$. 
So consider a cluster ensemble lattice map $p^*:N\to M^{\circ}$ for $\Gamma$ and the cluster ensemble lattice map $(p^\vee)^*:N^\vee\to (M^\vee)^\circ$ for $\Gamma^\vee$.
Recall from \eqref{eq:define K} that $K=\ker(p_2^*)$.
Similarly, we set

\[
K^\vee=\ker((p^\vee)_2^*)=\{k\in N^\circ\mid \{k,n\}=0 \text{ for all } n\in d\cdot N_{\rm uf}\},
\]
where $(p^\vee)_2^*$ is the map $p^*_2$ of \eqref{eq:p12star} for $\Gamma^\vee$.
Let $H_{\cA}\subseteq K^\circ$ be a saturated sublattice and consider the quotient $\cA/T_{H_\cA}$. 
Recall from \S\ref{sec:quotients-fibres} that $\cA/T_{H_\cA}$ is obtained by gluing tori of the form $T_{N^{\circ}/H_{\cA}}$. 
Since $N^{\circ}/H_{\cA} $ and $H_{\cA}^{\perp} \subset M^\circ$ are dual lattices the Fock--Goncharov dual of $\cA/T_{H_{\cA}}$ should be a fibre of $\cA^\vee$ obtained by gluing tori of the form $T_{H_{\cA}^{\perp}}$. In order to construct it notice that for $n$ in $\Nuf$ we have $\langle k,p^*(n)\rangle = -{d}^{-1}\{k,dn\}=\langle dk,-(p^\vee)^*(n)\rangle$. This implies that
\[
K^\circ=p^*(N_{\rm uf})^\perp=K^\vee.
\]
In particular, $H_{\cA}$ is a saturated sublattice of $K^\vee$ as it is saturated in $K^\circ$.
It is therefore possible to find $T_{H_{\cA}^*}$ as the base of a fibration of the form \eqref{eq:weight_map} for $\cA^{\vee}$ as we are allowed to set
\[
H_{\cA^{\vee}}=H_{\cA}\subseteq K^\vee.
\]
So consider the fibration
\[
w_{H_{\cA}}:\cAm \to T_{H_\cA^*}.
\]
Notice that the fibre $(\cAm)_{{\bf 1}_{T_{H_\cA^*}}}$ is obtained gluing tori of the form $T_{H_{\cA}^\perp}$ as desired. 
Therefore, we define the Fock--Goncharov dual of the quotient $\cA/T_{H_\cA}$ as
\[
\cAHAm := (\cAm)_{{\bf 1}_{T_{H_\cA^*}}}=\lrp{{}^L\cX}_{{\bf 1}_{T_{H_\cA^*}}}.
\]

Similarly, let $H_{\cX}\subseteq K$ be a saturated sublattice and let $w_{H_{\cX}}:\cX \to T_{H^*_\cX}$ be the associated fibration. 
Recall that $\cX_{{\bf 1}_{T_{H^*_\cX}}}$ is obtained by gluing tori of the form $T_{H_{\cX}^{\perp}}$. 
Its Fock--Goncharov dual is a quotient of $\cXm$ glued from tori of the form $T_{(H^\perp_\cX)^*}$ which we construct next.
A direct computation shows that $d\cdot H_{\cX}$ is a saturated sublattice of $(K^\vee)^\circ$. 
In particular, we are allowed to choose
\[
H_{\cXm}= d\cdot H_{\cX}\subseteq (K^\vee)^\circ 
\]
as a sublattice giving rise to a quotient $ \LA/T_{d\cdot H_{\cX}}$. This quotient is obtained by gluing tori of the form $T_{d\cdot N}/T_{d\cdot H_\cX}\cong T_{N/H_\cX}\cong T_{(H^{\perp}_\cX)^*}$. 
Therefore, we define the Fock--Goncharov dual of $\cX_{{\bf 1}_{T_{H^*_{\cX}}}}$ as 
\[
\cXeHm := \cXm/T_{ H_{\cX^{\vee}}}={}^L\cA/T_{d\cdot H_\cX}.
\]

In what follows, when we consider a saturated  sublattice $H$ of $K^\circ$ and write expressions such as $\cA/T_{H}$ or $w_{H}:\cAm\to T_{H^*}$ we will be implicitly assuming that we have set 
\[
H_{\cA}= H = H_{\cAm}.
\]
Similarly, when $H$ is a saturated sublattice of $K$ and we write expressions such as $w_{H}:\cX \to T_{H^*}$, $\cXe$ or $\cXem$ we will be implicitly assuming that we have set
\[
H_{\cX}= H = d^{-1}\cdot H_{\cXm},
\]
\[
\quad  \cXe = \cXeH \quad \quad \text{and} \quad \quad \cXem= \cXm /T_{H_{\cXm}}.
\]

\begin{remark}
Let $\cV$ be (a quotient of) $\cA$ or  (a fibre of) $\cX$. In the skew-symmetric case Arg\"uz and Bousseau \cite{AB22} showed that $\cV$ and $\cV^{\vee}$ are mirror dual schemes from the point of view of \cite{GS22}.
A similar result is proven for the skew-symmetrizable case when $\cV$ has dimension $2$ in \cite{Mandy_rank2_MS} with arguments that may be generalized to arbitrary dimension.
\end{remark}

\subsection{Scattering diagrams and theta functions}
\label{sec:scat}
Theta functions are a particular class of global function on (quotients and fibres of) cluster varieties introduced in \cite{GHKK}.
In this subsection we outline their construction.
The main case to consider is the one of $\cAp$ since scattering diagrams and theta functions for (quotients of) $\cA$ and (fibres of) $\cX$ can be constructed from this case.

\begin{remark}
\label{rem:full_rank_assumption}
From now on, whenever we consider  the variety $\cA=\cA_{\Gamma,\seed_0}$ we will assume $\Gamma$ is of {\bf full-rank}. 
By definition this means that the map $p_1^*:\Nuf \to M^\circ$ given by $n \mapsto \{ n , \cdot \}$ is injective. 
There are various results of this article for $\cA$ that are valid even if $\Gamma$ is not of full-rank. 
However, various key results we shall use do need the full-rank condition (\cf Remark \ref{rem:all_from_cAp}). 
Even though we are imposing full-rank assumption we will frequently recall that we are assuming it to insist on the necessity of the assumption.
\end{remark}

\subsubsection{Theta functions on full-rank $\cA$}
\label{sec:tf_A}
Throughout this section we systematically identify $\cA^{\vee}_{\seed^\vee}(\R^T)$ with $M^\circ_{\R}$, see \S\ref{ss:tropicalization}.
A {\bf wall} in $M^{\circ}_\R$ is a pair $(\wall, f_{\wall})$ where
 $\wall \subseteq M^{\circ}_\R$ is a convex rational polyhedral cone of codimension one, contained in $n^{\perp}$ for some $n \in N_{\uf, \seed}^+$, and 
 $f_{\wall}  = 1+ \sum_{k \geq 1} c_k z^{kp^*_1(n)}$ is called a {\bf scattering function}, where $c_k \in \Bbbk$. 
A {\bf scattering diagram} $\scat $ in $M^{\circ}_\R$ is a (possibly infinite) collection of walls satisfying a certain finiteness condition (see \cite[\S1.1]{GHKK}). 
The {\bf support} and the {\bf singular locus} of $\scat $ are defined as
\[
\Supp(\scat):= \bigcup_{\wall \in \scat} \wall \ \ \ \text{and} \ \ \ \Sing(\scat):= \bigcup_{\wall \in \scat} \partial\wall \ \cup \bigcup_{\overset{\wall_1,\wall_2 \in \scat}{\text{dim}(\wall_1 \cap \wall_2) = |I|-2}} \wall_1 \cap \wall_2.
\]

A wall $(\wall, f_{\wall})$ defines a {\bf wall-crossing automorphism} $\mathfrak{p}_{\wall}$ of $\Bbbk (M)$  
given in a generator $ z^m$ by $
\mathfrak{p}_{\wall}(z^m)=z^m f_{\wall}^{\langle n_{\wall}, m \rangle }$,
where $n_{\wall}$ is the primitive normal vector of the wall $\wall$ with a choice of direction going against the flow of the path $\gamma$. 
If we fix a scattering diagram $\scat$ and a piecewise linear proper map $ \gamma:[0,1]\to M^\circ_{\R}\setminus \Sing(\scat)$ intersecting $\text{Supp}( \mathfrak{D})$ transversely 
then the {\bf path ordered product} $ \mathfrak{p}_{\gamma , \scat}$ is defined as the composition of automorphisms of the form $\mathfrak{p}_{\wall}$, where we consider the walls $\wall$ that are transversely crossed by $\gamma $. However, observe that $\gamma$ might cross an infinite number of walls, therefore, we would be potentially composing an infinite number of automorphisms and such infinite composition is well defined.  
Again, the reader is referred to \cite[\S 1.1]{GHKK} for a detailed discussion.

\begin{definition}
A scattering diagram $\scat$ is {\bf consistent} if for all $\gamma$ as above $\mathfrak{p}_{\gamma, \scat}$ only depends on the endpoints of $\gamma$. Two scattering diagrams $\scat$ and $\scat'$ are {\bf equivalent} if $\mathfrak{p}_{\gamma, \scat}= \mathfrak{p}_{\gamma, \scat'}$ for all $\gamma$.
\end{definition}

To define cluster scattering diagrams for  $\cA$ one first considers
\[
\scat_{{\rm in}, \seed}^{\cA} := \lrc{\left.\left( e_i^{\perp} , 1+z^{ p_{1}^*\left( e_i \right)}\right) \right| \  i \in \Iuf }.
\]
A {\bf cluster scattering diagram} for $\cA$ is a consistent scattering diagram in $M^{\circ}_{ \R}$ containing $\scat_{{\rm in}, \seed}^{\cA}$. By the following theorem, cluster scattering diagrams for $\cA$ do exist (provided $\Gamma$ is of full-rank). 

\begin{theorem} \cite[Theorem 1.12 and 1.13]{GHKK}
\label{thm:consistent_scattering_diagrams}
Assume $\Gamma$ is of full-rank. Then for every seed $\seed$ there is a consistent scattering diagram $\scat_{\seed}^{\cA} $ such that $\scat_{{\rm in}, \seed}^{\cA} \subset \scat_{\seed}^{\cA}$.
Furthermore $\scat_{\seed}^{\cA}$ is equivalent to a scattering diagram all of whose scattering functions are of the form ${ f_{\wall} = (1+ z^{p_{1}^*(n)})^c}$, for some $n \in N$, and $c$ a positive integer. 
\end{theorem}

\begin{definition} \thlabel{def:genbroken}
Fix a cluster scattering diagram $\scat^{\cA}_{\seed}$.
Let $\mono \in M^{\circ} \setminus \{0\}$ and $x_0 \in M^{\circ}_{\R} \setminus \text{Supp}(\scat)$. 
A (generic) {\bf broken line} for $\scat^{\cA}_{\seed}$ with initial exponent $\mono$ and endpoint $x_0$ is a piecewise linear continuous proper path $\gamma : ( - \infty , 0 ] \rightarrow M^\circ_{\R} \setminus \Sing (\scat^{\cA}_{\seed})$ bending only at walls, with a finite number of domains of linearity $L$ and a monomial $c_L z^{\mono_L} \in \Bbbk[M^\circ]$ for each of these domains. The path $\gamma$ and the monomials $c_L z^{\mono_L}$ are required to satisfy the following conditions:
\begin{itemize}
\setlength\itemsep{0em}
    \item $\gamma(0) = x_0$.
    \item If $L$ is the unique unbounded domain of linearity of $\gamma$, then $c_L z^{\mono_L} = z^{\mono}$.
    \item For $t$ in a domain of linearity $L$, $\gamma'(t) = -\mono_L$.
    \item If $\gamma$ bends at a time $t$, passing from the domain of linearity $L$ to $ L'$ then $c_{L'}z^{\mono_{L'}}$ is a term in $\mathfrak{p}_{{\gamma}|_{(t-\epsilon,t+\epsilon)},\scat_t} (c_L z^{m_L}) $, where ${\scat_t = \lrc{\left.(\wall, f_{\wall}) \in \scat^{\cA}_{\seed}\right| \gamma (t) \in \wall }}$.
\end{itemize}
We refer to $\mono_L$ as the {\bf slope} or {\bf exponent vector} of $\gamma$ at $L $ and set 
\begin{itemize}
    \item $I(\gamma) = \mono$;
    \item $\text{Mono} (\gamma) = c(\gamma)z^{F(\gamma)}$
to be the monomial attached to the unique domain of linearity of $\gamma$ having $x_0 $ as an endpoint. 
\end{itemize}
\end{definition}

\begin{definition}\thlabel{def:theta}
Choose a point $x_0$ in the interior of $\mathcal{C}_\seed^+:=\{m\in M^{\circ}_{\R}\mid \langle e_i, m \rangle \geq 0 \text{ for all  } i \in \Iuf\}$ and let $\mono\in \cA^\vee_{\seed^\vee}(\Z^T)=M^\circ$. 
The {\bf theta function} on $\cA$ associated to $\mono$ is 
\begin{equation}
\label{eq:tf}
    \tf^{\cA}_{ \mono}:= \sum_{\gamma} \text{Mono} (\gamma),
\end{equation}
where the sum is over all broken lines $\gamma$ with $I(\gamma)=\mono$ and $\gamma(0)=x_0$. 
For $\mono = 0$ we define  $\vartheta^{\cA}_{0} =1$. We say $\tf^{\cA}_{ \mono}$ is {\bf polynomial} if the sum in \eqref{eq:tf} is finite.
\end{definition}

\begin{remark}
\label{rem:on_tf}
It is a nontrivial fact that $\tf^{\cA}_{\mono}$ is independent of the point $x_0\in \mathcal{C}_{\seed}^+$ we have chosen, see \cite[\S3]{GHKK}.
Moreover, in general $\tf^\cA_{\mono}$ can be an infinite sum and in order to think of $\tf^{\cA}_{\mono} $ as a function on a space one needs to work formally an consider a degeneration of $\cA$, see \cite[Proposition 3.4 and \S6]{GHKK} for the details.
However, in case $\tf^{\cA}_{\mono}$ is polynomial then $\tf^{\cA}_{\mono}\in H^0(\cA,\mathcal{O}_{\cA})$, that is, $\tf^{\cA}_{\mono}$ is an algebraic function on $\cA$. 
The definition of $\tf^{\cA}_{\mono}$ in \eqref{eq:tf} corresponds to the expression of such function written in the coordinates of the seed torus $\cA_{\seed}$.
\end{remark}

\subsubsection{Labeling by tropical points}

Recall that we are identifying $\cA^{\vee}_{\seed^\vee}(\R^T)$ and $ M^{\circ}_{\R}$. 
By construction, a theta function on $\cA$ is labeled by a point $\mono \in \cA^\vee_{\seed^\vee}(\Z^T)=  M^{\circ}$. 
By \cite[Proposition 3.6]{GHKK}, this labeling upgrades to a labeling by a point in $\cA^\vee(\Z^T)$. 
The main point being that if we let $m'=(\mu^{\cA^\vee}_{k})^T(m)\in \cA^{\vee}_{\mu_k(\seed^\vee)}(\Z^T)$ for $k \in \Iuf$ then $\tf^{\cA}_\mono$ and $\tf^{\cA}_{m'}$ correspond to the same function (see Remark \ref{rem:on_tf}) expressed, however, in different cluster coordinates.
This fact is of great importance for this paper so we would like to highlight it:

\begin{center}
    \emph{every theta function on $\cA$ is naturally labeled by a point of $\cA^\vee(\Z^T)$}.
\end{center}

In light of the discussion just above, from now on we label theta functions on $\cA$ either by elements of $\cA^{\vee}(\Z^T)$ or of $\cA^{\vee}_{\seed^\vee}(\Z^T)$. 
For sake of clarity, tropical points are denoted in bold font and as tuples. 
That is, ${\bf m}=(m_{\seed^\vee})_{\seed^\vee\in \orT}$ denotes an element of $\cA^\vee(\Z^T)$ and $m_{\seed^\vee}=\mathfrak{r}_{\seed^\vee}({\bf m})$. 
 With this notation we have the following identity
 \[
\tf^{\cA}_{\bf m}=\tf^{\cA}_{m_{\seed^\vee}}.
 \]

Even further, we can think of $\Supp(\scat^{\cA}_{\seed})$ as a subset of $\cA^{\vee}_{\seed^\vee}(\R^T) $. By \cite[Theroem 1.24]{GHKK} we have that for every $k \in \Iuf $, $\mu^{\cA^\vee}_{\seed^\vee,\mu_k(\seed^\vee)}\lrp{\Supp(\scat^{\cA}_{\seed})}=\Supp(\scat^{\cA}_{\mu_k(\seed)})$ and that $\scat^{\cA}_{\seed}$ and $\scat^{\cA}_{\mu_k(\seed)}$ are equivalent. Hence there is a well defined subset $\Supp(\scat^\cA) \subset \cA^{\vee}(\R^T)$ such that
\[
\mathfrak{r}_{\seed^\vee}\lrp{\Supp(\scat^\cA) }= \Supp(\scat^\cA_{\seed})
\]
for every $\seed\in \orT$. 
The point here is that $\Supp(\scat^\cA)$ is seed independent.
Similarly, the map $\mu^{\cA^\vee}_{\seed^\vee,\mu_k(\seed^\vee)} $ determines a bijection between the set of broken lines for $\scat^{\cA}_{\seed}$ and the set of broken lines for $\scat^{\cA}_{\mu_k(\seed)}$ (see \cite[Proposition 3.6]{GHKK}). 
In particular, supports of broken lines make sense in $\cA^\vee(\R^T)$.

\begin{remark}
    It is possible to upgrade $\Supp(\scat^\cA)$ to a scattering diagram inside $\cA^{\vee}(\R^T)$. In this generality scattering functions are described using log Gromov--Witten invariants. 
    See \cite{KY19} for details.
\end{remark}

\subsubsection{The middle cluster algebra}

Let us recall now that broken lines also encode the multiplication of theta functions. 
That is, given a product of arbitrary theta functions $\tf^{\cA}_p \tf^{\cA}_q$ with $p,q \in \cA^\vee_{\seed^\vee}(\Z^T)$,
we can use broken lines to express the structure constants $\alpha\lrp{p,q,r}$ in the expansion 
\begin{align} \label{eq:product}
    \vartheta^{\cA}_p \vartheta^{\cA}_q = \sum_{r\in \cA^\vee_{\seed^\vee}(\Z^T)} \alpha(p,q,r) \vartheta^{\cA}_r.
\end{align}
We review the construction here.
First, pick a general endpoint $z$ near $r$.
Then define (\cite[Definition-Lemma~6.2]{GHKK})
\eq{
    \alpha_z (p, q, r) := \sum_{\substack{\lrp{\gamma^{(1)}, \gamma^{(2)}} \\ I(\gamma^{(1)})= p,\ I(\gamma^{(2)})= q\\ \gamma^{(1)}(0) = \gamma^{(2)}(0) = z\\
    F(\gamma^{(1)}) + F(\gamma^{(2)}) = r   }}   c(\gamma^{(1)})\ c(\gamma^{(2)}), }{eq:multibrokenline}
where the sum is over all pairs of broken lines $\lrp{\gamma^{(1)}, \gamma^{(2)}}$ ending at $z$ with initial slopes $I(\gamma^{(1)}) = p$, $I(\gamma^{(2)}) = q$ and final slopes satisfying $F(\gamma^{(1)})+F(\gamma^{(2)}) =r$.
Gross--Hacking--Keel--Kontsevich show that for $z$ sufficiently close to $r$, $\alpha_z (p, q, r)$ is independent of $z$ and gives the structure constant $\alpha (p, q, r)$ (see \cite[Proposition~6.4]{GHKK}). 

\begin{definition}
\thlabel{def:cmid}
Let $\Theta(\cA):= \{ {\bf m} \in \cA^{\vee}(\Z^T) \mid \vartheta^{\cA}_{{\bf m}} \text{ is polynomial}\}$. The {\bf middle cluster algebra} $\text{mid}(\cA)$ is the $\Bbbk$-algebra whose underlying vector space is $\{ \tf^{\cA}_{{\bf m}} \mid {\bf m} \in \Theta(\cA) \}$, the multiplication of the basis elements is given by \eqref{eq:product} and extended linearly to all $\cmid(\cA)$.
\end{definition}

\subsubsection{Theta functions on $\cAp$}
The data $\Gamma_{\prin}$ is of full-rank. Therefore, this case is a particular case of \S\ref{sec:tf_A}. So we can talk about scattering diagrams, broken lines and theta functions for $\cAp $. The following result follows from Theorem \ref{thm:consistent_scattering_diagrams} and the definition of theta functions.

\begin{lemma}
Fix a seed $\widetilde{\seed}$ for $\cAp$ and express theta functions on the cluster coordinates determined by $\widetilde{\seed}$.
For $(m,n)\in \mathfrak{r}_{\widetilde{\seed}}(\Theta(\cAp))$ we have that $\tf^{\cAp}_{(m,n)}=\tf^{\cAp}_{(m,0)}\tf^{\cAp}_{(0,n)}$ and $\tf^{\cAp}_{(0,n)}$ is the Laurent monomial on the coefficients given by $n$.
\end{lemma}

Note that, for $(m_1,n_1),(m_2,n_2) \in M^{\circ}_{\rm prin}$, in general we have that $\tf^{\cAp}_{(m_1+m_2,n_1+n_2)} \neq \tf^{\cAp}_{(m_1,n_1)} \tf^{\cAp}_{(m_2,n_2)}$. 
The above lemma holds because the decomposition is only separating the unfrozen and frozen parts (\cf \thref{g_is_val} below). 

\begin{remark}
\label{rem:all_from_cAp}
Scattering diagrams for $\cAp $ can be used to define scattering diagrams, broken lines therein and theta functions on a variety $\cV $ of form $\cA$ (even if $\Gamma$ is not of full-rank), $\cX$, $\cA/T_{H}$ and $\cX_{{\bf 1}}$. 
Further, in each one of these cases we can define the associated middle cluster algebra $\cmid(\cV)$ and the set $\Theta(\cV)$ parametrizing its theta basis. 
In the following subsections we explain the cases of $\cA/T_{H}$, $\cX$, and $\cX_{{\bf 1}}$ individually. 
We do not treat the case of $ \cA $ for $\Gamma$ when $\Gamma$ is not of full-rank since the results of \S\ref{sec:cluster_valuations} do not apply to this case.
\end{remark}

\subsubsection{Theta functions on $\cA/T_{H}$}
\label{tf_quotient}
Suppose that $\Gamma $ is of full-rank (\cf Remark \ref{rem:all_from_cAp}). Let $H \subset K^\circ$ be a saturated sublattice and consider the quotient $\cA/T_{H}$ and the fibration $w_H: \cAm \to H^*$ (see the end of \S\ref{sec:FG_dual}).
The next result shows that theta functions on $\cA$ have a well defined $T_{H}$-weight.

\begin{proposition}
\thlabel{prop:dual_fibration}
Every polynomial theta function on $\cA$ is an eigenfunction with respect to the $T_{H}$-action. For every ${\bf q}\in \Theta(\cA)$ the $T_H$-weight of $\tf^{\cA}_{\bf q}$ is the image of ${\bf q} \in \cA^{\vee}(\Z^T)$ under the tropicalized map $w^{T}_{H}:\cAm(\Z^T) \to H^*$. Under the isomorphism $ H^* \cong M^\circ/H^\perp$ and in the lattice identification of $ \cA^\vee_{\seed^\vee}(\Z^T)$ of $\cA^\vee (\Z^T)$ the map $w^T_{H}$ is given by
\begin{align*}
   w^{T}_{H} : \ & \cA^{\vee}_{\seed^\vee}(\Z^T) \to M^\circ/H^{\perp},\\
     & \ \ \  \ \ \  q \longmapsto q + H^{\perp}.
\end{align*}
\end{proposition}
The claims are essentially contained in the literature already
(see for instance \cite[Proposition 7.7]{GHKK}). 
The differences are that we are acting by a potentially smaller torus (Gross--Hacking--Keel--Kontsevich act by $T_{K^\circ}$ rather than $T_{H}$) and, regarding the map $w_{H}: \cA^\vee \to T_{H^*}$, 
we are including $\Bbbk[H]$ into $\Bbbk[N^\vee]=\Bbbk[N^\circ]$ rather than including $\Bbbk[K^\circ]$ into $\Bbbk[N^\circ]$. 
For the convenience of the reader we give a proof of the statement.

\begin{proof}[Proof of Proposition~\ref{prop:dual_fibration}]
By \cite[Theorem~1.13]{GHKK} all scattering functions may be taken to be of the form $\lrp{1+z^{p^*(n)}}^c$ for some $n \in \Nuf$ and some positive integer $c$.\footnote{That is, the equivalence class of consistent scattering diagrams for $\cA$ contains a representative whose scattering functions are of this form.}
For $q\in \Theta(\cA)_{\seed^\vee}$ we have that $\tf^{\cA}_q$ is as a Laurent polynomial in $\Bbbk[M^\circ]$.
All monomial summands of $\tf^{\cA}_{q}$ have the form $c_m z^{q + m}$ for some $m \in p^*(\Nuf)$ and $c_m \in \Z_{>0}$.
The $T_{H}$-weight of this monomial is obtained from the map
\eqn{
T_{H} \to T_{\Z}=\Bbbk^*\quad \text{given by } \quad z^{h} \mapsto z^{\lra{q+m , h}} \quad \text{for $h \in {H}$}.
}
Since $H \subset p^*\lrp{\Nuf}^\perp$ we have that
$z^{\lra{q+m , h}} = z^{\lra{q, h}}$. 
That is, the $T_{H}$-weight of each monomial $z^{m'}$, $m'\in M^\circ$, is the character of $T_{H}$ given by $m' + H^\perp  \in M^\circ/H^\perp \cong H^*$. 
Moreover, all monomial summands of $\tf^{\cA}_q$ have the $T_{H}$-weight $q + H^\perp \in H^*$.
Next, the piecewise linear map $(\mu_k^{\cA^\vee})^T:M^\circ_\seed \to M^\circ_{\mu_k(\seed)}$ sends $m$ to $m+m'$ for some $m'\in p^*(\Nuf)$.
So, the choice of torus does not affect the $T_{H}$-weight.
Therefore, $\tf^{\cA}_q$ is an eigenfunction whose weight is $q + H^\perp$.
Furthermore, the projection 
\eqn{M^\circ &\to M^\circ/H^\perp\quad \text{given by} \quad q \mapsto q + H^\perp } 
dualizes the inclusion $H \hookrightarrow N^\circ$.
So, restricting to seed tori, this is precisely the tropicalization of the map ${T_{M^\circ} \rightarrow T_{H^*}}$ whose pullback is the inclusion $H \hookrightarrow N^\circ$.
Since $p^*$ commutes with mutation,
we see that the $T_{H}$-weight of $\tf^{\cA}_{\bf q}$ is the image of ${\bf q}$ under the tropicalization of 
$ w_{H}: \cA^\vee \rightarrow T_{H^*}$.
\end{proof}

Every weight $0$ eigenfunction on $ \cA$ induces a well defined function on $\cA/T_{H}$. 
So in order to construct a scattering-diagram-like structure $\scat^{\cA/T_H}$ defining theta functions on $\cA/T_{H}$ we consider the {\bf weight zero slice} inside $\cA^{\vee}(\R^T)$ defined as 
$(w^T_H)^{-1}(0)$. 
Observe that identifying $ \cA^\vee$ with $M^\vee$ via a choice of seed, then $(w^T_H)^{-1}(0)$ corresponds to $H^{\perp}_{\R}$.
With this in mind, we define $\supp(\scat^{\cA/T_{H}})$ as
\[
\supp(\scat^{\cA/T_{H}}):=\supp (\scat^{\cA})\cap (w^T_H)^{-1}(0).
\]
The scattering functions attached to the walls of $\mathfrak{r}_{\seed^\vee}(\supp(\scat^{\cA/T_{H}}))$ are the same as the corresponding functions attached to the walls of $\scat^{\cA}_\seed$. 
This gives rise to a scattering diagram $\scat^{\cA/T_{H}}_{\seed}$ inside $(\cA/T_{H})^\vee_{\seed^\vee}(\R^T)$ for every $\seed \in \orT$.
The broken lines for $\scat^{\cA/T_{H}}_\seed$ are the broken lines for $\scat^{\cA}_\seed $ entirely contained in $\mathfrak{r}_{\seed^\vee}(w^{-1}_{H}(0))$.

In order to label a theta function on $\cA/T_{H}$ with an element of $(\cA/T_{H})^{\vee}(\Z^T)$ it suffices to consider a bijection $
(\cA/T_{H})^{\vee}(\R^T) \overset{\sim}{\longrightarrow} (w_H^T)^{-1}(0)$.
Such a bijection can be obtained tropicalizing the inclusion $\mathfrak{i}_H:(\cA/T_{H})^{\vee} \hookrightarrow \cA^\vee $. 
Indeed, in lattice identifications of the tropical spaces given by a seed $\seed$, the map
$ \mathfrak{i}_H^T:(\cA/T_{H})^{\vee}_{\seed}(\Z^T)\hookrightarrow \cA^\vee_{\seed}(\Z^T)$ correspond to the inclusion $H^\perp \hookrightarrow M^\vee$ and $w^{-1}_{H}(0)(\Z) $ corresponds to $H^\perp$.

In particular, we obtain (as one should have expected) that the theta functions on $\cA/T_{H}$ are precisely the functions on $\cA/T_{H}$ induced by the $T_H$-weight zero theta functions on $\cA$. 
So we let
$\Theta(\cA/T_H)\subset (\cA/T_H)^{\vee}(\Z^T)$ be the preimage of $\Theta(\cA)\cap (w_H^T)^{-1}(0)$ under $\mathfrak{i}^T_H$ and define the middle cluster algebra $\cmid(\cA/T_H)$ as in the case of $\cA$ (see \thref{def:cmid}).
In particular, for ${\bf m}\in \Theta (\cA/T_H)$ the theta function $\tf^{\cA/T_H}_{\bf m}$ is the function on $\cA/T_H$ induced by $ \tf^{\cA}_{\mathfrak{i}^T_H({\bf m})}$. So,
\begin{center}
    \emph{every theta function on $\cA/T_H$ is naturally labeled by a point of $(\cA/T_H)^\vee(\Z^T)$}.
\end{center}

\subsubsection{Theta functions on $\cX$}
\label{sec:tf_X}
Recall from \S\ref{sec:principal_coefficients} that there is an isomorphism $\chi: \cAp/T_{H_{\cAp}}\to \cX$, where \[
H_{\cAp}=\lrc{\lrp{n,-(p^*)^*(n)}\in N^\circ_\prin \mid n \in N^\circ} \subset K^{\circ}_{\prin}. 
\]
Hence, the construction of theta functions on $\cX$ is already covered in the previous subsection. 
However, there is a very subtle difference created by treating $ \cAp/T_{H_{\cAp}}$ as a cluster $\cX$-variety as opposed to a quotient of $\cAp$: 
\begin{center}
\emph{every theta function on $ \cX$ is naturally labeled by a point of $\cX^\vee(\Z^t)$ as opposed to $\cX^\vee(\Z^T)$.}
\end{center} 
If we would proceed as in the previous subsection we would label theta functions on $ \cAp/T_{H_{\cAp}}$ by points of $(\cAp/T_{H_{\cAp}})^\vee(\R^T)$.
The origin of the difference is made explicit by the following lemma.

\begin{lemma}
\label{lem:right_tropical_space}
There is a canonical bijection between $\cX^{\vee}(\R^t)$ and $\lrp{w^T_{H_{\cAp}}}^{-1}(0)\subset \cAp^\vee(\R^T)$.
\end{lemma}
\begin{proof}
One can verify directly that the composition $\xi^T_{\Gamma^\vee}\circ i$ gives rise to the desired bijection, where \[
i:\cX^\vee(\R^t) \to \cX^\vee(\R^T)
\]
is the bijection discussed in \S\ref{ss:tropicalization} and \[
\xi_{\Gamma^\vee}^T:  \cX^\vee(\R^T) \to \cAp^{\vee}(\R^T) 
\]
is the tropicalization of the map $\xi_{\Gamma^\vee}:\cX^\vee=\cA_{\Gamma^\vee} \to  \cX_{(\Gamma^\vee)_{\prin}}\cong \cX_{(\Gamma_{\prin})^\vee}=\cAp^{\vee} $ described in \eqref{eq:def_xi}, see Remarks \ref{rem:labels} and \ref{rem:Lprin}. 
However, for the convenience of the reader we include computations that show in a rather explicit way the necessity to consider 
$\cX^\vee(\Z^t)$ as opposed to $\cX^\vee(\Z^T)$. For simplicity throughout this proof we denote $w_{H_{\cAp}}$ simply by $w$.

Pick a seed $\seed=(e_i)_{i \in I}\in \orT$ for $\Gamma$ and consider the seed $\seed^\vee$ for $\Gamma^\vee$. Denote by $\widetilde{\seed}^\vee$ the seed for $(\Gamma_{\prin})^\vee$ obtained mutating $ \seed_{0_{\prin}}$ in the same sequence of directions needed to obtain $\seed$ from $\seed_0$. Then
\[
\lrp{(w^T)^{-1}(0)}_{\widetilde{\seed}^\vee}(\R^T)=H^\perp_{\cAp}=\{(p^*(n),n)\in M^\circ_{\prin,\R} \mid n\in N_{\R}\}\subset M^\circ_{\prin, \R}=M^\circ_{\R}\oplus N_{\R}
\]
(see \eqref{eq:identification} to recall the meaning of $\lrp{(w^T)^{-1}(0)}_{\widetilde{\seed}^\vee}(\R^T)$). We now verify that for every $k\in \Iuf$ there is a commutative diagram
\[
\xymatrix{
\lrp{(w^T)^{-1}(0)}_{\widetilde{\seed}^{\vee}}(\R^T) \ar^{\lrp{\mu^{\cAp^\vee}_{k}}^T}[rr]  \ar_{\pi^{\cX^\vee}_1}[d] & & \lrp{(w^T)^{-1}  (0)}_{\mu_k(\widetilde{\seed}^{\vee})}(\R^T) \ar^{\pi^{\cX^\vee}_2}[d]
\\
\cX^{\vee}_{\seed^{\vee}}(\R^t)\ar^{\lrp{\mu^{\cX^\vee}_{k}}^t}[rr] & & \cX^{\vee}_{\mu_k(\seed^{\vee})}(\R^t),
}
\]
where the vertical maps $\pi^{\cX^\vee}_1$ and $\pi^{\cX^\vee}_2$ are both given by $(p^*(n),n)\mapsto dn$ (recall that $\cX^{\vee}_{\seed^{\vee}}(\R^t)=(N^\vee)^\circ_{\R}= (d\cdot N)_{\R} =\cX^{\vee}_{\mu_k(\seed^{\vee})} (\R^t)$). By definition we have that
\begin{eqnarray*}
    \lrp{\mu^{\cAp^\vee}_{k}}^T(p^*(n),n)& \overset{\eqref{eq:tropical_X_mutation}}{=} & (p^*(n),n)+[\langle (d e_k,0),(p^*(n),n)\rangle]_+\{d_ke_k, \cdot \}^{\vee}_{\prin} \\
    &=&  (p^*(n),n)+ [p^*(n)(de_k)]_+(\{d_ke_k, \cdot\}^\vee,d_ke_k)\\
    &=&(p^*(n) + [\{n, de_k\}]_+\{d_ke_k, \cdot\}^\vee,n+ [\{n, de_k\}]_+d_ke_k).
\end{eqnarray*}
Using the facts that $d, d_k>0$ and 
 that $d\max(a,b)=\max(da,db)$ and $\max(a,b)=-\min(-a,-b)$ for all $a,b \in \R$, we compute that
\begin{eqnarray*}
\pi^{\cX^\vee}_2\lrp{\lrp{\mu^{\cAp^\vee}_{k}}^T(p^*(n),n)} &=&  dn+ d[\{n, de_k\}^\vee]_+d_ke_k\\
&=&  dn+ [\{dn, de_k\}^\vee]_+d_ke_k\\
 &=&  dn+[-\{de_k,dn\}^{\vee}]_+d_ke_k\\
 &=&  dn+[-\{d_ke_k,dn\}^{\vee}]_+de_k\\
 &=&  dn-[\{d_ke_k,dn\}^{\vee}]_-de_k\\
  &=&  dn-[\langle v_k^\vee,dn\rangle]_-de_k\\
&=& dn+[\langle v_k^\vee,dn\rangle]_-(-d_k^\vee e^\vee_k)\\
&\overset{\eqref{eq:tropical_A_mutation}}{=}& \lrp{\mu^{\cX^\vee}_{k}}^t (dn)\\
&=& \lrp{\mu^{\cX^\vee}_{k}}^t \lrp{\pi^{\cX^\vee}_1(p^*(n),n)}.
\end{eqnarray*}
This gives the commutativity of the diagram. 
Notice moreover that $\pi^{\cX^\vee}_1$ and $\pi^{\cX^\vee}_2$ are canonical bijections. 
These two facts together imply that we have a well defined bijection
\[
\pi^{\cX^\vee}: (w^T)^{-1}(0)(\R^T) \overset{\sim}{\longrightarrow} \cX^{\vee}(\R^t).
\]
The fact that $\xi^T_{\Gamma^\vee}\circ i$ is the inverse of $\pi^{\cX^\vee} $ follows from noticing that, in lattice identifications of the domain and codomian of $\xi^T_{\Gamma^\vee}$ given by a choice of seed, we have that
\[
\xi^T_{\Gamma^\vee}(dn)=(-p^*(n),-n).
\]
\end{proof}
We can now define cluster scattering diagrams for $ \cX$ using cluster scattering diagrams for $\cAp$ and the quotient map $\tilde{p}:\cAp \to \cX$ described in \eqref{eq:def_tilde_p} and the content of Lemma \ref{lem:right_tropical_space}.
We define $\supp(\scat^{\cX})$ as 
\[
\supp(\scat^{\cX}):=\pi^{\cX^\vee}\lrp{\supp (\scat^{\cAp})\cap (w^T_H)^{-1}(0)}\subset \cX^\vee(\Z^t).
\]
By definition the support of the scattering diagram $ \scat^{\cX}_{\seed}$ is $\mathfrak{r}_{\seed^\vee}\lrp{\supp(\scat^{\cX})}$.
The scattering functions attached to the walls of $\supp(\scat^{\cX}_\seed)$ are obtained by applying $\tilde{p}^*$ to the scattering functions of the corresponding walls of $\scat^{\cAp}_\seed$. 
We proceed in an analogous way to define broken lines for $\scat^{\cX}_\seed$.
As in the previous cases, supports of broken lines are well defined inside $\cX^\vee(\Z^t)$.

The labeling of a theta function on $\cX$ with an element of $\cX^{\vee}(\Z^t)$ is obtained using the bijection of Lemma \ref{lem:right_tropical_space}. More precisely,
for ${\bf n} \in \cX^\vee(\Z^t)$ with ${\bf n}\in \Theta(\cX)$ we have
\[
\tilde{p}^*(\tf^\cX_{\bf n})=\tf^{\cAp}_{\xi^T_{\Gamma^\vee}\circ i({\bf n})}.
\]
Explicitly, in lattice identifications of the tropical spaces, we have that for $dn \in \cX^\vee_{\seed^\vee}(\Z^t) $
\[
\tilde{p}^*\lrp{\tf^{\cX}_{dn}}:= 
\tf^{\cAp}_{(p^*(n),n)}.
\]

\begin{example}
\label{running_example_1}
Let $\epsilon
=
\lrp{\begin{matrix}
0 & 2 \\
-1 & 0
\end{matrix}}$
and $d_1=1, d_2=2$. Using the above parametrization we compute
\[
\tf^{\cX}_{2(-1,-2)}=X_1^{-1}X_2^{-2}+2X_1^{-1}X_2^{-1}+X_1^{-1}.
\]
Indeed, we have that $\xi^T_{\Gamma^\vee}\circ i(2(-1,-2))=(2,-2)$ and
\[
\tf^{\cAp}_{(2,-2),(-1,-2)}= \lrp{\tf^{\cAp}_{(1,-1),(0,0)}}^2 \tf^{\cAp}_{(0,0),(-1,-2)} = \lrp{\dfrac{A_1+t_2}{A_2}}^2t_1^{-1}t_2^{-2}= \tilde{p}^*(X_1^{-1}X_2^{-2}+2X_1^{-1}X_2^{-1}+X_1^{-1}).
\]
\end{example}

\subsubsection{Theta functions on $\cXe$}
\label{tf_fibre}
As in the previous subsections we would like to highlight that 
\begin{center}
\emph{every theta function on $ \cXe$ is naturally labeled by a point of $(\cXe)^\vee(\Z^t)$}
\end{center} 
as we now explain.
The tropical space $ (\cXe)^{\vee}(\R^t)  $ is the quotient of $ \cX^{\vee} (\R^t)$ by the tropicalization of the action of $T_H$ on $\cX^{\vee}$. 
In other words, since the variety $(\cXe)^{\vee}$ is a quotient of $\cX^\vee$, we can consider the quotient map by $\varpi_{H}: \cX^\vee \to (\cXe)^\vee$
to obtain a surjection
\[
\varpi_H^t:  \cX^{\vee}(\R^t) \to (\cXe)^{\vee} (\R^t).
\]
Then, given $\overline{\bf n}\in (\cXe)^{\vee} (\R^t)$ and ${\bf n}\in (\varpi_H^t)^{-1}(\overline{\bf n})$ we define 
\[
\tf^{\cXe}_{\overline{\bf n}}=\tf^{\cX}_{\bf n}|_{\cXe}.
\]
More concretely, working in lattice identifications of the tropical spaces, we have that $\cX^{\vee}(\R^t)_{\seed^\vee} = N_\R$ and $ (\cXe)^{\vee}_{\seed^\vee}(\R^t) {\cong}N_\R/H_{\R}$. 
Then for every $n \in N$
\[
\tf^{\cXe}_{d n + H}=\tf^{\cX}_{dn}|_{\cXe}.
\]
One can proceed in an analogous way as in the previous cases to construct a scattering diagram like structure $\scat^{\cXe}_{\seed}$ inside $(\cXe)^\vee_{\seed}(\Z^t)$. In turn we obtain a description of  $ \tf^{\cXe}_{\overline{\bf n}}$ using broken lines and use these to define $\cmid (\cXe)$ and $\Theta(\cXe)$.

\subsubsection{The full Fock--Goncharov conjecture}
\label{sec:FG_conj}

Let $\cV$ be a scheme of the form $ \cA$, $\cX$, $\cA/T_{H}$ or $\cX_{{\bf 1}}$. 
The {\bf upper cluster algebra} of $\cV$ is defined as 
\[
\text{up}(\cV):=H^0(\cV,\mathcal{O}_{\cV}).
\]
Every polynomial theta function on $\cV$ belongs to $\text{up}(\cV)$, therefore, we have a natural $\Bbbk$-linear map $\cmid(\cV)\to \text{up}(\cV)$.
If $\cV$ is one of $\cA$ (see Remark \ref{rem:full_rank_assumption}) or $\cX$ it was proved in \cite[Theorem 7.5, Corollary 7.13, Theorem 7.16]{GHKK} that this map is in fact an injective homomorphism of algebras.
These cases already imply that the same is true is $\cV$ if of the form $\cA/T_H$ or $\cXe$.

\begin{remark}
\label{rem:integral_domain}
    If $\cV= \cA$, $\cX$, $\cA/T_{H}$ or $\cX_{{\bf 1}}$ then $\cmid(\cV)$ is an integral domain. Indeed, $\cmid(\cV)$ is a subalgebra of $ \up(\cV)=H^0(\cV,\mathcal{O}_{\cV})$ which is a domain as $\cV$ is irreducible.
\end{remark}

As we have seen in the previous subsections theta functions on varieties of the form $ \cA$ or $\cA/T_H$ are naturally labeled by the $\Z^T$-points of its Fock--Goncharov dual, whereas theta functions on varieties of the form $ \cX$ or $\cXe$ are naturally labeled by the $\Z^t$-points of its Fock--Goncharov dual. 
Since we would like to consider all these cases simultaneously we introduce the following notation. For $G= \Z, \Q$ or $\R$ we set

\begin{equation}
\label{eq:unif}
    \Trop_G(\cV):=
    \begin{cases}
        \cV(G^t) &\text{ if } \cV=\cA \text{ or } \cV=\cA/T_H\vspace{1mm}\\
        \cV(G^T) \ & \text{ if } \cV=\cX \text{ or } \cV=\cXe.
    \end{cases}
\end{equation}

Similarly, for a positive rational function $g: \cV \dashrightarrow \Bbbk $  we let
\begin{equation}
\label{eq:unif_function}
    \Trop_G(g):=
    \begin{cases}
        g^t &\text{ if } \cV=\cA \text{ or } \cV=\cA/T_H\vspace{1mm}\\
        g^T \ &\text{ if } \cV=\cX \text{ or } \cV=\cXe.
    \end{cases}
\end{equation}

In particular, if we think of the seed torus $\cV_\seed$ as a cluster variety with only frozen directions then $\Trop_G(\cV_\seed)=\mathfrak{r}_{\seed}(\Trop_G(\cV))=\cV_{\seed}(G^t)$, if $\cV$ is of the form $\cA$ or $\cA/T_H$ and $\Trop_G(\cV_\seed)=\mathfrak{r}_{\seed}(\Trop_G(\cV))=\cV_{\seed}(G^T)$, if $\cV$ is of the form $\cX$ or $\cXe$. For later use we also set
\begin{equation}
    \label{eq:Theta_seed}
\Theta(\cV)_{\seed^\vee}:=\mathfrak{r}_{\seed^\vee}(\Theta(\cV))\subset \Trop_{\Z}(\cV^\vee),
\end{equation}
see the line just below equation \eqref{eq:identification}. 
Following \cite{GHKK} we introduce the following definition.
\begin{definition}
\label{def:full_FG}
Let $\cV$ be a scheme of the form $ \cA$, $\cX$, $\cA/T_{H}$ or $\cX_{{\bf 1}}$. We  say that {\bf the full Fock--Goncharov conjecture} holds for $\cV$ if
\begin{itemize}
    \item $\Theta(\cV)=\Trop_{\Z}(\cV^{\vee})$, and
    \item the natural map $\cmid(\cV) \to \text{up}(\cV)$ is an isomorphism.
\end{itemize}
\end{definition}

\section{Bases of theta functions for partial minimal models}
\label{sec:minimal_models}

In \cite{GHKK}, the authors obtained nearly optimal conditions ensuring that the full Fock--Goncharov conjecture holds for a cluster variety. 
However, they were able to prove that the ring of regular functions of a partial compactifications of a cluster varieties has a basis of theta functions under much stronger conditions. 
In this section we outline this framework, including quotients and fibres of cluster varieties, and refer to \cite[\S9]{GHKK} for a detailed treatment. 
The main class of (partial) compactifications we shall consider are the (partial) minimal models defined below.

\begin{definition}{\cite{GHK_birational}}
\label{def:cv_minimal_model}
Let $\cV$ be a scheme of the form $\cA, \cX, \cA/T_H$ or $\cXe$. An inclusion $\cV \subset Y$ as an open subscheme of a normal variety $Y$ is a {\bf partial minimal model} of $ \cV$ if the canonical volume form on $\cV$ has a simple pole along every irreducible divisor of $Y$ contained in $ Y \setminus \cV$. It is a {\bf minimal model} if $Y$ is, in addition, projective. We call $ Y \setminus \cV$ the {\bf boundary} of $\cV  \subset Y$.
\end{definition}

For example, if $\cV$ is a cluster $\cA$-variety with frozen variables we can let these variables vanish to obtain a partial minimal model of $\cV$ as in \cite[Construction B.9]{GHKK}. 
Similarly, if we consider a torus as a cluster variety (by letting $\Iuf = \emptyset$) then a partial minimal model is simply a normal toric variety.

Given a partial minimal model $\cV\subset Y$, where $\cV$ is a scheme of the form $\cA, \cX, \cA/T_H$ or $\cXe$, we would like to describe the set of theta functions on $\cV$ (resp. $\cV^\vee$) that extend to $Y$ in a similar way as the ring of algebraic functions on a normal toric variety is described in toric geometry using polyhedral fans. 
In order to be able to do so we need that the pair $(\cV, \cV^\vee)$ satisfies a technical condition --\emph{theta reciprocity}-- that we will introduce shortly. 
For this, we need to discuss first the \emph{tropical pairings} associated to the pair $(\cV,\cV^{\vee})$.
 
In order to define the tropical pairings we temporarily assume that $\cV$ is a variety of the form $\cA$ or $\cA/T_{H}$ so that $\cV^\vee$ is a cluster $\cX$-variety or a fibre of a cluster $\cX$-variety, respectively. 
In particular, $\Theta(\cV)\subset \cV^\vee(\Z^T)= \Trop_{\Z}(\cV^\vee)$ and $\Theta(\cV^\vee)\subset \cV(\Z^t)=\Trop_{\Z}(\cV)$, see \eqref{eq:unif}.  
Recall from Remark~\ref{rmk:geometric trop} that the set $\cV(\Z^t)$ (resp. $\cV^\vee(\Z^t)$) is canonically identified with the geometric tropicalization
$\cV^\trop(\Z)$ (resp. $(\cV^\vee)^\trop(\Z)$). 
Therefore, we systematically think of the elements of $\cV(\Z^t)$ (resp. $\cV^\vee(\Z^t)$) as divisorial discrete
valuations on $\Bbbk(\cV)$ (resp. $\Bbbk(\cV^\vee)$).
We also consider the bijection $i : \cV^\vee(\Z^T) \to \cV^\vee(\Z^t )$ introduced in \S\ref{ss:tropicalization} (see the comment bellow \eqref{eq:imap}).
The {\bf tropical pairings} associated to the pair $(\cV,\cV^\vee) $ are the functions $
    \langle \cdot , \cdot \rangle : \Theta(\cV^{\vee})  \times \Theta (\cV)  \to \Z   $ and $ \langle \cdot , \cdot \rangle^{\vee} : \Theta(\cV^{\vee})  \times \Theta (\cV)  \to \Z$ given by
\[
    \langle {\bf v} , {\bf b} \rangle = {\bf v}(\tf^{\cV}_{\bf b}) \ \ \ \ \ \ \ \text{and} \ \ \ \ \ \ \ \langle {\bf v} , {\bf b} \rangle^{\vee} = i({\bf b}) (\tf^{\cV^{\vee}}_{\bf v}),
\]

\begin{definition}
\label{def:theta_reciprocity}
 Let $\cV$ be a scheme of the form $\cA, \cX, \cA/T_H$ or $\cXe$. The pair $(\cV,\cV^\vee)$ has {\bf theta reciprocity} if $\Theta(\cV)=\Trop_{\Z}(\cV^\vee)$, $\Theta(\cV^{\vee})=\Trop_{\Z}(\cV)$, and $ \langle {\bf v} , {\bf b} \rangle = \langle {\bf v} , {\bf b} \rangle^{\vee} $ for all $({\bf v},{\bf b})\in \Trop_{\Z}(\cV) \times \Trop_{\Z}(\cV^\vee)$.
\end{definition}
\begin{remark}
    Definition \ref{def:theta_reciprocity} shall not be considered artificial. In fact, an analogous conjecture for affine log Calabi--Yau varieties with maximal boundary is expected to hold true, see \cite[Remark 9.11]{GHKK}.
\end{remark}

\begin{lemma}
\label{lem:tf_that_extend}
    Let $\cV$ be a scheme of the form $\cA, \cX, \cA/T_H$ or $\cXe$ and let $\cV\subset Y$ be a (partial) minimal model. Suppose that the pair $(\cV,\cV^\vee)$ has theta reciprocity.
    Then for every seed $\seed\in \orT$ the set of theta functions on $\cV$ that extend to $Y$ can be described as the intersection of $\Theta(\cV^\vee)_{\seed^\vee}$ (see \eqref{eq:Theta_seed}) with a polyhedral cone of the vector space $\Trop_{\R}(\cV^{\vee}_{\seed^\vee})$ (see the sentence bellow equation \eqref{eq:unif}).
\end{lemma}
\begin{proof}
We treat the cases $\cV= \cA$ or $\cA/T_H$ as the proof is completely analogous for the cases $\cV= \cX$ or $\cXe$.
Let $D_1, \dots, D_s$ be the irreducible divisors of $Y$ contained in the boundary of $\cV \subset Y $. 
Since $Y$ is normal, to describe the theta functions on $\cV$ that extend to $Y$ it is enough to describe the set of theta functions that extend to $D_1, \dots , D_s$ since $Y\setminus (\cV \cup D_1, \dots , D_s)$ has co-dimension greater or equal to $2$ in $Y$.
Let $\ord_{D_j}$ be the discrete valuation on $ \Bbbk(\cV)\setminus \{ 0 \}$ associated to the irreducible divisor $D_j$.
Since $\cV \subset Y$ is a partial minimal model, $\ord_{D_j}$ determines a point of $  \cV(\Z^t) $. Since $\Theta(\cV^{\vee})= \cV(\Z^t)$ we have $\ord_{D_j} \in \Theta (\cV^{\vee})$. Therefore, $
\tf^{\cV^{\vee}}_{\ord_{D_j}}$ is a polynomial theta function and its tropicalization is the function
\[
(\tf_{\ord_{D_j}}^{\cV^\vee})^t:\cV^{\vee}( \Z^t)\to \Z\quad \text{given by} \quad  v \mapsto v (\tf^{\cV^{\vee}}_{\ord_{D_j}}).
\]
In other words, $(\tf_{\ord_{D_j}}^{\cV^{\vee}})^t(v)=\langle \ord_{D_j}, i(v) \rangle$. 
Since $\Theta(\cV)= \cV^\vee(\Z^T)$ we have that $i(v)\in \Theta(\cV)$ and, therefore, $\tf^\cV_{i(v)}$ is a polynomial theta function.
The assumption $ \langle{\bf v} , {\bf b} \rangle = \langle {\bf v} , {\bf b} \rangle^{\vee} $ for all ${\bf v}$ and ${\bf b}$ implies that
\[
(\tf_{\ord_{D_j}}^{\cV^{\vee}})^t(v)= (\tf^{\cV}_{i(v)})^t(\ord_{D_j}),
\]
since
\[
(\tf_{\ord_{D_j}}^{\cV^{\vee}})^t(v) =
\langle \ord_{D_j}, i(v) \rangle =
\langle \ord_{D_j}, i(v)\rangle^{\vee} =
\ord_{D_j}(\tf^{\cV}_{i(v)}) =
(\tf^{\cV}_{i(v)})^t(\ord_{D_j}).
\]
Thus a theta function $\tf^{\cV}_{i(v)} \in \cmid(\cV)$ extends to $D_j$ if and only if $0\leq (\tf^{\cV^{\vee}}_{\ord_{D_j}})^t(v)$. 
In particular, a theta function $\tf^\cV_{i(v)}$ extends to $Y$ if and only if
\[
i(v)\in \bigcap_{i=1}^s\{b\in\cV^{\vee}_{\seed^\vee}(\R^T)\mid 0\leq (\tf^{\cV^{\vee}}_{\ord_{D_j}})^T(b)\}
\]
since $g^T(b)=g^t(i(b))$ for every positive function $g$ on $\cV$, see \eqref{eq:comparing_tropicalizations}. By definition of tropicalization, the set  $\bigcap_{i=1}^s\{b\in\cV^{\vee}_{\seed^\vee}(\R^T)\mid 0\leq (\tf^{\cV^{\vee}}_{\ord_{D_j}})^T(b)\}$ is a polyhedral cone of $\cV^{\vee}_{\seed^\vee}(\R^T)=\Trop_{\R}(\cV^{\vee}_{\seed^\vee})$. 
\end{proof}

We now turn to the problem of understanding when the theta functions on $\cV$ that extend to a (partial) minimal model $\cV \subset Y$ form a basis of $H^0(Y, \mathcal{O}_Y)$. 
The following notion is central.
\begin{definition}
    \label{def:respect_order}
    Let $\cV$ be a scheme of the form $\cA, \cX, \cA/T_H$ or $\cXe$. We say that the theta functions on $\cV$ {\bf respect the order of vanishing} if
    for all ${\bf v}\in \cV(\Z^t)$ and $\displaystyle \sum_{{\bf q}\in \Theta(\cV)} \alpha_{\bf q} \tf^{\cV}_{\bf q}\in \cmid(\cV)$ then 
\[
{\bf v}\lrp{\sum_{{\bf q}\in \Theta(\cV)} \alpha_{\bf q} \tf^{\cV}_{\bf q}} \geq 0 \ \ \text{ if and only if }\ \  {\bf v}(\tf_{\bf q})\geq 0  \text{ for all } {\bf q} \text{ such that } \alpha_{\bf q}\neq 0.
\]
\end{definition}

Notice that in \cite[Conjecture 9.8]{GHKK} the authors conjecture that the theta functions on $\cAp$ respect the order of vanishing.
The {\bf superpotential} associated to a partial minimal model $\cV \subset Y $ is the function on $\cV^\vee$ defined as
\begin{equation}\label{eq:def superpotential}
    W_{Y}:=\sum_{j=1}^n \tf^{\cV^{\vee}}_{j},
\end{equation}
where 
\begin{equation}
\label{eq:def superpotential_summands}
   \tf^{\cV^{\vee}}_{j}=\begin{cases}
       \tf^{\cV^{\vee}}_{\ord_{D_j}} &\text{ if } \cV=\cA \text{ or } \cV=\cA/T_H\vspace{1mm}\\
       \tf^{\cV^{\vee}}_{i(\ord_{D_j})} \ &\text{ if } \cV=\cX \text{ or } \cV=\cXe.
   \end{cases}
\end{equation}
The {\bf superpotential cone} associated to $W_Y$ is
\begin{equation}\label{eq:def Xi}
    \Xi_Y:= \{ {\bf v} \in \Trop_{\R}(\cV^\vee) \mid \Trop_{\R}(W_Y)({\bf v})\geq0 \},
\end{equation} 
see equation \eqref{eq:unif_function}.

We further set $\Xi_{Y;\seed^\vee}:= \mathfrak{r}_{\seed^\vee}(\Xi_Y)\subset \Trop_{\R}(\cV^{\vee}_{\seed^\vee})$.
Notice that if the theta functions on $\cV$ respect the order of vanishing then $\Xi_{Y;\seed}$ is precisely the polyhedral subset of Lemma \ref{lem:tf_that_extend}.
The next results follows at once from the definitions.

\begin{lemma}
\label{lem:basis_for_pmm}
    Let $\cV$ be a scheme of the form $\cA, \cX, \cA/T_H$ or $\cXe$ and let $\cV\subset Y$ be a (partial) minimal model. Suppose that the full Fock--Goncharov conjecture holds for $\cV$, that the pair $(\cV, \cV^\vee) $ has theta reciprocity and that the theta functions on $\cV$ respect the order of vanishing. 
    Then the set of theta functions on $\cV$ parametrized by the points of $\Xi_Y(\Z)$ is a basis of $H^0(Y, \mathcal{O}(Y))$.
\end{lemma}

\begin{lemma}
Suppose there is a cluster ensemble map $p:\cA \to \cX$ that is an isomorphism. Then theta functions on $\cA$ respect the order of vanishing if and only theta functions on $\cX$ respect the order of vanishing.
\end{lemma}
\begin{proof}
    The result follows at once from the fact that $p^*(\tf^{\cX}_{\bf n})= \tf^{\cA}_{(p^{\vee})^T\circ i ({\bf n})}$. 
\end{proof}

We propose the following definition that allows to have the benefits of Lemma \ref{lem:basis_for_pmm} without having to verify all its assumptions. 
We apply this in \S\ref{sec:NO_Grass}.

\begin{definition}
\label{def:enough_tf}
We say that $ \cV \subset Y$ has {\bf enough theta functions} if the full Fock--Goncharov conjecture holds for $\cV$ and the theta functions on $\cV$ parametrized  by $\Xi_{Y} (\Z)$ form a basis of $H^0(Y, \mathcal{O}_Y)$.
\end{definition}

We now recall an important notion introduced in \cite[Definition 9.1]{GHKK} that can be used to verify in a combinatorial way that a partial minimal model $\cA\subset Y$ has enough theta functions provided $Y$ is obtained by letting the frozen variables vanish.

\begin{definition}
We say that a seed $\seed=(e_i)_{ i \in I}$ is {\bf optimized} for a point $ {\bf n} \in \cA(\Z^t) $ if under the identification of $\cA(\Z^t)$ with $N^\circ$ afforded by $\seed$ we have that $\{ e_k, n_{\seed} \}\geq 0 $ for all $k \in \Iuf$.
\end{definition}

\begin{lemma}
\thlabel{lemm:enough_tf}
Assume that $\cA$ satisfies the full Fock--Goncharov conjecture. Let $\cA\subset Y$ be a partial minimal model of $\cA$ and let $D_1, \dots , D_s$ be the irreducible divisors of $Y$ contained in $Y\setminus \cA$.  Assume that $p^*_2|_{N^{\circ}}: N^{\circ}\to \Nuf^*$ is surjective and that the point $\ord_{D_j}\in \cA^{\vee}(\Z^t)$ has an optimized seed for every $1 \leq j \leq s$. Then the partial minimal model $\cA \subset Y$ has enough theta functions.
\end{lemma}
\begin{proof}
Since $p^*_2|_{N^{\circ}}$ is surjective we have that $\cAp$ is isomorphic to $\cA\times T_M $ (see \cite[Lemma B.7]{GHKK}). 
Consider the partial compactification $\cAp \subset Y \times T_M$. Its boundary is isomorphic to $D\times T_M$ and the irreducible components of the boundary are the divisors $\widetilde{D}_1, \dots, \widetilde{D}_s$, where $\widetilde{D}_j:=D_j \times T_M $. 
By hypothesis $\ord_{D_j} $ is optimized for some seed $\seed_j$. 
Let $\widetilde{\seed}_j$ be the seed for $\Gamma_{\prin}$ obtained mutating $ \seed_{0_{\prin}}$ in the same sequence of directions needed to obtain $\seed_j$ from $\seed_0$.
Observe that for every $1\leq j \leq s$, under the identifications 
\[
\cA_{\prin,\widetilde{\seed}_j}(\Z^t) = N_\prin^{\circ} = \cA_{\seed_j}(\Z^t) \oplus  T_M(\Z^t), 
\]
the point $ \ord_{\widetilde{D}_j}$ of $\cAp(\Z^t)$ corresponds to the point $ (\ord_{D_j},0)$ of $\cA(\Z^t)\times T_M(\Z^t)$.

Recall that the index set of unfrozen indices for $\cAp$ is $\Iuf$. 
In particular, for every $k \in \Iuf$ we have that the $k^{\text{th}}$ element of $\widetilde{\seed}_{j}$ is of the form $( e_{k;j},0)$, where $e_{k;j}$ is the $k^{\text{th}}$ element of $\seed_j$. Then for each $1\leq j\leq s$ we compute
\begin{align*}
\{ (e_{k;j},0),  \ord_{\widetilde{D}_j}\} & = \{ (e_{k;j},0), (\ord_{D_j},0)\} \\
& = \{e_{k;j}, \ord_{D_j} \} \geq 0.
\end{align*}

This tells us that $\ord_{\widetilde{D}_j}$ is optimized for $\widetilde{\seed}_j$.
Let $W_{Y\times T_M}=\sum_{j}^{s}\tf^{\cAp^\vee}_{\ord{\widetilde{D}_j}}$ be the superpotential associated to $\cAp \subset Y \times T_M$.
By Proposition 9.7 and Lemma 9.10 (3) of \cite{GHKK} the integral points of $\Xi_{Y \times T_M}$ can be described as
\[
\Xi_{Y \times T_M}\cap (\Z) = \{ b \in \Theta(\cAp) \mid \ord_{i(b)} (\tf^{\cAp^{\vee}}_j)\geq 0 \text{ for all } j\}.
\] 
We define $\cmid(Y\times T_M)$ to be the vector subspace of $\cmid (\cAp)$ spanned by the theta functions parametrized by $\Xi_{Y \times T_M}(\Z^T)$. 
For the convenience of the reader we point out that in the notation of \cite[\S9]{GHKK} the partial compactification $Y\times T_M$ of $\cAp$ would be denoted by $\overline{\cA}_{\text{prin}}^{S}$ and $\Xi_{Y \times T_M}(\Z)$ by $\Theta(\overline{\cA}_{\text{prin}}^{S})$, where  $S:=\{ i(\ord_{\widetilde{D}_1}),\dots , i(\ord_{\widetilde{D}_s})\}$.
By \cite[Lemma 9.10(2)]{GHKK} we have
\[\cmid(Y\times T_M)=H^0(Y\times T_M, \mathcal{O}_{Y\times T_M}) \cong H^0(Y, \mathcal{O}_{Y})\otimes_{\Bbbk} H^0( T_M, \mathcal{O}_{ T_M}).
\]
In particular, $H^0(Y\times T_M, \mathcal{O}_{Y\times T_M})$ has a theta basis parametrized by $\Theta(Y\times T_M)$.
The theta function $ \tf^{\cA}_{\ord_{D_j}}$ is obtained from $\tf^{\cAp}_{\widetilde{D}_j}$ by specializing the coefficients to $1$. This implies that
\[
\Xi_{Y}\cap \Trop_{\Z}(\cA^{\vee})= \Xi_{Y \times T_M} \cap \Trop_{\Z}(\cA^{\vee}).
\]
We conclude that  $ H^0(Y, \mathcal{O}_Y)$ has a theta basis parametrized by the integral point of $\Xi_{Y}$.

\end{proof}

\section{Valuations on middle cluster algebras and adapted bases}
\label{sec:cluster_valuations}
 In \cite{FO20} the authors noticed that the so-called {\bf g}-vectors associated to cluster variables can be used to construct valuations on $\Bbbk(\cA)$ provided $\Gamma$ is of full-rank. In this section we study some properties of these valuations. We extend this approach for quotients of $\cA$ and (fibres of) $\cX$.

 Let $\cV$ be a scheme of the form $\cA, \cX, \cA/T_H$ or $\cXe$. 
 Recall from \S\ref{sec:tf_and_parametrizations} that every theta function on $\cV$ is labeled with a point of $\Trop_{\Z}(\cV^\vee)$, see \eqref{eq:unif}. 

\begin{definition}
\label{def:dom_order}
Suppose $\Gamma$ is of full-rank and let $ \seed \in \orT$ be a seed for $\Gamma$. 
The {\bf opposite dominance order} on $M^\circ$ defined by $\seed$ is the partial order $\preceq_{\seed}$ on $M^\circ$ determined by the following condition:
\begin{equation}
\label{eq:dom_order}
m_1 \prec_{\seed} m_2 \ \Leftrightarrow \ m_2=  m_1 + p^{\ast}_1(n) \text{ for some }n\in N^+_{\uf, \seed}.
\end{equation}
\end{definition}

\begin{remark}
\label{rem:dom_order}
In Definition \ref{def:dom_order}, $m_1\preceq_{\seed} m_2$ means that either $m_1 \prec_\seed m_2 $ or $m_1=m_2$. We will also adopt this notation for other orders we consider.
The dominance order was originally considered in \cite[Proof of Proposition 4.3]{Labardini_et_al_CC-alg} and it is the opposite order to the one given in Definition \ref{def:dom_order}. 
This order was exploited by \cite{Qin17,Qintropical} in his work on bases for cluster algebras.
The full-rank condition is needed so that $\preceq_{\seed}$ is reflexive. However, observe that for every seed $\seed$ such that $\text{ker}(p_1^*)\cap N^+_{\uf , \seed} = \emptyset$, equation \eqref{eq:dom_order} still determines a partial order on $M^\circ$ even if $\Gamma$ is not of full-rank. 
Nonetheless, whenever we talk about an (opposite) dominance order in this paper we will be tacitly assuming that $\Gamma$ is of full-rank.
\end{remark}

It is straightforward to verify that $\preceq_{\seed}$ is {\bf linear}. That is, $ m_1 \preceq_\seed m_2 $ implies that $m_1 + m \preceq_\seed m_2 + m$ for all $m \in M^\circ$.

\begin{definition}
\label{def:val}
Let $A$ be an integral domain with a $\Bbbk$-algebra structure, $L$ a lattice isomorphic to $\Z^r$ and $\leq$ a total order on $ L$. A {\bf valuation} on $A$ with values in $L$ is a function $\nu : A\setminus \{0 \} \to (L,<)$ such that 
\begin{itemize}
    \item[(1)] $\nu(f+g) \geq  \min\{\nu(f), \nu(g)\}$, unless $f+g=0$,
    \item[(2)] $\nu(fg)= \nu(f) + \nu(g)$,
    \item[(3)] $\nu(cf)=\nu(f)$ for all $c \in \Bbbk^* $.
\end{itemize}
For $l \in L$ we define the subspace
$
A_{\nu \geq l}:= \{ x\in A \setminus \{0\} \mid \nu(x)\geq l\} \cup \{ 0 \}
$
of $A$. The subspace $
A_{\nu > l}$ is defined analogously.
We say that $\nu $ has {\bf 1-dimensional leaves} if the dimension of the quotient
\begin{equation}
\label{eq:graded_piece}
A_l:=A_{\nu \geq l} \big{/} A_{\nu > l}    
\end{equation}
is either $0$ or $1$ for all $l\in L$. A basis $B$ of $A$ is {\bf adapted} for $ \nu $ if for all $l\in L$ the set $B\cap A_{\nu \geq l}$ is a basis of $A_{\nu\geq l}$.
\end{definition}

\begin{lemma}
\thlabel{product_and_order}

Assume $\Gamma $ is of full-rank. 
Let $\vartheta^{\cA}_{m_1},\vartheta^{\cA}_{m_2}\in \text{mid}(\cA)$ with $m_1,m_2\in =\Trop_{\Z}(\cA^{\vee}_{\seed^\vee})=M^\circ$. 
Then the product $\vartheta^{\cA}_{m_1}\vartheta^{\cA}_{m_2}$ expressed in the theta basis of $\text{mid}(\cA)$ has the following form
\[
\vartheta^{\cA}_{m_1}\vartheta^{\cA}_{m_2}= \vartheta^{\cA}_{m_1+m_2}+ \sum_{m_1+m_2 \prec_{\seed}  m}c_{m}\vartheta^{\cA}_{m}.
\]
\end{lemma}

\begin{proof}
First notice that for any broken line $\gamma $ we have that
\begin{equation*}
F(\gamma)=I(\gamma) +a_{1}p^*_1(n_{1})+ \dots + a_{r}p^*_1(n_{r}),  
\end{equation*}
where $a_1, \dots , a_r $ are non-negative integers and $n_1, \dots , n_r \in N^+_{\uf, \seed}$. 
This follows from \cite[Theorem 1.13]{GHKK} and the bending rule of broken lines (\ie \thref{def:genbroken}(4)). 
In particular, we have that $a_{1}n_{1}+ \dots + a_{r}n_{r}\in N^+_{\uf, \seed} \cup \{ 0 \} $. 
Moreover, 
$a_{1}p^*_1(n_{1})+ \dots + a_{r}p^*_1(n_{r}) = 0$ if and only if $a_1=\cdots = a_r =0$.
Therefore, $I(\gamma) \preceq_{\seed} F(\gamma)$ and $I(\gamma)=F(\gamma)$ if and only if $\gamma$ does not bend at all.

The statement we want to prove already follows from the observations made above.
Indeed, by \cite[Definition-Lemma 6.2]{GHKK} we know that $ \alpha(m_1,m_2,m)\neq 0$ if and only if there exist broken lines $\gamma_1$ and $\gamma_2$ such that 
$I(\gamma_i)=m_i$ for $i \in \{1,2\}$ and $F(\gamma_1)+F(\gamma_2)=m=\gamma_1(0)=\gamma_2(0)$.
Therefore, if $ \alpha(m_1,m_2,m)\neq 0$ then $ m_1 + m_2=I(\gamma_1) + I(\gamma_2) \preceq_{\seed} m $. 
Moreover, the equality
$ m_1+ m_2=m$
holds if and only if both $\gamma_1$ and $\gamma_2$ do not bend at all. 
This latter case can be realized in a unique way, therefore, $\alpha(m_1,m_2,m_1+m_2)=1$. 
\end{proof}

From now on the symbol $\leq_{\seed} $ is used to denote a total order on $M^\circ$ refining $\preceq_{\seed}$.

\begin{definition}
Let ${\bf m}=(m_{\seed^\vee})\in \Trop_{\Z}(\cA^{\vee})$. The {\bf g-vector of} $ \tf^{\cA}_{\bf m}$ {\bf with respect to} $\seed $ is
\begin{equation}
\label{eq:red-g-val-A}
{\bf g}_{\seed}\left(\tf^{\cA}_{\bf m}\right):= m_{\seed^\vee}
\in \Trop_{\Z}(\cA^{\vee}_{\seed^\vee}).
\end{equation}
\end{definition}

\begin{definition}
\thlabel{g_valuation_A}
Assume $\Gamma$ is of full-rank and think of $M^{\circ}$ as $\Trop_{\Z}(\cA^{\vee}_{\seed^\vee})$. Let $\gv_{\seed}:\cmid(\cA) \setminus \{ 0\} \to (M^{\circ},\leq_{\seed})$ be the map given by 
\begin{equation}
\label{eq:g_val}
    \gv_{\seed}(f):= \min{}_{\leq_{\seed}}\{m_1, \dots , m_t\},
\end{equation}
where $f=c_1\vartheta^{\cA}_{m_1} + \dots + c_t\vartheta^{\cA}_{m_t}$, $m_j\in M^\circ$ and $c_j\not=0$ for all $j=1,\dots,t$ is the expression of $f$ in the theta basis of $\text{mid}(\cA)$.
\end{definition}

\begin{lemma}
\thlabel{g_is_val}
For every seed $\seed$ the map $\gv_{\seed} $ is a valuation on $\cmid(\cA)$ with 1-dimensional leaves and the theta basis $\{ \tf_{m} \mid m\in \Theta (\cA) \}$ is adapted for $\gv_{\seed} $.
\end{lemma}
\begin{proof}
This statement follows from \cite[Remark 2.30]{KM_Khovanskii_bases} but for the convenience of the reader we give a proof here. Items (1) and (3) of Definition~\ref{def:val} follow directly from the definition of $\gv_{\seed}$. 
For item (2) consider the expressions $f=\sum_{i=1}^r c_i\vartheta^{\cA}_{m_i}$ and $g=\sum_{j=1}^s c'_j\vartheta^{\cA}_{m'_j}$ where all $c_i$ and $c'_j$ are non-zero.
Then by \thref{product_and_order}
\begin{eqnarray}\label{eq:fg in basis}
fg=\sum_{i,j} c_ic'_j\left(\vartheta^{\cA}_{m_i+m'_j} + \sum_{m_i+m'_j\prec_{\seed} m}c_{m}\vartheta^{\cA}_{m}\right).
\end{eqnarray}
By definition of $\gv_{\seed}$ we have $m_\mu:=\gv_{\seed}(f)\prec_{\seed} m_i$ for all $i\in \{1,\dots, r\} \setminus \{ \mu \}$ and $m'_\nu:=\gv_{\seed}(g)\prec_{\seed} m'_j$ for all $j\in \{1,\dots,s\}\setminus \{\nu\}$. 
We need to show that the term $\vartheta_{m_\mu+m'_\nu}$ appears with non-zero coefficient in $fg$.
Assume there exist $i\not =\mu$ and $j\not=\nu$ such that $m_\mu+m'_\nu=m_i+m'_j$. 
Then as $\prec_{\seed}$ is linear we have
\[
m_\mu +m'_\nu \prec_{\seed} m_\mu + m'_j \prec_{\seed} m_i + m'_j, 
\]
a contradiction. 
Hence, the term $\vartheta_{m_\mu+m'_\nu}$ appears in the expression \eqref{eq:fg in basis} of $fg$ with coefficient $c_\mu c'_\nu\not =0$ and $\gv_{\seed}(fg)=m_\mu+m'_\nu=\gv_{\seed}(f)+\gv_{\seed}(g)$.

The fact that ${\bf g}_{\seed}$ has one dimension leaves follows directly from (\ref{eq:g_val}). It is also clear from the definitions that for $m\in M^{\circ}$ the subspace $\cmid(\cA)_{m} $ as in (\ref{eq:graded_piece}) is isomorphic to $\Bbbk\cdot \tf^{\cA}_{m}$ if $m \in \Theta(\cA)$ and $0$-dimensional otherwise. 
In particular, the fact that we have a bijection between the set of values of $\gv_\seed$ and the elements of the theta basis is equivalent to the theta basis being an adapted basis, see \cite[Remark 2.30]{KM_Khovanskii_bases} .
\end{proof}

\begin{corollary}
The image of the valuation ${\bf g}_{\seed}$ is independent of the linear refinement $\leq_{\seed}$ of $\preceq_{\seed}$.
\end{corollary}
\begin{proof}
Since the theta basis is adapted for ${\bf g}_{\seed}$ we have 
\[
{\bf g}_{\seed}\lrp{\cmid(\cA)\setminus \{0\}}= {\bf g}_{\seed}\lrp{\Theta(\cA)}.
\]
The result follows.
\end{proof}

\begin{remark}
\thlabel{g-val-field}
Since $\cmid(\cA)$ is a domain (see Remark \ref{rem:integral_domain}) whose associated field of fractions is isomorphic to $\Bbbk(A_i :i \in I)$, we can extend the valuation $ {\bf g}_{\vb s} $ on $\text{mid}(\cA)$ to a valuation on $\Bbbk(A_i :i \in I)$ by declaring ${\bf g}_{\vb s} (f/g):={\bf g}_{\vb s} (f)- {\bf g}_{\vb s} (g) $.
\end{remark}

The valuation ${\bf g}_{\seed} $ is called the {\bf {\bf g}-vector valuation associated to $\seed$}. 

We now turn our attention to quotients of $\cA $. We keep the assumption that $\Gamma$ is of full-rank and consider a saturated sublattice $H=H_{\cA}$ of $K^\circ$. 
Recall from \S\ref{tf_quotient} that 
\[
\Trop_{\Z}((\cA/T_H)^{\vee}_{\seed^\vee})= H^{\perp}.
\]
Since $ \Theta(\cA/T_{H})_{\seed^\vee}\subset  H^ \perp$, we can restrict restrict the total order $\leq_{\seed}$ on $M^{\circ}$ to $H^{\perp}$ to obtain a {\bf g}-vector valuation on $\cmid(\cA/T_{H})$ associated to $\seed$ as in the previous cases:
\[
{\bf g}_{\seed}: \cmid(\cA/T_H)\setminus\{0\} \to \Trop_{\Z}((\cA/T_H)^{\vee}_{\seed^\vee}).
\]

\begin{remark}
\label{rem:g_val_quotient}
    As opposed to the case of $\cA$, in general the field of fractions of $\cmid(\cA/T_H)$ might not be isomorphic to $\Bbbk(\cA/T_H)$. 
    This fails for example if the smallest cone in $\Trop_{\R}((\cA/T_H)^{\vee}_{\seed^\vee})$ containing $\Theta(\cA/T_H)_{\seed^\vee}$ is not full-dimensional.
    However, the field of fractions of $\cmid(\cA/T_H)$ is isomorphic to $\Bbbk(\cA/T_H)$ provided $\cA/T_H$ satisfies the full Fock--Goncharov conjecture. 
    In such a case, a {\bf g}-vector valuation on $\cmid(\cA/T_H)$ can be extended to $\Bbbk(\cA/T_H)$ as in \thref{g-val-field}.
\end{remark}

We now treat the case of $\cX$. So fix a cluster ensemble lattice map $p^*:N \to M^{\circ} $ and a seed $\seed $. 
Consider the identifications $\Trop_{\Z}(\cX^\vee_{\seed})= d\cdot N$ and  $\Trop_{\Z}(\cA^{\vee}_{\prin,\widetilde{\seed}^\vee}) = M^{\circ}_{\prin}=M^{\circ}\oplus N$ where $\widetilde{\seed}$ is the seed for $\Gamma_\prin$ obtained mutating $\seed_{0_\prin}$ in the same sequence of directions needed to obtain $\seed$ from $\seed_0$. 
Recall from \S\ref{sec:tf_X} that we have an inclusion $\Trop_{\Z}(\cX^\vee_{\seed})\to \Trop_{\Z}(\cA^{\vee}_{\prin,\widetilde{\seed}^\vee})$ given by $dn \mapsto (p^*(n),n)$.

\begin{definition}
Let ${\bf n}=(dn_{\seed^\vee})\in \Trop_{\Z}(\cX^{\vee})$. The {\bf c-vector of} $ \tf^{\cX}_{\bf n}$ with respect to $\seed $ is
\begin{equation}
\label{eq:red-g-val-X}
{\bf c}_{\seed}\left(\tf^{\cX}_{\bf n}\right):= dn_{\seed^\vee}
\in \Trop_{\Z}(\cX^{\vee}_{\seed^\vee}).
\end{equation}
\end{definition}

\begin{remark}
Observe that $\cv_\seed (\tf^{\cX}_{\bf n})$ is an element of $d\cdot N$. In practice we could work with the lattice $N$ as opposed to $d\cdot N$ as they are canonically isomorphic. 
The lattice $N$ is the set where the ${\bf c}$-vectors (in the sense of \cite{NZ}) live. 
\end{remark}

\begin{definition}
The {\bf divisibility order} on $ N$ determined by $\seed$ is the partial order $\preceq_{\seed, \text{div}}$ given by
\[
n_1 \preceq_{\seed, \text{div}} n_2 \text{ if and only if } 
n_2- n_1 \in N_{\seed}^+.
\]
\end{definition}

\begin{lemma}
\thlabel{lem:restriction}
The restriction of $\preceq_{\widetilde{\seed}^\vee}$ to the $N$ component of $ M^\circ_{\prin}$ coincides with the divisibility order $\prec_{\seed,\text{div}}$ on $N$.
\end{lemma}
\begin{proof}
Let $p^*_{\prin,1}:N_{\uf, \prin}\to M^\circ_\prin$ be the given by $(n,m)\mapsto \{ (n,m), \cdot \}_{\prin}$ (in other words, $p^*_{\prin,1}$ corresponds to the map $p_1^*$ in \eqref{eq:p12star} for $\Gamma_{\prin}$). 
In particular, $p^*_{\prin,1} (n,0) = (p^*_1(n), n) $.
Let $n_1,n_2 \in N$ be distinct elements such that $n_2-n_1 \in N^+_\seed$. Let $ \widetilde{m}_i=(p_1(n_i),n_i)$ for $i = 1,2$. Then $\widetilde{m}_2 -\widetilde{m}_1= (p^*_1(n_2-n_1), n_2 -n_1)$. The result follows.
\end{proof}

The next result follows at once from \thref{lem:restriction} and \thref{product_and_order}.

\begin{lemma}
\label{lem:tf_X_pointedness}
Let $ \tf^{\cX}_{dn_1},\tf^{\cX}_{dn_2} \in \cmid(\cX)$ with $d_1n_1, d_2n_2 \in \Trop_{\Z}(\cX^\vee_{\seed^\vee})=d\cdot N$.
 Then the product $\vartheta^{\cX}_{dn_1}\vartheta^{\cX}_{dn_2}$ expressed in the theta basis of $\cmid(\cX)$ is of the following form
\[
\vartheta^{\cX}_{dn_1}\vartheta^{\cX}_{dn_2}= \vartheta^{\cX}_{dn_1+dn_2}+ \sum_{n_1+n_2 \ \prec_{\seed, \text{div}} \ n} c_{n}\vartheta^{\cX}_{dn}.
\]
\end{lemma}

From now on we let $\leq_{\seed,\text{div}}$ be any total order refining $\preceq_{\seed, \text{div}}$. 

\begin{corollary}\label{cor:gv on midX}
Let ${\bf c}_{\seed}:\cmid(\cX) \setminus \{ 0\} \to (d \cdot N,\leq_{\seed,\text{div}})$ be the map defined by 
\[
{\bf c }_{\seed}(f):= \min{}_{\leq_{\seed,\text{div}}}\{n_1, \dots , n_t\},
\]
where $f=c_1\tf^{\cX}_{d n_1} + \dots + c_t\tf^{\cX}_{d n_t}$ is the expression of $f$ in the theta basis of $\text{mid}(\cX)$. Then ${\bf c }_{\seed}$ is a valuation with 1-dimensional leaves and the theta basis for $\cmid(\cX) $ is adapted for ${\bf c}_\seed$.
\end{corollary}

We now let $\cX_{\bf 1}$ be the fibre of $\cX$ associated to a sublattice $H:= H_{\cX} \subset K$. In order to define a {\bf c}-vector valuation on $\cmid(\cX_{\bf 1})$ we need that 
\[
H\cap N^+_{\seed}= \emptyset.
\]
Since, if this condition holds,  $\preceq_{\seed, \text{div}}$ induces a well partial order on $N/H =\mathcal X_{\bf 1,\seed}$ defined as
\[
n_1 + H \preceq_{\seed, \text{div}} n_2+H \quad \text{ if and only if } \quad n_2 - n_1 \in N^+_{\seed}+ H.
\]
The rest of the construction follows from the cases already treated.

\begin{lemma}\label{lem:cval_gval}
Suppose $\Gamma$ is of full-rank and let $p: \cA \to \cX$ be a cluster ensemble map. Then we have a commutative diagram
\[
\xymatrix{
\cmid(\cX) \setminus \{0\} \ar^{p^*}[r] \ar_{{\bf c}_{\seed}}[d] &  \cmid(\cA) \setminus \{0\} \ar^{{\bf g}_{\seed}}[d] \\
\Trop_{\Z}(\cX^{\vee}_{\seed^\vee}) \ar_{(p^\vee)^T\circ i} [r] & \Trop_{\Z}(\cA^{\vee}_{\seed^\vee}) 
}
\]
\end{lemma}
\begin{proof}
It is enough to show that for ${\bf n} \in \Theta(\cX) $ we have
\[
\gv_{\seed}(p^*(\tf^\cX_{\bf n}))=(p^\vee)^T\circ i({\bf c}_{\seed} (\tf^\cX_{\bf n}))
\]
Let $dn=\mathfrak{r}_{\seed^\vee}({\bf n})$. We have that
\[
\tf^{\cX}_{dn}=z^n + \sum_{n\prec_{\seed}n'}a_{n'}z^{n'}.
\]
Therefore,
\[
p^*(\tf^{\cX}_{dn})=z^{p^*(n)} + \sum_{n<_{\seed, \text{div}}n'}a_{n'}z^{p^*(n')}.
\]
We conclude that $\gv_{\seed}(p^*(\tf^\cX_{\bf n}))=p^*(n)$. On the other hand we have that $ {\bf c}_{\seed} (\tf^\cX_{\bf n})=dn$. We compute
\begin{align*}
    (\Lp)^T\circ i (dn)= ((\Lp)^*)^*(-dn)=\lrp{-\frac{1}{d}(p^*)^*)}^*(-dn)=p^*(n).
\end{align*}
The claim follows.
\end{proof}

We would like to treat {\bf g}-vector valuations for varieties of the form $\cA$ and $\cA/T_H$ and {\bf c}-vector valuations on $\cX$ and $\cXe$ in a uniform way.
With this in mind we introduce the following notation.

\begin{notation}
\thlabel{not:g-val}
Let $\cV$ be a cluster variety and $ \cV^{\vee}$ its Fock--Goncharov dual. The cluster valuation on $\cmid(\cV)$ associated to a seed $\seed\in \orT$ is 
\[
\nu_{\seed}:\cmid (\cV)\setminus\{0\} \to (\Trop_{\Z}(\cV^\vee_{\seed^\vee}), <_{\seed}),
\]
where $\Trop_{\Z}(\cV^\vee_{\seed^\vee})$ is as in \eqref{eq:unif} and $<_{\seed}$ is a linear order on $\Trop_{\Z}(\cV^\vee_{\seed^\vee})$ refining $\prec_\seed$ in case $\cV=\cA$ or $\cA/T_H$ and it refines $\prec_{\seed,\text{div}}$ if $\cV=\cX$ or $\cXe$. 
\end{notation}

\section{Newton--Okounkov bodies} \label{sec:no}

In this section we provide a general approach to construct Newton--Okounkov bodies associated to certain partial minimal models of varieties with a cluster structure. 
In particular, we treat a situation that often arises in representation theory where the universal torsor of a projective variety has a cluster structure of type $\cA$.  
The Newton--Okounkov bodies we construct depend on the choice of an initial seed. Hence we discuss how the bodies associated to different choices of initial seed are related and introduce  the intrinsic Newton--Okounkov body which is seed independent.

\subsection{Schemes and ensembles with cluster structure}

\begin{definition}\thlabel{def:cluster-structure}
We say a smooth 
scheme (over $\Bbbk$) $V$ {\bf can be endowed with cluster structure of type} $\cV$ if there is a birational map $ \Phi: \cV \dashrightarrow V$ which is an isomorphism outside a codimension two subscheme of the domain and range.
In this setting, we say that the pair $(V,\Phi)$ is {\bf{a scheme with cluster structure of type}} $\cV$.
\end{definition}

\begin{remark}
We are straying slightly from \cite{CMNcpt} in \thref{def:cluster-structure}.
Specifically, we are now including $\Phi$ as part of the data defining a scheme with cluster structure.
So, given two different birational maps $\Phi_1:\cV_1 \dashrightarrow V$ and $\Phi_2: \cV_2 \dashrightarrow V$ as in \thref{def:cluster-structure}, we now consider $(V,\Phi_1)$ and $(V,\Phi_2)$ different as schemes with cluster structure (as is the case, for example, for open positroid varieties, see Remark~\ref{rmk:open positroid}). 
Nevertheless, when the map $\Phi$ is clear from the context or we are just dealing with a single birational map $\cV \dashrightarrow V$, we will simply say that $V$ has a cluster structure of type $\cV$.
\end{remark}

Let $V=(V,\Phi)$ be a scheme with a cluster structure of type $\cV$. Since $V$ is normal and isomorphic to $\cV$ up to co-dimension $2$ then $V$ and $\cV$ have isomorphic rings of regular functions. In turn, we can talk about polynomial theta functions on $V$ which we denote by $\tf^V_{\bf v}$ for ${\bf v}\in \Theta (\cV)$.
Moreover, recall that $\cV$ is log Calabi--Yau. By \cite[Lemma~1.4]{GHK_birational} $V$ is also log Calabi--Yau. Hence, $V$ has a canonical volume form whose pullback by $\Phi $ coincides with the canonical volume form on $\cV$. Moreover, a (partial) minimal model $V\subset Y$ and its boundary can be defined as in Definition \ref{def:cv_minimal_model}.

\begin{definition}{\cite{GHK_birational}}
 An inclusion $V \subset Y$ as an open subscheme of a normal variety $Y$ is a {\bf partial minimal model} of $ V$ if the canonical volume form on $V$ has a simple pole along every irreducible divisor of $Y$ contained in $ Y \setminus V$. It is a {\bf minimal model} if $Y$ is, in addition, projective. We call $ Y \setminus V$ the {\bf boundary} of $V  \subset Y$.
\end{definition}

\begin{definition}
    Suppose $\Phi:\cV \dashrightarrow V$ endows $V$ with a cluster structure of type $\cV$ and that the cluster valuation $\nu_{\seed}$ extends to $\Bbbk(\cV) $. 
    Then the {\bf cluster valuation} $\nu^{\Phi}_{\seed}:\Bbbk(V)^*\to \Trop_{\Z}(\cV^\vee)$ is given by
    \[
    \nu^{\Phi}_{\seed}(f)=  \nu_{\seed}(\Phi^*(f)).
    \]
\end{definition}

\begin{definition}
    Suppose $\Phi_{\cA}:\cA \dashrightarrow V_1$ and $\Phi_{\cX}:\cX \dashrightarrow V_2$ endow $V_1$ (resp. $V_2$) with cluster structures of type $ \cA$ (resp. $\cX$). We say that $V_1 \overset{\tau}{\to} V_2$ is a cluster ensemble structure if there exists a cluster ensemble map $p:\cA \to \cX$ such that the following diagram commutes
    \[
    \xymatrix{
    V_1 \ar^{\tau}[r] & V_2 \\
    \cA \ar@{-->}^{\Phi_{\cA}}[u] \ar_p[r] & \cX \ar@{-->}_{\Phi_{\cX}}[u].
    }
    \]
\end{definition}

\subsection{Newton--Okounkov bodies for Weil divisors supported on the boundary}
\label{sec:NO_bodies}
Throughout this section we let $\cV$ be a scheme of the form $ \cA$, $\cX$, $\cA/T_{H}$ or $\cX_{{\bf 1}}$. Whenever we talk about a cluster valuation on $\cmid(\cV)$ we are implicitly assuming we are in a setting where such valuation exist, see \S\ref{sec:cluster_valuations}.

\begin{definition}{\cite{GHKK}}
\label{def:positive_set}
A closed subset $S\subseteq \Trop_{\R}(\cV^{\vee})$ is {\bf positive} if for any positive integers $d_1, d_2$, any $p_1\in d_1\cdot S(\Z)$, $p_2\in d_2\cdot S(\Z)$ and any $r \in \Trop_{\Z}(\cV^{\vee})$ such that $\alpha (p_1,p_2,r)\neq 0$, we have that $r \in (d_1 +d_2)\cdot S(\Z)$. 
\end{definition}

\begin{remark}
\label{rem:positive_sets_in_vs}
  We can also define positive sets inside $\Trop_{\R}(\cV^{\vee})_{\seed^\vee}$ in exactly the same way they are defined in Definition \ref{def:positive_set}. 
  In particular we have that  $S\subset \Trop_{\R}(\cV^{\vee})$ is positive if and only if $\mathfrak{r}_{\seed^\vee}(S)\subset \Trop_{\R}(\cV^\vee_{\seed})$ is positive.
\end{remark}

In \cite[\S8]{GHKK} the authors discuss how positive sets give rise to both, partial minimal models of cluster varieties and toric degenerations of such. In this section we study the inverse problem. 
Namely, we let $(V,\Phi)$ be a scheme with a cluster structure of type $\cV$ and construct Newton--Okounkov bodies associated to a partial minimal model $V \subset Y$ (see \S\ref{sec:minimal_models}). 
Then we show that under suitable hypotheses these Newton--Okounkov bodies are positive sets. 
We let $D_1, \dots , D_s $ be the irreducible divisors of $Y$ contained in the boundary of $V\subset Y$ and let $D:=\bigcup_{j=1}^s D_j$.

Given a Weil divisor $D'$ on $Y$ we denote by $R(D')$ the associated {\bf section ring}. 
Recall that $R(D')$ can be described as the $\Z_{\geq 0}$-graded ring whose $k^{\mathrm{th}}$ homogeneous component is
\begin{equation*}
    R_k(D') := H^0(Y, \mathcal{O}(kD'))= \lrc{  f\in \Bbbk(Y)^* \mid \text{div}(f)+kD'\geq 0 }\cup \{ 0\},
\end{equation*}
where $\text{div}(f)$ is the principal divisor associated to $f$.
Even more concretely, if $D'=c_1  D'_1 + \cdots + c_{s'}D'_{s'}$, where $D'_1, \dots , D'_{s'}$ are distinct prime divisors of $Y$ and $c_1, \dots , c_{s'}$ are non-negative integers, then $ R_k(D')$ is the vector space consisting of the rational functions on $Y$ that are regular on the complement of $\bigcup_{j=1}^{s'} D'_j$ and whose order of vanishing along every prime divisor $D'_j$ is bounded below by $-kc_j$. The multiplication of $R(D')$ is induced by the multiplication on $ \Bbbk(Y)$. 

\begin{definition}
\thlabel{def:NOlb}
Let $\nu:\Bbbk(Y)\setminus \{ 0 \} \to L$ be a valuation, where $(L, < )$ is a linearly ordered lattice. Let $D'$ be a Weil divisor on $Y $ having a non-zero global section. For a choice of non-zero section $\tau \in R_1 (D')$ the associated {\bf Newton--Okounkov body} is
\eqn{
\Delta_\nu(D',\tau) := \overline{\conv\Bigg( \bigcup_{k\geq 1}  \lrc{\frac{\nu\lrp{f/\tau^k}}{k} \mid f\in R_k(D')\setminus \{0\} } \Bigg) }\subseteq L\otimes \R,
}
where $\conv $ denotes the convex hull and the closure is taken with respect to the standard topology of $L\otimes \R$.
\end{definition}

From now on we assume that $D'$ has a non-zero global section.
We would like to use a cluster valuation $\cval: \Bbbk(V)\setminus \{ 0\} \to (\Trop_{\Z}(\cV^{\vee}_{\seed^\vee}),<_{\seed})$ to construct Newton--Okounkov bodies. Notice that if $\cV$ satisfies the full Fock--Goncharov conjecture, then it is possible to do so as we can extend $\nu_{\seed}$ from $\cmid(\cV)=\up(\cV)$ to $\Bbbk(\cV) = \Bbbk(Y)$. 
Observe, moreover, that if $D'$ is supported on $D$ (that is $D'=\sum_{j=1}^s c_jD_j$ for some integers $c_1,\dots , c_s$) then every graded piece $R_k(D') $ is contained in $H^{0}(V,\mathcal{O}_V)\cong H^{0}(\cV,\mathcal{O}_{\cV})$, so elements of $R_k(D')$ can be described using the theta basis for $H^0(\cV,\mathcal{O}_{\cV})$. 
Moreover, $\ord_{D_j}\in \cV(\Z^t)$, so we can define $\tf^{\cV}_j$ as in \eqref{eq:def superpotential_summands}.

\begin{definition}
\label{def:graded_theta_basis}
Assume $\cV$ satisfies the full Fock--Goncharov conjecture and that $D'$ is of the form $D'=\sum_{j=1}^s c_jD_j$. We say that $R(D')$ {\bf has a graded theta basis} if for every integer $k\geq 0$ the set of theta functions on $\cV$ parametrized by the integral points of
\[
P_k(D'):= \bigcap_{j=1}^s \lrc{b\in \Trop_{\R}(\cV^{\vee}) \mid  \Trop_{\R}(\tf^{\cV^\vee}_j)(b) \geq -kc_j}
\] 
is a basis for $R_k(D')$.
\end{definition}

The reader should notice that in case $\cV$ has theta reciprocity (see Definition \ref{def:theta_reciprocity}), then the definition of $P_k(D')$ becomes very natural from the perspective of toric geometry, see \S\ref{sec:minimal_models}. We now introduce a notion that allows us to make a good choice for the section $\tau$.

 \begin{definition}
 \label{def:linear_action}
 A subset $L\subset \Theta(\cV)$ is {\bf linear} if 
 \begin{itemize}
 \item for any $a,b\in L$ there exists a unique $r\in\Theta(\cV)$ such that $\alpha(a,b,r)\neq 0$ and moreover, $r\in L$,
 \item for each $a\in L$ there exists a unique $b\in L$ such that $\tf^{\cV}_a \tf^{\cV}_b=1 $. 
 \end{itemize}
We further say that a linear subset $L$ {\bf acts linearly} on $\Theta(\cV)$ if for any $a\in L$ and $ b \in \Theta(\cV)$ there exists a unique $r\in \Trop_{\Z}(\cV^{\vee})$ such that $\alpha(a,b,r)\neq 0$. 
 \end{definition}

For example, if $\cV=\cA$  then $\mathfrak{r}_{\seed}^{-1}(\Nuf^\perp)$ is linear and acts linearly on $\Theta(\cV)$. If $\cV =\cX$
then $\mathfrak{r}_{\seed}^{-1}(\ker(p_2^*))$ is linear and acts linearly on $\Theta(\cV)$.

\begin{theorem}
\label{NO_bodies_are_positive}
Let $V\subset Y$ be a partial minimal model. Assume the full Fock--Goncharov conjecture holds for $ \cV$. Let $D'=\sum_{j=1}^s c_j D_j$ be a Weil divisor on $Y$ supported on $D$ such that $R(D')$ has a graded theta basis. Let $\tau\in R_1(D')$ be such that $\nu^{\Phi}_{\seed}(\tau) $ belongs to a linear subset of $ \Trop_{\Z}(\cV^{\vee}) $ acting linearly on $\Trop_{\Z}(\cV^{\vee}) $. Then the Newton--Okounkov body  $\Delta_{\nu^{\Phi}_{\seed}}(D',\tau)\subset \Trop_{\R}(\cV^{\vee}_{\seed^\vee})$ is a positive set.
\end{theorem}

\begin{proof}
To make notation lighter, throughout this proof we denote $\Delta_{\nu_{\seed}}(D',\tau) $ simply by $ \Delta $, $P_k(D')_\seed$ by $P_k$ and $\nu^{\Phi}_{\seed}$ by $\nu_{\seed}$.
We work in the lattice identification $ \Trop_{\Z}(\cV^{\vee}_{\seed^\vee})$ of $\Trop_{\Z}(\cV^{\vee})$.
The linear subset of the statement corresponds to a sublattice $L \subseteq \Trop_{\Z}(\cV^{\vee}_{\seed^\vee})$.

Consider $d_1, d_2 \in \Z_{>0}$ and $p_1\in d_1\Delta(\Z)$, $p_2\in d_2\Delta(\Z)$. We have to show that for any $r \in \Trop_{\Z}(\cV^{\vee}_{\seed^\vee})$ with $\alpha (p_1,p_2,r)\neq 0$ then $r \in (d_1 +d_2)\Delta(\Z)$. 
For this it is enough to show that $k\Delta = P_k - k\nu_{\seed}(\tau)$ for all $k \in \Z_{>0}$ as we now explain.\footnote{In fact, it is enough that the equality holds at the level of integral points, namely, $k\Delta(\Z)= P_k(\Z) - k\nu_{\seed}(\tau)$. However, we are able to show the stronger condition $k\Delta= P_k - k\nu_{\seed}(\tau)$.}
If this is the case then for $i=1,2$, the point $p_i+d_i\nu_\seed(\tau)$ belongs to $P_{d_i}(\Z)$.
By hypothesis $\tf^V_{p_i+d_i \nu_{\seed}(\tau)}\in R_{d_i}(D')$.
In particular, the product $\tf^V_{p_1+d_1 \nu_{\seed}(\tau)}\tf^V_{p_2+d_2 \nu_{\seed}(\tau)} $ must belong to $R_{d_1+d_2}(D')$ and this product must be expressed as a linear combination of theta functions that belong to  $R_{d_1+d_2}(D')$.  
To finish we just need to convince ourselves that 
\[
\alpha(p_1+d_1\nu_\seed(\tau),p_2+d_2\nu_\seed(\tau), r+(d_1+d_2)\nu_\seed(\tau))\neq 0
\]
as this would imply 
\[
r+(d_1+d_2)\nu_\seed(\tau)\in P_{d_1+d_2}(\Z)=(d_1+d_2)\Delta(\Z)+ (d_1+d_2)\nu_\seed(\tau) .
\]
However, this follows at once from the fact that $\nu_\seed(\tau)$ belongs to the linear subset $L$. 
Indeed, the condition $\alpha(p_1,p_2,r)\neq 0$ implies the existence of a pair of broken lines $\gamma_1, \gamma_2$ such that $I(\gamma_i)=p_i$ and $ F(\gamma_1)+F(\gamma_2)=r$. 
Since $\nu_\seed(\tau)\in L$ we can construct new broken lines $\gamma'_1$ and $\gamma'_2$ such that $I(\gamma'_i)=p_i+d_i\nu_\seed(\tau)$ and $ F(\gamma'_1)+F(\gamma'_2)=r+(d_1+d_2)\nu_\seed(\tau)$ by changing the direction of all the domains of linearity of $\gamma_i$ by  $d_i\nu_\seed(\tau)$. 

We now proceed to show that $k\Delta= P_k-k\nu_\seed(\tau) $ for all $k \in \Z_{>0}$. 
First notice that $aP_1= P_a$ for all $a\in \R_{\geq 0}$ (if $g$ is a positive Laurent polynomial then $g^T(ax)=ag^T(x)$ provided $a$ is non-negative). 
Since $P_k$ is closed and convex in order to show that $k \Delta  \subset P_k- k\nu_\seed(\tau)$ it is enough to show that $\frac{k}{k'}\ \nu_\seed(f/\tau^{k'})=\frac{k}{k'}\ \nu_\seed(f)-k \nu_\seed(\tau)$ belongs to $P_k-k\nu_\seed(\tau)$ for all $k'\geq 1$ and all $f\in R_{k'}(D')\setminus \{0\}$. This follows at once from the fact that $\frac{k}{k'}\nu_\seed(f)\in P_k$ as $\frac{k}{k'}P_{k'}=P_k$.
To obtain the reverse inclusion it is enough to show that the inclusion holds at the level of rational points, namely, $P_k(\Q)-k\nu_\seed(\tau)\subset k\Delta(\Q)$. 
Indeed, since $P_k$ is a finite intersection of rational hyperplanes in $\Trop_{\R}(\cV^{\vee}_{\seed^\vee})$ it can be described as the convex hull of its rational points. 
If $x\in P_k(\Q)$ then $\frac{x}{k}\in \frac{1}{k}P_k(\Q)=P_1(\Q)$. 
Let $d\in \Z_{>0}$ be such that $x':=\frac{dx}{k} \in \Trop_{\Z}(\cV^{\vee}_{\seed^\vee})$. 
In particular, $x'\in P_{d}(\Z)_{\seed}$ which gives that $d^{-1}\nu_\seed(\frac{\tf_{x'}}{\tau^{d}})\in \Delta$. Finally, notice that $d^{-1}\nu_\seed(\frac{\tf_{x'}}{\tau^{d}})=d^{-1}(\nu_\seed(\tf_{x'})-d\nu_\seed(\tau))=d^{-1}x'-\nu_\seed(\tau)$ which implies $x-k\nu_\seed(\tau) \in k\Delta$. 
\end{proof}

In Theorem \ref{NO_bodies_are_positive} the assumption that $R(D')$ has a graded theta basis might seem rather strong. We now provide a situation in which this hypothesis holds and in the next subsection we treat a more robust framework in which this condition follows directly from the equivariant nature of theta functions.

\begin{lemma}
\label{lem:graded_theta_basis}
Let $V\subset Y$ be a minimal model. Assume $D=\sum_{j=1}^n D_j$ is ample with $D'=cD$ very ample for some $c\in \Z_{>0}$. Assume further that the image of the embedding of $Y$ into a projective space given by $D'$ is projectively normal. 
If $\cV$ has theta reciprocity and the theta functions on $\cV$ respect the order of vanishing (see Definition~\ref{def:respect_order}), then $R(D')$ has a graded theta basis.
\end{lemma}
\begin{proof}
It is enough to treat the case $\cV=V$.
Consider the affine cone $\widetilde{Y}$ of the embedding of $Y$ into a projective space given by $D'$. We consider the canonical projection $\widetilde{Y}\setminus \{ 0\} \overset{\pi}{\to } Y $ and let $ \cV':= \pi^{-1}(\cV) $.
Observe that $\cV'\cong \cV \times \C^*$. 
We may think of $\cV'$ as the cluster variety obtained from $\cV$ by adding a frozen index and extending trivially the bilinear form in the fixed data defining $\cV$. 
In particular, $\text{up}(\cV')= \text{up}(\cV)[x^{\pm 1}]$, where $x$ is the coordinate for the $\C^*$ component. Notice that the theta functions on $\cV'$ are of the form $\tf^{\cV'}_{(p,h)}=\tf^{\cV'}_{(0,h)}\tf^{\cV'}_{(p,0)} =x^h\tf^{\cV}_p$, where $\tf^{\cV}_p$ is a theta function on $\cV$ and $h \in \Z=\Trop_{\Z}(\C^*) $.
An analogous description holds for the theta functions on $(\cV')^\vee \cong \cV^\vee \times \C^* $. Namely, these theta functions are of the form $x^h\tf_q^{\cV^\vee}$ for some $h\in \Z$.
We consider the inclusion $R(D')\hookrightarrow \text{up}(\cV')$ given by sending a homogeneous element $f\in R_k(D')$ to $x^kf$. The map is well defined since $f$ is regular on $ \cV $. 
Moreover, if we let $\widetilde{D}_j:= \pi^{-1}(D_j)$ then for all $j$ we have $\ord_{\widetilde{D}_j}\lrp{x^{k}}=k$ and $\ord_{\widetilde{D}_j}\lrp{\tf^{\cV'}_{(p,0)}}=\ord_{D_j}\lrp{\tf^V_p}$. 
In particular, thinking of $\Trop_{\Z}(\cV')$ as $\Trop_{\Z}(\cV)\times \Z$ we have  $\ord_{\widetilde{D}_k}=(\ord_{D_k},1)$. 
Since theta functions on $\cV$ respect the order of vanishing, the same holds for the theta functions on $\cV'$.
This implies that for every $a \in \Z$ and every $j$, $\ord_{D_j}\lrp{\sum_q \alpha_q \tf_q^{\cV}}\geq a$ if and only if $\ord_{D_j}(\tf_q^{\cV})\geq a$ for all $q$ such that $\alpha_q \neq 0$. 
To see this there is only one implication to be checked (the other follows from the axioms of valuations). 
So assume $\ord_{D_j}\lrp{\sum_q \alpha_q \tf_q^{\cV}}\geq a$.
Since $\ord_{D_j}\lrp{\sum_q \alpha_q \tf_q^{\cV}}=\ord_{\widetilde{D}_j}\lrp{\sum_q \alpha_q \tf_q^{\cV}}$ and $x^{-a}\tf_q^{\cV}$ is a theta function on $\cV'$ for all $q$  we have the following
\begin{align*}
    \ord_{D_j}\lrp{\sum_q \alpha_q \tf_q^{\cV}}\geq a & \Longleftrightarrow   \ord_{\widetilde{D}_j}\lrp{\sum_q \alpha_q \tf_q^{\cV}}\geq a \\
    & \Longleftrightarrow   \ord_{\widetilde{D}_j}\lrp{x^{-a}\sum_q \alpha_q \tf_q^{\cV}} \geq 0 \\
    & \Longleftrightarrow   \ord_{\widetilde{D}_j}(x^{-a} \tf_q^{\cV})\geq 0  \text{ for all } q \text{ such that } \alpha_q\neq 0 \\
     & \Longleftrightarrow   \ord_{\widetilde{D}_j}( \tf_q^{\cV})\geq a  \text{ for all } q \text{ such that } \alpha_q\neq 0 \\
     & \Longleftrightarrow   \ord_{D_j}( \tf_q^{\cV})\geq a  \text{ for all } q \text{ such that } \alpha_q\neq 0.
\end{align*}
Since $D'$ is very ample and $Y$ is projectively normal in its embedding given by $D'$ we have that $H^0(\widetilde{Y}, \mathcal{O}_{\widetilde{Y}}) \cong R(D') \hookrightarrow \text{up}(\cV')$.
In particular, if we express $f \in R_k(D')$ as $f= \sum_q \alpha_q \tf^{\cV}_q$, we have that $\ord_{D_j}\lrp{\tf^\cV}\geq -kc$ for all $j$ and all $q $ such that $\alpha_q \neq 0$. This means that $\tf_q^{\cV} \in R_k(D')$ for all such $q$. 
In particular, the theta functions of $\cV$ that lie in $R_k(D')$ have to be a basis a of $R_k(D')$. By theta reciprocity, such theta functions are precisely those parametrized by $P_k(D')$.
\end{proof}

\begin{remark}\label{rmk:toric degen}
If $R(D')$ is finitely generated and the semigroup generated by the image of $\nu^\cV_{\seed}$ is of full-rank and finitely generated then there is a one parameter toric degeneration of $Y$ to the toric variety associated to $ \Delta_{\nu^\cV_{\seed}}(D',\tau)$ \cite{An13}\footnote{That is, there is a scheme $\mathcal Y$ and a flat morphism $ \mathcal{Y}\to \mathbb A^1$ whose generic fibre is isomorphic to $Y$ and special fibre isomorphic to the toric variety associated to $ \Delta_{\nu_{\seed}}(D',\tau)$.}.
As explained in \cite[\S8.5]{GHKK} for cluster varieties of type $\cA$ (regardless of the full-rank assumption) a polyhedral positive set defines a partial compactification $\cA_{\text{prin}} \subset \overline{\cA}_{\text{prin}}$. 
This compactification comes with a flat morphism $\overline{\cA}_{\text{prin}}\to \mathbb A^r$ having $\overline{\cA}=Y$ as fibre over ${\bf 1}=(1, \dots , 1)$ and whose fibre over $0$ is the toric variety associated to the positive set.
Therefore, both constructions can be used to degenerate varieties with a cluster structure to the same toric variety. However, the variety given by the latter construction contains various intermediate fibres that lie in between $\mathcal A=\cV$ and a toric variety. Moreover, while Anderson's degenerations produces a $(\Bbbk^*)$-equivariant family, for the latter degeneration this is the case if and only if $\Gamma$ is of full-rank.
\end{remark}

\subsection{Newton--Okounkov bodies for line bundles via universal torsors}
\label{sec:universal_torsors}

In this section we consider a particularly nice geometric situation that arises often in representation theory. We let $Y$ be an irreducible normal projective scheme whose Picard group $\text{Pic}(Y) $ is free of finite rank $\rho \in \Z_{>0}$ (recall that $\text{Pic}(Y) $ is always abelian).
Following \cite[\S2]{Hau02} (see also \cite[\S3]{BH03}, \cite[Chapter 1]{ADHL}, \cite[\S4]{GHK_birational} or \cite[\S2]{HK00}), we consider the universal torsor of $Y$ and the associated Cox ring (\cf Remark \ref{rem:Cox}). For the convenience of the reader we recall these concepts. We begin by considering the quasi-coherent sheaf of $\mathcal{O}_Y$-modules
\[
\bigoplus_{[\lb] \in \text{Pic}(Y)} \lb. 
\]
In essence, the universal torsor of $Y$ is obtained by applying a relative spectrum construction (also denoted by {\bf Spec}) to this sheaf. 
However, the choice of the representative $\lb $ in the class $[\lb]$ prevents this sheaf from having a natural $\mathcal{O}_Y$-algebra structure. 
To address this situation one can proceed as in \cite[\S2]{HK00} and consider line bundles $\lb_1, \dots, \lb_{\rho} $ whose isomorphism classes form a basis of $\text{Pic}(Y)$. For $v=(v_{1},\dots, v_{\rho})\in \Z^{\rho}$ we let $\lb^{v}= \lb_1^{\otimes v_1}\otimes \cdots \otimes \lb_{\rho}^{\otimes v_{\rho}}$ and consider the quasi-coherent sheaf
\[
\bigoplus_{v \in \Z^{\rho}}\lb^{v}.
\]
This sheaf has a natural structure of a reduced $\mathcal{O}_Y$-algebra that is locally of finite type over $\mathcal{O}_Y$ (the component associated to the zero element of $\text{Pic}(Y)$).
This means that for sufficiently small
affine open subsets $U$ of $Y$, the space $\bigoplus_{v \in \Z^{\rho}}\lb^{v}(U)$ is a finitely generated $\mathcal{O}_Y(U)$-algebra.
The universal torsor of $Y$ is obtained by gluing the affine schemes $\text{Spec}\lrp{\bigoplus_{v \in \Z^{\rho}}\lb^{v}(U)}$.

\begin{definition}
The {\bf universal torsor} of $ Y$ is 
\[
\UT_Y= \textbf{Spec}\lrp{\bigoplus_{v \in \Z^{\rho}}\lb^{v} }.
\]
The {\bf Cox ring} of $Y$ is
\[
\text{Cox}(Y)= H^0 (\UT_Y,\mathcal{O}_{\UT_Y}).
\]
\end{definition}

Universal torsors can be used to generalize the construction of a projective variety from its affine cone as follows.
Observe that the inclusion of $ \mathcal{O}_Y  $ as the degree $0$ part of $\bigoplus_{v \in \Z^{\rho}}\lb^{v} $ gives rise to an affine regular map $\UT_Y\to Y$.
Since $\text{Cox}(Y)$ is $\text{Pic}(Y)$-graded there is an action of $T_{\text{Pic}(Y)^*}= \text{Spec}(\C[\text{Pic}(Y)])$ on $\UT_Y$. 
This action is free and the map $\UT_Y\to Y$ is the associated quotient map (see \cite[Remark 1.4]{Hau02}).

\begin{remark}
\label{rem:Cox}
The notion of a Cox ring associated to a projective variety (satisfying some technical assumptions) was first introduced in \cite[Definition 2.6]{HK00}.
This notion was generalized in \cite{BH03} for any divisorial variety with only constant globally invertible functions, in particular, for any quasi-projective variety (over very general ground fields). However, in \cite{BH03} the term \emph{Cox ring} was not used.
The importance of considering universal torsors and Cox rings in the context of cluster varieties was pointed out in \cite[\S4]{GHK_birational} (see also \cite{Man19}) and satisfactorily pursued in representation theoretic contexts where Cox rings arise naturally, see for example \cite{Mag20}.
\end{remark}

\begin{remark}
For simplicity we are assuming that $\text{Pic}(Y)$ is free. In case it has torsion we can still construct a universal torsor which might not be unique as it depends on the choice of a \emph{shifting family} as in \cite[\S3]{BH03} (see \cite[\S3]{Man19} for a related discussion). Generalizations of the results of this section to the torsion case shall be treated elsewhere. 
\end{remark}

\begin{remark}
If $Y$ is smooth we can construct the Cox ring of $Y$ and the universal torsor (still assuming that $\text{Pic(Y)}$ is torsion free) in an equivalent way. The Cox ring can be defined as $\text{Cox}(Y)=\bigoplus_{v\in \Z^{\rho}} H^0 (Y, \lb^v)$. If $\text{Cox}(Y)$ is finitely generated over $\mathcal{O}_Y$-algebra then the universal torsor $\UT_Y $ is obtained from $\text{Spec}(\text{Cox}(Y))$ by removing the unstable locus of the natural $T_{\text{Pic(Y)}^*}$-action on $\text{Spec}(\text{Cox}(Y))$.
\end{remark}

From now on we assume $ V\subset \UT_Y$ is a partial minimal model where $(V,\Phi)$ is a scheme with a cluster structure of type $\cA$. 
In most of the result of this section we assume that $ V\subset \UT_Y$ has enough theta functions.
Under certain conditions that we discuss next, it is possible to show that $Y$ is a minimal model for a scheme with a cluster structure given by a quotient of $\cA$ and construct Newton--Okounkov bodies for elements of $\text{Pic}(Y)$.
The key point is to relate the action of $T_{\text{Pic}(Y)^*}$ on $\UT_Y$ with the torus actions on $\cA$ arising from cluster ensemble maps.

\begin{lemma}
\label{lem:positivity_of_q_slice}
Let $p:\cA \to \cX$ be a cluster ensemble map and $H\subset K^{\circ}$ be a saturated sublattice. Consider the quotient $\cA/T_H$ and the fibration $w_H:\cA^\vee \to T_{H^*}$ (see \S\ref{sec:FG_dual}).
Then the set
\[
\lrc{ \tf^{\cA}_{\bf m} \in \cmid(\cA)  \mid {\bf m} \in \lrp{\Trop_{\Z}(w_H)}^{-1}(q) \cap \Theta(\cA)}
\]
consists precisely of the polynomial theta functions on $\cA$ whose $T_H$-weight is $q$. Moreover, 
for every $q \in H^*$ the set $\lrp{\Trop_{\R}(w_H)}^{-1}(q)\subset \Trop_{\R}(\cA^{\vee})$ is positive.
\end{lemma}
\begin{proof}
The first claim follows from \thref{prop:dual_fibration}. 
So we only need to show that $\lrp{\Trop_{\R}(w_H)}^{-1}(q)$ is positive. 
In order to show this it is convenient to work with a condition equivalent to positivity called broken line convexity, see \S\ref{sec:intrinsic_NOB}.
We work in the lattice identification $ \Trop_{\R}(\cA^{\vee}_{\seed^\vee})$ of $\Trop_{\R}(\cA^{\vee})$.
We first argue that the set $ \lrp{\Trop_{\R}(w_H)}^{-1}(0)_{\seed^\vee}$ is positive.
First notice that any linear segment $L$ of a broken line segment contained in $ \lrp{\Trop_{\R}(w_H)}^{-1}(0)_{\seed^\vee}$ has itself tangent direction in $ \lrp{\Trop_{\Z}(w_H)}^{-1}(0)_{\seed^\vee}$. 
Let $m\in \lrp{\Trop_{\Z}(w_H)}^{-1}(0)_{\seed^\vee}$ be the tangent direction of $L$. The tangent direction of the following linear segment is of form $m+cp^*(n)$ for some $n\in N^+_{\seed}$ and $c\in  \Z_{\geq 0}$.
For any $h\in H^\circ$ we have
\[
\langle m+cp^*(n),h\rangle =\langle m,h\rangle + c\{n,h\}=0,
\]
as $H^\circ\subset K^\circ$. So the next tangent direction also belongs to $ \lrp{\Trop_{\Z}(w_H)}^{-1}(0)_{\seed^\vee}$.
We conclude that the set  $\lrp{\Trop_{\R}(w_H)}^{-1}(0)_{\seed^\vee}$ is broken line convex and by the main result of \cite{CMNcpt} (see Theorem \ref{thm:mainCMN} below) the set $\lrp{\Trop_{\Z}(w_H)}^{-1}(0)_{\seed^\vee}$ is positive.
This already implies that for any $ x\in \lrp{\Trop_{\R}(w_H)}^{-1}(q)_{\seed^\vee}$ the set $x+ \lrp{\Trop_{\R}(w_H)}^{-1}(0)_{\seed^\vee}$ remains positive. 
Indeed, let $y, z \in x+ \lrp{\Trop_{\R}(w_H)}^{-1}(0)_{\seed^\vee}$. Then $y- z \in \lrp{\Trop_{\R}(w_H)}^{-1}(0)_{\seed^\vee} $. 
In other words, any line segment within the set $ x+\lrp{\Trop_{\R}(w_H)}^{-1}(0)_{\seed^\vee}$ has tangent direction in $\lrp{\Trop_{\R}(w_H)}^{-1}(0)_{\seed^\vee}$. 
Therefore, after bending it will remain in the set $x+\lrp{\Trop_{\R}(w_H)}^{-1}(0)_{\seed^\vee}$. 
Finally, observe that $x+\lrp{\Trop_{\R}(w_H)}^{-1}(0)_{\seed^\vee}=\lrp{\Trop_{\R}(w_H)}^{-1}(q)_{\seed^\vee}$. 
\end{proof}

Having in mind Proposition~\ref{prop:dual_fibration} and the action of the $T_{\text{Pic}(Y)^*}$ on $\UT_Y$ we introduce the following notion.

\begin{definition}
\thlabel{k_and_pic}
The pair $(p,H) $ has the {\bf Picard property} with respect to $V\subset \UT_Y$ if
\begin{itemize}
   \item $H$ and $\text{Pic}(Y)^*$ have the same rank, and
    \item the action of $T_{H}$ on $\cA$ coincides with the action of $T_{\text{Pic}(Y)^*}$ on $\UT_Y$ restricted to the image of $\Phi:\cA \dashrightarrow V $.
\end{itemize}
\end{definition}

Recall the definitions of the superpotential and its associated cone of tropical points from \eqref{eq:def superpotential} and \eqref{eq:def Xi} in \S\ref{sec:minimal_models}.
The following result adapts the content of Proposition \ref{prop:dual_fibration} to this framework.

\begin{lemma}
\label{lem:basis_of_tf}
Suppose that $ V \subset \UT_Y$ is a partial minimal model with enough theta functions and that  $(p,H)$ has the Picard property with respect to this model.
Then for every class $[\lb]\in \text{Pic}(Y)\cong H^*$ we have that the theta functions parametrized by the integral points of the set $ \lrp{\Trop_{\R}(w_H)}^{-1}\lrp{[\lb]}\cap \Xi_{\UT_Y}$ is a basis for $H^0(Y, \lb)$.
In particular, $\Cox(Y)$ has a basis of theta functions which are $T_{\text{Pic}(Y)^*}$-eigenfunctions.  
\end{lemma}

We consider the section ring $R(\lb)=\bigoplus_{k\geq 0} R_k(\lb)  $. The  $k^{\mathrm{th}}$ homogeneous component is defined as $R_k(\lb )=H^0(Y, \lb^{\otimes k})$. The product of $R(\lb)$ is given by the tensor product of sections. 
Fix a seed $\seed\in \orT$, a linear dominance order $<_{\seed}$ on $ \Trop_{\Z}(\cA^{\vee}_{\seed^\vee})$ and consider the valuation
$
\gv^{\Phi}_{\seed}:\Bbbk(V)\setminus \{ 0 \} \to (\Trop_{\Z}(\cA^{\vee}_{\seed^\vee}), <_{\seed}). 
$
Observe that $R_k(\lb)\subset \Cox(Y)$ for all $k$. Hence we can define the Newton--Okounkov body
\eqn{
\Delta_{\gv^{\Phi}_{\seed}}(\lb) := \overline{\conv\Bigg( \bigcup_{k\geq 1}\lrc{ \frac{1}{k}\gv^{\Phi}_{\seed} (f) \mid f\in R_k(\mathcal L)\setminus \{0\} } \Bigg) }\subseteq \Trop_{\Z}(\cA^\vee_{\seed^\vee})=M^{\circ}_{\R}.
}

\begin{theorem}\thlabel{thm:k_and_pic}
Suppose that $ V \subset \UT_Y$ is a partial minimal model with enough theta functions and that  $(p,H)$ has the Picard property with respect to this model.
Then for any line bundle $\lb $ on $Y$
\[
\Delta_{{\bf g}^{\Phi}_{\seed}}(\lb)=\lrp{\Trop_{\R}(w_H)}^{-1}\lrp{[\lb]}_{\seed^\vee}\cap \Xi_{\UT_Y, \seed^\vee}.
\]
In particular, $\Delta_{{\bf g}^{\Phi}_{\seed}}(\lb)$ is a positive subset of $\Trop_{\R}(\cA^{\vee}_{\seed^\vee})$.
\end{theorem}
\begin{proof}
To make notation lighter, throughout this proof we let $S=\lrp{\Trop_{\R}(w_H)}^{-1}\lrp{[\lb]}_{\seed^\vee}\cap \Xi_{\UT_Y,\seed^\vee}$ and denote $\gv^{\Phi}_{\seed}$ simply by $\gv_{\seed}$. Observe that $[\lb^{\otimes k}]=k[\lb]$ in $\text{Pic}(Y)$. 
Therefore, by Lemma \ref{lem:basis_of_tf} we have that ${\bf g}_{\seed}(R_k(\lb))\subseteq \lrp{\Trop_{\R}(w_H)}^{-1}\lrp{k[\lb]}_{\seed^\vee}$ for all $k\geq 1$. 
In particular, $\dfrac{1}{k}{\bf g}_{\seed}(R_k(\lb))\subseteq \lrp{\Trop_{\R}(w_H)}^{-1}\lrp{[\lb]}_{\seed^\vee}$ for all $k \geq 1$.
Since $\lrp{\Trop_{\R}(w_H)}^{-1}\lrp{[\lb]}_{\seed^\vee} $ is closed in $\Trop_{\R}(\cA^{\vee}_{\seed^\vee})$ and convex we have that $\Delta_{{\bf g}_{\seed}}(\lb)\subseteq \lrp{\Trop_{\R}(w_H)}^{-1}\lrp{[\lb]}_{\seed^\vee}$.
Let $ \mathbb B_k$ be the theta basis of $R_k(\lb)$, see \thref{g_is_val}. Since the theta basis is adapted for ${\bf g}_{\seed}$ we have that ${\bf g}_{\seed}(R_k(\lb))={\bf g}_{\seed}(\mathbb B_k)$. 
Since $\cA \subseteq \UT_Y$ has enough theta functions, every theta function $\tf \in \mathbb B_k$ is a global function on $\UT_Y$, therefore, we have that ${\bf g}_{\seed}(\tf) \in \Xi_{\UT_Y}$.
Since $\Xi_{\UT_Y}$ is closed in $\Trop_{\R}(\cA^{\vee}_{\seed^\vee})$, convex and closed under positive scaling then $\Delta_{{\bf g}_{\seed}}(\lb)\subseteq \Xi_{\UT_Y,\seed^\vee}$.
Hence, $\Delta_{{\bf g}_{\seed}}(\lb)\subseteq S$. 
To see the reverse inclusion we notice that the set of rational points of $S$ coincide with the set $ \bigcup_{k\geq 1} \frac{1}{k}  \gv_{\seed }(\mathbb B_k)=  \bigcup_{k\geq 1} \frac{1}{k}  \gv_{\seed }\lrp{R_k(\lb)}$.
Since $S$ can be expressed as the closure of its set of rational points we have that $S\subseteq \Delta_{\gv_{\seed}}(\lb)$.
Finally, since $\lrp{\Trop_{\R}(w_H)}^{-1}\lrp{[\lb]}_{\seed^\vee}$ and $\Xi_{\UT_Y,\seed^\vee}$ are positive sets then $S=\Delta_{{\bf g}_{\seed}}(\lb)$ is an intersection of positive sets. Hence, it is positive.
\end{proof}

\begin{remark}
\label{rem:comparing_NO_bodies}
Under the assumptions of \thref{thm:k_and_pic} we have that $Y$ is a minimal model with enough theta functions for an open subscheme $V'\subset Y$ with a cluster structure given by a birational map $ \Phi':\cA/T_{H}\dashrightarrow V'$ induced by $\Phi$.
To relate the Newton--Okounkov bodies constructed in this section with those constructed in the former we let $\lb $ be isomorphic to $\mathcal{O}(D')$ for some Weil divisor $D'$ on $Y$ satisfying the framework of \S\ref{sec:NO_bodies}.
Under the identification  $ \Trop_{\R}(\cA^{\vee}_{\seed^\vee}) = M^\circ_\R$ we realize $ \Trop_{\R}((\cA/T_H)^{\vee}_{\seed^\vee})$ as the subset of $M^\circ_\R$ orthogonal to $H$ (see \S\ref{tf_quotient}). 
For any $\tau \in R_1(D')$ we have $\Delta_{\gv^{\Phi'}_\seed}(D',\tau)\subset M_{\R}^\circ\cap \lrp{\Trop_{\R}(w_H)}^{-1}(0)_{\seed^\vee}$ and by construction
\[
\Delta_{\gv^{\Phi'}_\seed}(D',\tau) =\Delta_{\gv^{\Phi}_\seed}(\lb)- \gv^{\Phi}_{\seed}(\tau).
\]
\end{remark}

\begin{example}\thlabel{exp:full_flag}
An important class of examples is provided by the base affine spaces. 
Consider $G=SL_{n+1}(\Bbbk)$ and $B\subset G$ a Borel subgroup with unipotent subgroup $U\subset B$.
Then $G/U$ is a universal torsor for $G/B$. 
Moreover, $G/U$ carries a cluster structure induced by the double Bruhat cell $G^{e,w_0}:=B_-\cap Bw_0B$, where $B_-\subset G$ is the Borel subgroup opposite to $B$ (i.e. $B\cap B^-=:T$ is a maximal torus) and $w_0$ the longest element in this Weyl group $S_n$ is identified with a matrix representative in $N_G(T)/C_G(T)$ (the normalizer of $T$ modulo the centralizer of $T$).
The cluster structure on $G^{e,w_0}$ was introduced by Berenstein--Fomin--Zelevinsky in \cite{BFZ05} and it follows that (up to co-dimension 2) $G^{e,w_0}$ agrees with the corresponding $\mathcal A$-cluster variety. 
By \cite[Proposition 23]{Mag15} there is an embedding $G^{e,w_0}\hookrightarrow G/U$ compatible with the cluster structure. 
In particular, $G/U$ is a partial compactification of the $\mathcal A$-cluster variety $G^{e,w_0}$ obtained by adding the locus where frozen variables are allowed to vanish.
Magee further proved in \cite{Mag20} that the full Fock--Goncharov conjecture holds and a cluster ensemble map satisfying \thref{k_and_pic} is provided in \cite{Mag20}.
Similar results were also shown in \cite{Fei17} and \cite{Fei21} from the perspective of quivers with potential.
Hence, we obtain a ${\bf g}$-vector valuation ${\bf g}_{\seed}$ on $H^0(G/U,\mathcal O_{G/U})$ for every choice of seed $\seed$.

In particular, \thref{thm:k_and_pic} applies: recall that the Picard group of $G/B$ is isomorphic to the lattice spanned by the fundamental weights $\omega_1,\dots,\omega_{n}$. 
Let $\Lambda$ denote the dominant weights, \ie its elements are $\lambda=a_1\omega_1+\dots+a_n\omega_n$ with $a_i\in \mathbb Z_{\ge 0}$ and let $\mathcal L_\lambda\to G/B$ be the associated line bundle.
The ring of regular functions on the quasi-affine variety $G/U$ coincides with the Cox ring of the flag variety:
\[
H^0(G/U,\mathcal O_{G/U})\cong \bigoplus_{\lambda \in \Lambda} H^0 (G/B,\mathcal L_\lambda).
\]
Hence, we may restrict the ${\bf g}$-vector valuations ${\bf g}_{\seed}$ for all seeds $\seed$ to the section ring of any line bundle on $G/B$.
The resulting Newton--Okounkov polytopes coincide with slices of the tropicalization of the superpotential corresponding to the compactification. It has been shown in \cite{BF,GKS_typeA} that for certain choices of seeds these polytopes are unimodularly equivalent to Littelmann's string polytopes (see \cite{Lit98,BZ01}).
\end{example}

\begin{example}
Grassmannians also form a distinguished class of examples fitting this framework. We treat this class separately in \S\ref{sec:NO_Grass}.
\end{example}

\subsection{The intrinsic Newton--Okounkov body}\label{sec:intrinsic_NOB} 
In the situation of \S\ref{sec:NO_bodies} or \S\ref{sec:universal_torsors}, we can choose two seeds $\seed, \seed'\orT$ to obtain two Newton--Okounkov bodies, say $\Delta_{\nu_\seed}$ and $\Delta_{\nu_\seed'}$ (these are associated to a line bundle $\lb$ in case we are in a framework as in  \S\ref{sec:universal_torsors} or to a divisor $D'$ and a section $\tau$ in case our framework is as in \S\ref{sec:NO_bodies}).
In the same spirit as in \cite{EH20,FH21} (see also \cite[\S4]{BMNC} and \cite{HN23,CHM22}), in this section we show that if one of $\Delta_{\nu_{\seed}}$ or $\Delta_{\nu_{\seed'}}$ (equivalently both) is a positive set then these Newton--Okounkov bodies are related to each other by a distinguished piecewise linear transformation and, moreover, any such Newton--Okounkov body can be intrinsically described as a \emph{broken line convex hull} (see Theorems \ref{thm:intrinsic} and \ref{thm:intrinsic_lb} below).
In order to obtain the last assertion we rely on \cite{CMNcpt}. Along the way we introduce a theta function analog of the Newton polytope associated to a regular function on a torus.

We start by considering Newton--Okounkov bodies associated to Weil divisors as in \S\ref{sec:NO_bodies}. Let $\cV$ be a scheme of the form $\cA$, $\cX$, $\cA/T_{H}$ or $\cX_{\bf 1}$ and $(V, \Phi)$ a scheme with a cluster structure of type $\cV$.
Denote by $\mathbb{B}_{\tf}(\cV)=\{\tf^{\cV}_{\bf v}\mid {\bf v}\in \Theta(\cV)\}$ the theta basis of $\cmid(\cV)$.
We begin by observing that a cluster valuation $\nu_{\seed}$ on $\cmid(\cV)$ can be thought of as an extension of the composition of the seed-independent map 
\begin{eqnarray}
    \label{eq:nu_seed_free}
    \nu: \mathbb{B}_{\tf}(\cV)  &\to&   \Trop_{\Z}(\cV^\vee)\\
    \nonumber
\tf^{\cV}_{\bf v} &\mapsto &{\bf v},
\end{eqnarray}
with the identification $\mathfrak{r}_{\seed^\vee}:\Trop_{\Z}(\cV^\vee) \to  \Trop_{\Z}(\cV^\vee_{\seed^\vee})$. 
If $ \mathbb B_{\tf}(V)$ denotes the set of polynomial theta functions on $V$ then we can define $\nu^{\Phi} : \mathbb B_{\tf}(V) \to \Trop_{\Z}(\cV^\vee)$ analogously.
Moreover, even though $\Trop_{\Z}(\cV^{\vee})$ may not have a linear structure, if $\Theta (\cV)= \Trop_{\Z}(\cV^{\vee})$ and $L\subseteq \Trop_{\Z}(\cV^{\vee})$ is a linear subset acting linearly on $\Trop_{\Z}(\cV^{\vee})$ (see Definition \ref{def:linear_action}) then for every $y\in L$ we have a well defined ``subtraction" function
\eqn{ (\ \cdot \ )-y: \Trop_{\Z}(\cV^{\vee})  &\to   \Trop_{\Z}(\cV^{\vee})\\
x &\mapsto x-y, }
where $-y $ is the unique point of $ \Trop_{\Z}(\cV^{\vee})$ such that $\tf_y\tf_{-y}=1$ and $x-y$ is the unique point of $\Trop_{\Z}(\cV^{\vee})$ such that $\tf_{x}\tf_{-y}=\tf_{x-y}$. 

We now define our notion of convexity. Recall from  \S\ref{sec:tf_A} that we might think of supports of broken lines as seed independent objects. In light of this we consider the following.

\begin{definition}\label{def:blc_intro} \cite{CMNcpt}
A closed subset $S$ of $\Trop_{\R}(\cV)$ is {\bf{broken line convex}} 
if for every pair of rational points $s_1, s_2$ in $S(\Q)$,
every segment of a broken line with endpoints $s_1$ and $s_2$ is entirely contained in $S$.
\end{definition}

\begin{remark}
\label{rem:non-generic_bl}
The broken lines considered in  Definition \ref{def:blc_intro} include those that are \emph{non-generic}. Namely, broken lines that are obtained as limits of the generic broken lines introduced in \thref{def:genbroken}. See \cite[Definition~3.3]{CMNcpt} for details. 
\end{remark}

The main result of \cite{CMNcpt} asserts that positivity of a set is equivalent to its broken line convexity:

\begin{theorem}
\thlabel{thm:mainCMN}
\cite[Theorem 6.1]{CMNcpt}
Let $\cV$ be a variety of the form $\cA $, $\cX$, $\cA/T_{H}$ or $\cX_{\bf 1}$. Then a closed subset $S$ of $\Trop_{\R}(\cV)$ is is broken line convex if and only if it is positive.
\end{theorem}

Morally, this means that broken line convexity in $\Trop_{\R}(\cV^\vee)$ play the same role in describing partial minimal models of $\cV$ that usual convexity in $M_\R$ plays in describing normal toric varieties $T_N \subset X$.
One appealing feature of the broken line convexity notion is that it makes no reference to any auxiliary data-- given $\cV$, we can talk about broken line convexity in $\Trop_{\R}(\cV^{\vee})$.
In contrast, the Newton--Okounkov bodies we discussed in \S\ref{sec:NO_bodies}
and \S\ref{sec:universal_torsors}
 are convex bodies whose construction depends upon a choice of seed $\seed$.
More generally, a usual Newton--Okounkov body depends not only on the geometric data of a projective variety together with a divisor but also on the auxiliary data of a choice of valuation.
Broken line convexity makes no reference to any such auxiliary data and will lead us to an intrinsic version of a Newton--Okounkov body. 

\begin{definition}
\label{def:bl_convex_hull}
Let $S \subset\Trop_{\R}(\cV^{\vee})$ be a set. The {\bf{broken line convex hull of $S$}}, denoted by $\bconv(S)$, is the intersection of all broken line convex sets containing $S$. 
\end{definition}

\begin{remark}
  We can also define broken line convexity 
 and broken line convex hulls inside $\Trop_{\R}(\cV^{\vee}_{\seed^\vee})$ in exactly the same way they are defined in Definitions \ref{def:blc_intro}  and \ref{def:bl_convex_hull}. 
  In particular, we have that  $S\subset \Trop_{\R}(\cV^{\vee})$ is broken line convex if and only if $\mathfrak{r}_{\seed^\vee}(S)\subset \Trop_{\R}(\cV^\vee_{\seed})$ is broken line convex.
\end{remark}

Using this convexity notion, we describe a set analogous to the Newton polytope of a function on a torus.
\begin{definition}
	Given a regular function ${f= \sum_{{\bf v} \in \Trop_{\Z}(\cV^{\vee})}} a_{\bf v} \tf^{V}_{\bf v}$ on $V$, we define the {\bf{$\tf$-function analogue of the Newton polytope of $f$}} to be 
	\eqn{ \NewtT(f) := \bconv\lrc{ {\bf v} \in \Trop_{\Z}(\cV^{\vee}) \mid a_{\bf v} \neq 0 }. }
\end{definition}

This leads to an intrinsic version of the Newton--Okounkov bodies we have constructed. So consider a partial minimal model $ V \subset Y$ and let $D'$ be a  divisor on $Y$ supported on the boundary of $V\subset Y$.  

\begin{definition}
Assume that $R(D')$ has a graded theta basis (see Definition \ref{def:graded_theta_basis}). Then the associated {\bf{intrinsic Newton--Okounkov body}} is
\eqn{
\Delta_{\mathrm{BL}}(D'):= \bconv\Bigg( \bigcup_{k\geq 1} \Bigg(\bigcup_{f \in R_k(D')} \frac{1}{k} \NewtT(f) \Bigg)  \Bigg)\subseteq \Trop_{\R}(\cV^\vee).
}
\end{definition}

In order to describe how the different realizations of intrinsic Newton--Okounkov bodies are related we record the tropicalization of the gluing map  $\mu^{\cV^\vee}_k:\cV^{\vee}_\seed \dashrightarrow \cV^{\vee}_{\seed'}$ in terms of the fixed data $\Gamma$ and inital seed $\seed_0=(e_i)_{i\in I}$ defining $\cV$. 
\begin{equation*}
\Trop_{\R}\lrp{\mu^{\cA^\vee}_{k}}(m)=\begin{cases} m + \langle d_ke_k, m \rangle  v_k & \text{if } \langle e_k, m \rangle \geq 0,\\
m & \text{if } \langle e_k, m \rangle \leq 0,
\end{cases}
\end{equation*}
for $m \in M^{\circ}$. 
\[
\Trop_{\R}\lrp{\mu^{\cX^\vee}_{k}}(n)=\begin{cases} n + \{n,d_ke_k \} e_k & \text{if } \{  n,e_K \}\geq 0,\\
n & \text{if } \{ n,e_K\} \leq 0,
\end{cases}
\]
for $n \in N$. 
\[
\Trop_{\R}\lrp{\mu^{(\cXe)^\vee}_{k}}(n+H)=\begin{cases} n + \{n,d_ke_k \}e_k + H & \text{if } \{ n, e_k \}\geq 0,\\
n + H& \text{if } \{ n, e_k \} \leq 0,
\end{cases}
\]
for $n + H \in N/H$. 
\[
\Trop_{\R}\lrp{\mu^{(\cA/T_H)^\vee}_{k}} =  \Trop_{\R}\lrp{\mu^{\cA^\vee}_{k}} \mid_{H^\perp}.
\]

\begin{theorem}
\thlabel{thm:intrinsic}
Let $(V,\Phi)$ be a scheme with a cluster structure of type $\cV$ and let $V \subset Y$ be a partial minimal model. Assume that the full Fock--Goncharov conjecture holds for $\cV$ and that there exists a theta function $\tau \in R_1(D')$ such that $\nu^{\Phi}_{\seed}(\tau)$ lies in a linear subset of $\Trop_{\Z}(\cV^\vee)$. If $\Delta_{\nu^{\Phi}_{\seed}}(D',\tau)$ is positive then for every seed $\seed \in \orT$ we have that $\mathfrak{r}_{\seed^\vee}(\Delta_{\mathrm{BL}}(D')-\nu^{\Phi}(\tau))= \Delta_{\nu^{\Phi}_{\seed}}(D',\tau) $.
In particular, for any other seed $\seed'\in \orT $ we have that
\[
\Delta_{\nu_{\seed'}}(D', \tau )= \Trop_{\R}\lrp{\mu^{\cV^\vee}_{\seed,\seed'}}\lrp{\Delta_{\nu_{\seed}}(D', \tau)}.
\] 
\end{theorem}
\begin{proof}
It is enough to treat the case $V= \cV$. We consider the broken line convex hull of 
\[
S=\bigcup_{k\geq 1}\lrc{\dfrac{\nu_{\seed}(f)}{k}-\nu_{\seed}(\tau)\mid f\in R_k(D') \setminus\{0\}} 
\]
in $\Trop_{\R}(\cV^{\vee}_{\seed^\vee})$. 
Since all line segments of $\Trop_{\R}(\cV^{\vee}_{\seed^\vee})$ can be thought of as a segment of a broken line and $\Delta_{\nu_{\seed}}(D', \tau)$ is closed we have that $ \Delta_{\nu_{\seed}}(D', \tau)\subseteq \bconv(S)$. 
By \thref{thm:mainCMN} $\Delta_{\nu_{\seed}}(D', \tau)$ is broken line convex. Since $S\subset \Delta_{\nu_{\seed}}(D', \tau)$ we have the reverse inclusion. 
The last statement follows from the fact that broken line convex sets are preserved by $\Trop_{\R}(\mu^{\cV^{\vee}}_k)$. 
\end{proof}

There is an analogous result for line bundles fitting the framework of \S\ref{sec:universal_torsors}.

\begin{definition}
\label{def:intrinsic_lb}
Let $Y$ be a projective variety such that $\text{Pic}(Y)$ is free of finite rank. Assume $(V, \Phi)$ is a scheme with a cluster structure of type $\cA$ and that $V \subset \UT_Y$ is a partial minimal model with enough theta functions. Let $(p, H)$ have the Picard property (see \thref{k_and_pic}). The {\bf{intrinsic Newton--Okounkov body} associated to a class $[ \lb ]\in \text{Pic}(Y)\cong H^*$} is
\eqn{
\Delta_{\mathrm{BL}}(\lb):= \bconv\Bigg( \bigcup_{k\geq 1} \Bigg(\bigcup_{f \in R_k(\lb)} \frac{1}{k} \NewtT(f) \Bigg)  \Bigg)\subseteq \Trop_{\R}(\cA^{\vee}).
}
\end{definition}

In this case we have the following theorem whose proof is completely analogous to the proof of \thref{thm:intrinsic}. Moreover, it uses the fact that $\nu_{\seed}(\lb)$ is a positive set, as shown in \thref{thm:k_and_pic}.

\begin{theorem}
\thlabel{thm:intrinsic_lb}
Keep the assumptions of Definition \ref{def:intrinsic_lb}.
For every seed $\seed\in \orT$ we have that $\Delta_{\nu^{\Phi}_{\seed}}(\lb)=\mathfrak{r}_{\seed^\vee}(\Delta_{\mathrm{BL}}(\lb))$. In particular, for every $\seed' \in \orT $ we have that
\[
\Delta_{\nu^\Phi_{\seed'}}(\lb )= \lrp{\mu^{\cV^\vee}_{\seed^\vee, \seed'^\vee}}^T(\Delta_{\nu^\Phi_{\seed}}(\lb)).\] 
\end{theorem}
\begin{proof}
We showed in \thref{thm:k_and_pic} that $\Delta_{\nu_{\seed}}(\lb)$ is a positive set. The proof of this result is completely analogous to the proof of \thref{thm:intrinsic}.
\end{proof}

In either situation (divisors or line bundles) we are of course free to compute the intrinsic Newton--Okounkov body as a usual Newton--Okounkov body in any vector space realization of $\Trop_{\R}(\cV^{\vee})$.
However, the intrinsic definition has certain advantages as we now explain.
For simplicity, from now on we concentrate on line bundles as in \thref{thm:intrinsic_lb}; the reader can make the appropriate changes for the case of divisors as in \thref{thm:intrinsic}.
It is often the case that $\Delta_{\mathrm{BL}}(\lb) = \bconv \Big( \bigcup_{k=1}^\ell \frac{1}{k} \nu^{\Phi}\lrp{R_k(\lb)} \Big)$ for some finite $\ell$, meaning in these cases the infinite union reduces to finite union.
Consider such an instance and let $\ell_{\seed}$ be the smallest integer such that $\Delta_{\nu^{\Phi}_{\seed}}(\lb)=\conv \Big( \bigcup_{k=1}^{\ell_{\seed}} \frac{1}{k} \nu^{\Phi}_{\seed}\lrp{R_k(\lb)} \Big)$.  
Then the corresponding $\ell$ for the intrinsic Newton--Okounkov body is at most $\min_{\seed}\lrc{\ell_{\seed}}$.
Moreover, we can give conditions indicating when $\ell$ has been attained. 
We will start with a condition that, after adopting a slightly different perspective on theta functions, becomes tautological.\footnote{This perspective is essentially the {\it{jagged path}} description of theta functions rather than the broken line description.  See for example \cite[Section~3]{GS12}.}  
We will then adapt this condition to give a sufficient criterion that is more likely to be known for a given minimal model (and a known line bundle or Weil divisor).

\begin{proposition}\thlabel{taut}
Let $\lb$ be as in \thref{thm:intrinsic_lb}. Suppose there exists a positive integer $\ell$ such that for all $h>\ell$, each theta function $\tf^V_r$ in $R_h(\lb)$ appears as a summand (with non-zero coefficient) of some product $\tf^V_p \tf^V_q$, where $\tf^V_p \in R_i(\lb)$ and $\tf^V_q \in R_j(\lb)$  for some positive integers $i$ and $j$ with $i+j =h$. 
Then \eqn{
\Delta_{\mathrm{BL}}(\lb) = \bconv\Bigg( \bigcup_{k=1}^{\ell} \Bigg(\bigcup_{f \in R_k(\lb)} \frac{1}{k} \NewtT(f) \Bigg)  \Bigg) .
}
\end{proposition}

\begin{proof}
This is an immediate consequence of results in \cite{CMNcpt}.  We adopt the terminology and conventions of {\it loc. cit.} for this proof.  In particular, we allow non-generic broken lines (see Remark \ref{rem:non-generic_bl}).

Since the structure constant $\alpha(p,q,r)$ is non-zero, there exists a pair of broken lines $\lrp{\gamma_1,\gamma_2}$ with $I(\gamma_1) = p $, $I(\gamma_2) = q $, $\gamma_1(0)=\gamma_2(0) = r$, and $F(\gamma_1)+ F(\gamma_2) = r$.
Then the construction of \cite[\S4]{CMNcpt} yields a broken line segment from $\frac{p}{i}$ to $\frac{q}{j}$ passing through $\frac{r}{h}$. 
As a consequence, we have 
\eqn{\frac{r}{h} \in \bconv\Bigg( \bigcup_{k=1}^{\max(i,j)} \Bigg(\bigcup_{f \in R_k(\lb)} \frac{1}{k} \NewtT(f) \Bigg)  \Bigg) . }
By hypothesis, $R_k(\lb)$ has a basis of theta functions for all $k$, so 
\eqn{
\bconv \Bigg(\bigcup_{f \in R_h(\lb)} \frac{1}{h} \NewtT(f) \Bigg) =  \bconv \lrp{ \frac{r}{h} \mid \tf^V_r \in R_h(\lb)} .
}
We have just seen that each such $\frac{r}{h}$ is contained in 
\eqn{\bconv\Bigg( \bigcup_{k=1}^{h-1} \Bigg(\bigcup_{f \in R_h(\lb)} \frac{1}{h} \NewtT(f) \Bigg)  \Bigg), }
so 
\eqn{
\bconv \Bigg(\bigcup_{f \in R_h(\lb)} \frac{1}{h} \NewtT(f) \Bigg) \subset  \bconv\Bigg( \bigcup_{k=1}^{h-1} \Bigg(\bigcup_{f \in R_k(\lb)} \frac{1}{k} \NewtT(f) \Bigg)  \Bigg).
}
As this holds for all $h>\ell$, we conclude that
\eqn{
\Delta_{\mathrm{BL}}(\lb) = \bconv\Bigg( \bigcup_{k=1}^{\ell} \Bigg(\bigcup_{f \in R_k(\lb)} \frac{1}{k} \NewtT(f) \Bigg)  \Bigg) .
}
\end{proof}

\begin{remark}
In dimension 2, Mandel \cite{Man16} showed that the assumption in \thref{taut} implies that $r=p+q$ in some seed. It is a very interesting problem to determine if this holds for higher dimensions. 
\end{remark}

Note that as we have (by assumption) a theta basis for $R(\lb)$, the condition of \thref{taut} is implied by the following condition:

\noindent
\begin{condition}\thlabel{condition_section ring}
There exists a positve integer $\ell$ such that for all $h>\ell$, the natural map $R_i (\lb) \otimes R_j(\lb) \to R_h (\lb)$ is surjective for some positive integers $i$ and $j$ with $i+j =h$.
\end{condition} 

\begin{remark}\label{rmk:borel weil bott}
The \thref{condition_section ring} is satisfied in our main class of examples coming from representation theory: recall the setting of \thref{exp:full_flag} where line bundles $\mathcal L_\lambda$ of the full flag variety $G/B$ are indexed by dominant weights $\lambda$.
By the Borel--Weil--Bott Theorem the graded pieces $R_i(\mathcal L_\lambda)$ of the section rings of these line bundles satisfy
\[
R_i(\mathcal L_\lambda)\cong  V(i\lambda)^*,
\]
where $V(i\lambda)$ is the irreducible $G$-representation of highest weight $i\lambda$ and $i\ge 0$. 
By work of Baur \cite{Baur_CartanComp} the tensor product $V(i\lambda)\otimes V(j\lambda)$ contains among its irreducible components the unique component of maximal weight, called Cartan component, which is $V((i+j)\lambda)$. 
Hence,
\[
R_i(\mathcal L_\lambda)\otimes R_j(\mathcal{L}_\lambda)\cong V(i\lambda)^*\otimes V(j\lambda)^*\twoheadrightarrow V((i+j)\lambda)^*\cong R_{i+j}(\mathcal L_\lambda).
\]
Although in \thref{exp:full_flag} we only treat the case of $SL_{n+1}(\Bbbk)$ it is worth noticing that the Borel--Weil(--Bott) Theorem holds for semisimple Lie groups and algebraic groups over $\Bbbk$ and Baur's result holds for irreducible representations of connected, simply-connected complex reductive groups.
Notice further that these observations also hold for partial flag varieties, \ie quotient $G/P$ by parabolic subgroups $P\subset G$ as the cohomology of an equivariant line bundles on $G/P$ is equal to the cohomology of its pullback along the natural projection $G/B\twoheadrightarrow G/P$. So the cohomology of the line bundle on $G/P$ can be calculated using the usual Borel--Weil(--Bott) Theorem for $G/B$, by the Leray spectral sequence.
\end{remark}

\section{The case of the Grassmannian}
\label{sec:NO_Grass}

We now consider in detail the case of the Grassmannians. 
Throughout this section we work over the complex numbers, fix two positive integers $k<n$ and let
\[
Y= \Grass_{n-k}(\C^n)
\]
be the corresponding Grassmannian. 
Let $\widetilde{Y} $ be the affine cone of $Y$ in its Pl\"ucker embedding $Y\hookrightarrow \mathbb{P}^{\binom{n}{n-k}-1}$ and $\lb_e$ be the bundle over $Y$ obtained by pullback of $\mathcal{O}(1)$ under this embedding.
By definition, the Pl\"ucker coordinates are a basis for $H^0(Y,\lb_e)$.
It is well known that $\Pic({Y})$ is free of rank one and $[\lb_e]$ is a generator.
Moreover, the universal torsor of $Y$ is 
\[
\UT_{Y}\cong \widetilde{Y}\setminus\{ 0\}
\]
and the action of $T_{\Pic({Y})^*}$ on $\UT_{Y}$ coincides with the diagonal action of $\C^*$. 
Pl\"ucker coordinates are denoted by $p_{J}$ where $J\in \binom{[n]}{n-k}$ is an $n-k$-element subset of $\{1, \dots , n\}$.
Working with cyclic intervals, we let $ D_i=\{ p_{[i+1, i+k]} =0 \}$ and consider the divisor 
\[
D=\bigcup_{i=1}^nD_{i} \subset Y.
\]
For any $i$ the line bundle $\mathcal{O}_Y(D_i)$ is isomorphic to $\lb_e$ and the Weil divisor $\sum_{i=1}^nD_i$ is anticanonical.
We let $\widetilde{D}_i \subset \UT_Y$ be the preimage of $D_i$ under the quotient map $\UT_Y\twoheadrightarrow Y$ and set $\widetilde{D}= \bigcup_{i=1}^n \widetilde{D}_i$.
The divisor $\sum_{i=1}^n\widetilde{D}_i$ is anticanonical.
It follows from the work of Scott \cite{Sco06} that the log Calabi--Yau variety $\UT_Y \setminus \widetilde{D}$ has a cluster structure of type $\cA$ which is skew-symmetric (that is, for its fixed data all $d_i=1$ and $N=N^\circ$) and such that the frozen variables are precisely the Pl\"ucker variables $\{ p_{[i+1,i+k]}\}_{i=1}^{n}$.
This cluster structure is given by an inclusion 
\[
\cA\hookrightarrow \UT_Y  \setminus \widetilde{D}.
\]
Since $\widetilde{D}$ is the locus in $\UT_{Y}$ where the frozen variables vanish,
we have that $ \cA \subset \UT_Y$ is a partial minimal model (see the example below Definition~\ref{def:cv_minimal_model}).
In \cite{MS16} (see also \cite{SW18}), the authors show that $\cA$ has a seed with a maximal green sequence so we can use \cite[Proposition 0.4]{GHKK} to conclude that the full Fock--Goncharov conjecture holds for $\cA$.
Proposition 9.4 in \cite{GHKK} together with \thref{lemm:enough_tf} imply that $\cA \subset \UT_Y$ is a partial minimal model with enough theta functions in the sense of Definition \ref{def:enough_tf}.
In the following subsection we exhibit a cluster ensemble lattice map $p^*$ for $\cA$ such that for $K := \ker(p^*)$, the pair $(p,K)$ has the Picard property in the sense of \thref{k_and_pic} with respect to $\cA \subset \UT_Y$. 
These considerations allow us to apply to all the results of \S\ref{sec:universal_torsors} and \S\ref{sec:NO_bodies}.
In particular, we can think of the Grassmannian as a minimal model for the quotient $\cA/T_K$.

\begin{remark}\label{rmk:open positroid}
The variety $Y \setminus D$ is usually called the open positroid variety. This variety can be endowed with a cluster structure of any of the kinds we consider in this paper: $\cA$, $\cX$, a quotient of $\cA$ or a fibre of $\cX$.
\end{remark}

\subsection{The Picard property}
\label{sec:Pic_property}
In this section we verify that the Picard property (\thref{k_and_pic}) holds for a certain choice of cluster ensemble map and sublattice. This condition is necessary in order to apply \thref{thm:k_and_pic} to the Grassmannian.  

We rely on background from \cite{RW} but recall important notions below.
For background on plabic graphs we refer the reader to \emph{loc. cit.}. 
Recall, that plabic graphs\footnote{To be precise, we are only interested in reduced plabic graphs with trip permutation $\pi_{k,n}$.} are combinatorial objects encoding those seeds whose associated $\cA$-cluster variables are Pl\"ucker coordinates.
To simplify the exposition we do not distinguish between a plabic graph and its associated seed.
Given an index set $J\in\binom{[n]}{n-k}$ we construct a Young diagram $\mu_J$ inside an $(n-k)\times k$ grid inside a rectangle.
Let $w_J$ be the path along edges of the grid from north east  to south west corner whose south steps are in $J$. 
Then $\mu_J$ is the Young diagram (inside the rectangle attached to the north west corner) whose south east border is $w_J$.
Among all plabic graphs there is a particularly symmetric one know as the {\bf rectangles plabic graph} $G_{\Yng(1)}:=G^{\rm rec}_{k,n}$.
The associated cluster variables are naturally indexed by {\it rectangular} Young diagrams (together with the {\it empty} rectangle, denoted by $\varnothing$). 
In what follows we focus on this plabic graph as the initial seed and denote by
\begin{equation}
\label{eq_seed}
\seed_{\Yng(1)}=(e_\varnothing)\cup(e_{i\times j}\mid  1\le i\le n-k, 1\le j\le k),    
\end{equation}
the induced basis of $N=N^\circ\cong \mathbb Z^{k(n-k)+1}$. 
Let $\{f_\varnothing\}\cup\{f_{i\times j}\mid  1\le i\le n-k, 1\le j\le k\}$ denote the corresponding basis of $M^\circ=M$.
We write $N_{\seed}$ respectively $M_\seed$ whenever we think of the lattices together with a choice of basis induced by a seed $\seed$. 
\begin{figure}
\centering
\begin{tikzpicture}[scale=.95]
\node[blue] at (-1.5,10) {$\tiny{\varnothing}$};
\draw[->] (-1.25,9.875) -- (-.25,9.125);
\draw[->,blue,dashed] (7,9.25) to [out=150,in=0] (-1.25,10.125);
\draw[->,blue,dashed] (-.25,4.75) to [out=155,in=-90] (-1.5,9.75);
\node at (0,9) {$\tiny{\yng(1)}$};
\node at (1.5,9) {$\tiny{\yng(2)}$};
\node at (3,9) {$\tiny{\yng(3)}$};
\node at (5,9) {$\tiny{\yng(4)}$};
\node[blue] at (7,9) {$\tiny{\yng(5)}$};
\draw[->] (.25,9) -- (1.125,9);
\draw[->] (1.875,9) -- (2.5,9);
\draw[->] (3.5,9) -- (4.375,9);
\draw[->] (5.625,9) -- (6.25,9);
\draw[->] (0,8.75) -- (0,8.25);
\draw[->] (1.5,8.75) -- (1.5,8.25);
\draw[->] (3,8.75) -- (3,8.25);
\draw[->] (5,8.75) -- (5,8.25);
\draw[->,blue,dashed] (7,8.75) -- (7,8.25);
\draw[<-] (.25,8.75) -- (1.25,8.25);
\draw[<-] (1.875,8.75) -- (2.75,8.25);
\draw[<-] (3.5,8.75) -- (4.375,8.25);
\draw[<-] (5.625,8.75) -- (6.5,8.25);

\node at (0,7.875) {$\tiny{\yng(1,1)}$};
\node at (1.5,7.875) {$\tiny{\yng(2,2)}$};
\node at (3,7.875) {$\tiny{\yng(3,3)}$};
\node at (5,7.875) {$\tiny{\yng(4,4)}$};
\node[blue] at (7,7.875) {$\tiny{\yng(5,5)}$};
\draw[->] (.25,7.875) -- (1.125,7.875);
\draw[->] (1.875,7.875) -- (2.5,7.875);
\draw[->] (3.5,7.875) -- (4.375,7.875);
\draw[->] (5.625,7.875) -- (6.25,7.875);

\draw[->] (0,7.5) -- (0,7);
\draw[->] (1.5,7.5) -- (1.5,7);
\draw[->] (3,7.5) -- (3,7);
\draw[->] (5,7.5) -- (5,7);
\draw[->,blue,dashed] (7,7.5) -- (7,7);
\draw[<-] (.25,7.5) -- (1.125,7);
\draw[<-] (1.875,7.5) -- (2.75,7);
\draw[<-] (3.5,7.5) -- (4.375,7);
\draw[<-] (5.625,7.5) -- (6.5,7);

\node at (0,6.5) {$\tiny{\yng(1,1,1)}$};
\node at (1.5,6.5) {$\tiny{\yng(2,2,2)}$};
\node at (3,6.5) {$\tiny{\yng(3,3,3)}$};
\node at (5,6.5) {$\tiny{\yng(4,4,4)}$};
\node[blue] at (7,6.5) {$\tiny{\yng(5,5,5)}$};
\draw[->] (.25,6.5) -- (1.125,6.5);
\draw[->] (1.875,6.5) -- (2.5,6.5);
\draw[->] (3.5,6.5) -- (4.375,6.5);
\draw[->] (5.625,6.5) -- (6.25,6.5);
\draw[->] (0,6) -- (0,5.375);
\draw[->] (1.5,6) -- (1.5,5.375);
\draw[->] (3,6) -- (3,5.375);
\draw[->] (5,6) -- (5,5.375);
\draw[->,blue,dashed] (7,6) -- (7,5.375);
\draw[<-] (.25,6) -- (1.125,5.375);
\draw[<-] (1.875,6) -- (2.75,5.375);
\draw[<-] (3.5,6) -- (4.375,5.375);
\draw[<-] (5.625,6) -- (6.5,5.375);

\node[blue] at (0,4.75) {$\tiny{\yng(1,1,1,1)}$};
\node[blue] at (1.5,4.75) {$\tiny{\yng(2,2,2,2)}$};
\node[blue] at (3,4.75) {$\tiny{\yng(3,3,3,3)}$};
\node[blue] at (5,4.75) {$\tiny{\yng(4,4,4,4)}$};
\node[blue] at (7,4.75) {$\tiny{\yng(5,5,5,5)}$};
\draw[->,blue,dashed] (.25,4.75) -- (1.125,4.75);
\draw[->,blue,dashed] (1.875,4.75) -- (2.5,4.75);
\draw[->,blue,dashed] (3.5,4.75) -- (4.375,4.75);
\draw[->,blue,dashed] (5.625,4.75) -- (6.25,4.75);
\end{tikzpicture}
\caption{The quiver of the plabic graph $G^{\rm rec}_{5,9}$ forming the initial seed for $\mathcal A\subset \widetilde{Y}=\widetilde{\text{Gr}}_{4}(\mathbb C^9)$ with the frozen arrows determining the cluster ensemble map $p^*$.}
\label{fig:quivGrec59}
\end{figure}
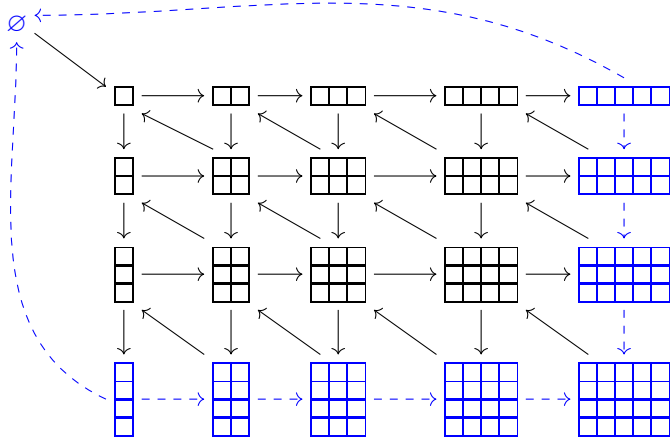

We start by defining a lattice map 
\[
\psi:N_{\seed_{\Yng(1)}} \to  M_{\seed_{\Yng(1)}}
\]
which is given with respect to the bases induced by $s_{\Yng(1)}$ as follows:
for $i\times j$ a mutable vertex and $a\times b$ with either $a=n-k$ or $b=k$ a frozen vertex we define
\begin{eqnarray*}
    e_{i\times j} &\mapsto& f_{(i-1)\times(j-1)} - f_{(i-1)\times j} + f_{i\times(j+1)} - f_{(i+1)\times (j+1)} + f_{(i+1)\times j} - f_{i\times (j-1)} \\
    e_{a\times b} &\mapsto& f_{a\times b} - f_{(a-1)\times b} + f_{(a-1)\times(b-1)} - f_{a\times (b-1)} \\
    e_{\varnothing} &\mapsto& f_{\varnothing} - f_{1\times k} + f_{1\times 1} - f_{(n-k)\times 1}
\end{eqnarray*}
with the convention that $f_{0\times j}=f_{i\times 0}=0$ whenever $i,j\not =0$ and $f_{0\times 0}=f_{\varnothing}$. 
We may present the map pictorially by recording the coefficient of the basis element $e_{i\times j}$ in the $i\times j$'th position of the grid (with an extra position $0\times 0$ representing  the vertex $\varnothing$). 
\begin{eqnarray}\label{eq:pictorial p*}
\begin{tikzpicture}[scale=.4]
\node at (-3,0){$e_{i\times j}$};
\draw[dashed,opacity=.5] (-1.5,.5) -- (1.5,.5);
\draw[dashed,opacity=.5] (-1.5,-0.5) -- (1.5,-0.5);
\draw[dashed,opacity=.5] (-0.5,-1.5) -- (-0.5,1.5);
\draw[dashed,opacity=.5] (0.5,-1.5) -- (0.5,1.5);
\node[opacity=.5] at (-1,-1) {\small $0$};
\node[opacity=.5] at (-1,0) {\small$0$};
\node[opacity=.5] at (-1,1) {\small$0$};
\node[opacity=.5] at (0,-1) {\small$0$};
\node at (0,0) {\small$1$};
\node[opacity=.5] at (0,1) {\small$0$};
\node[opacity=.5] at (1,-1) {\small$0$};
\node[opacity=.5] at (1,0) {\small$0$};
\node[opacity=.5] at (1,1) {\small$0$};

\node at (2.5,0) {$\mapsto$};

\begin{scope}[xshift=5cm]
\draw[dashed,opacity=.5] (-1.5,.5) -- (1.5,.5);
\draw[dashed,opacity=.5] (-1.5,-0.5) -- (1.5,-0.5);
\draw[dashed,opacity=.5] (-0.5,-1.5) -- (-0.5,1.5);
\draw[dashed,opacity=.5] (0.5,-1.5) -- (0.5,1.5);
\node[opacity=.5]  at (-1,-1) {\small$0$};
\node at (-1,0) {\small$-1$};
\node at (-1,1) {\small$1$};
\node at (0,-1) {\small$1$};
\node[opacity=.5]  at (0,0) {\small$0$};
\node at (0,1) {\small$-1$};
\node at (1,-1) {\small$-1$};
\node at (1,0) {\small$1$};
\node[opacity=.5]  at (1,1) {\small$0$};
\end{scope}

\begin{scope}[xshift=15cm]
\node at (-3,0){$e_{a\times b}$};
\draw[dashed,opacity=.5] (-1.5,.5) -- (1.5,.5);
\draw[dashed,opacity=.5] (-1.5,-0.5) -- (1.5,-0.5);
\draw[dashed,opacity=.5] (-0.5,-1.5) -- (-0.5,1.5);
\draw[dashed,opacity=.5] (0.5,-1.5) -- (0.5,1.5);
\node[opacity=.5] at (-1,-1) {\small $0$};
\node[opacity=.5] at (-1,0) {\small$0$};
\node[opacity=.5] at (-1,1) {\small$0$};
\node[opacity=.5] at (0,-1) {\small$0$};
\node at (0,0) {\small$1$};
\node[opacity=.5] at (0,1) {\small$0$};
\node[opacity=.5] at (1,-1) {\small$0$};
\node[opacity=.5] at (1,0) {\small$0$};
\node[opacity=.5] at (1,1) {\small$0$};

\node at (2.5,0) {$\mapsto$};

\begin{scope}[xshift=5cm]
\draw[dashed,opacity=.5] (-1.5,.5) -- (1.5,.5);
\draw[dashed,opacity=.5] (-1.5,-0.5) -- (1.5,-0.5);
\draw[dashed,opacity=.5] (-0.5,-1.5) -- (-0.5,1.5);
\draw[dashed,opacity=.5] (0.5,-1.5) -- (0.5,1.5);
\node[opacity=.5]  at (-1,-1) {\small$0$};
\node at (-1,0) {\small$-1$};
\node at (-1,1) {\small$1$};
\node[opacity=.5]  at (0,-1) {\small$0$};
\node at (0,0) {\small$1$};
\node at (0,1) {\small$-1$};
\node[opacity=.5]  at (1,-1) {\small$0$};
\node[opacity=.5]  at (1,0) {\small$0$};
\node[opacity=.5]  at (1,1) {\small$0$};
\end{scope}
\end{scope}
\end{tikzpicture}
\end{eqnarray}
All entries in the grid above \emph{not} corresponding to vertices in the particular case considered should simply be neglected.
A straightforward computation reveals the following

\begin{proposition}\thlabel{prop:-p* dual}
We have $\ker(\psi)=\langle{(1,1,\dots,1)}\rangle=K_{\seed_{\Yng(1)}}$ and $\psi(N_{\seed_{\Yng(1)}})={(1,1,\dots,1)}^\perp$. So, the induced map $\psi:N/K\to K^\perp$ is a lattice isomorphism. 
\end{proposition}

In fact, $\psi$ defines a cluster ensemble lattice map (Definition~\ref{def:p-star}), so we obtain
\begin{eqnarray}\label{eq:p-map Gr}
    p:\cA\to\cX, \quad \text{determined by }\quad  (p\vert_{\cA_{\seed_{\Yng(1)}}})^*=\psi.
\end{eqnarray}
There is a combinatorial way to obtain the map $\psi$ by introducing \emph{frozen arrows} to the quiver of the initial seed to \emph{close cycles} involving frozen vertices (see Figure~\ref{fig:quivGrec59}).
These arrows are used to determine the submatrix denoted by $*$ in \eqref{eq:Mp*}.

As a direct consequence of \eqref{eq:p-map Gr} and \thref{prop:-p* dual} we observe that the action of $T_K$ on $\cA$ coincides with the $\C^*$-action (of simultaneously scaling Pl\"ucker coordinates) on $\UT_Y$ restricted to $\cA$.
In particular:
\begin{corollary}
    The Picard property holds for $(p,K)$ with respect to $\mathcal A\hookrightarrow\UT_Y$.
\end{corollary}

\subsection{Valuations and Newton--Okounkov bodies}
\label{sec:GHKK_and_RW}
This subsection is the core of our application to the Grassmannian. 
We show in Theorem~\ref{thm: val and gv} that certain Newton--Okounkov bodies as they appear in \thref{thm:k_and_pic} (see also Remark~\ref{rem:comparing_NO_bodies}) are unimodularly equivalent to Newton--Okounkov bodies of Rietsch--Williams. 
We first introduce the combinatorics that govern Rietsch--Williams' flow valuation and the ${\bf g}$-vector valuation in this case.

\subsubsection{The flow valuation}
Based on Postnokiv's \emph{boundary measurement map} for plabic networks \cite[\S11]{Pos06} Rietsch--Williams associate a {\bf flow valuation} \cite[Definition 8.1]{RW} to every plabic graph $G$ or more generally every seed $\seed$ making use of the $\cX$-type cluster structure on the Grassmannian
We denote it by 
\[
\val_\seed:\mathbb C(Y)\setminus \{0\}\to \mathbb Z^{(n-k)\times k}.
\]
The valuation is defined as the multidegree of the lowest degree summand (with respect a fixed graded lexicographic order) on Laurent polynomials in $\cX$ variables and then extended to rational functions in the natural way.
The lattice is of dimension $(n-k)k$ (as apposed to $(n-k)k+1$ which is the number of vertices), as the the variable corresponding to $\varnothing$ never appears (more details below in \S\ref{sec:NO_Grass_equal}).
Notice that it therefore coincides with our definition of a {\bf c}-vector valuation for cluster $\cX$ varieties (Corollary~\ref{cor:gv on midX}).
For $G=G_{\Yng(1)}$ we simply write $\val_G=\val_{\Yng(1)}$.
The flow valuation  with respect to the rectangles plabic graph can be computed in a particularly explicit way as Rietsch--Williams show in \cite[\S14]{RW}.
We briefly summarize some of their findings.

\begin{proposition}\cite[Proposition 14.4 and Figure 18]{RW}\thlabel{prop:val grec}
For $J\in\binom{[n]}{n-k}$, the valuation $\val_{\Yng(1)}(p_J)$ can be represented by a \emph{GT tableau} (defined as follows, see \cite[\S14]{RW}) of size $(n-k)\times k$ whose $(i\times j)^{\text{th}}$ entry represents the coefficient of the corresponding basis element.
The entries of the GT tableau are obtained as in four steps:
\begin{itemize}
    \item[\bf Step 1:] draw the Young diagram $\mu_J$ whose south border is the path $w_J$ associated to $J$ in the $(n-k)\times k$-rectangle; 
    \item[\bf Step 2:] draw another copy of $w_{J}$ shifted by {\it one step south} and {\it one step east} (this implies that some steps of the new path $w_J^1$ lie outside of the $(n-k)\times k$-rectangle);
    \item[\bf Step 3:] continue repeating Step 2 until the new copy of $w_J$ lies \emph{entirely} outside of the $(n-k)\times k$-rectangle;
    \item[\bf Step 4:] lastly, place an $i$ inside every box (that is part of the $(n-k)\times k$-rectangle) in between the paths $w_{J}^{i-1}$ and $w_{J}^i$.
\end{itemize}
All other boxes are filled with zeros.
\end{proposition}
Rietsch--Williams compute the Newton--Okounkov bodies associated to this valuation. In our notation they are of form $\Delta_{\val_{\Yng(1)}}(D_{n-k},p_{(n-k)\times k})$, where $p_{(n-k)\times k}=p_{[1,n-k]}$ is the Pl\"ucker coordinate (and hence section of $\mathcal L_e$) associated to the frozen vertex $(n-k)\times k$.

\begin{example}\label{exp:Grec6,13}
    The procedure of \thref{prop:val grec} is depicted in Figure~\ref{fig:val grec} for $J=\{3,4,7,9,11,12\}\subset [13]$. 
\end{example}

\begin{figure}
    \centering
\begin{tikzpicture}[scale=.4]
\node[left] at (-1,3) {$\val_{\Yng(1)}(p_J)=$};
\draw (0,0) -- (0,6) -- (5,6);
\draw[thick,magenta] (7,6) -- (5,6) -- (5,4) -- (3,4) -- (3,3) -- (2,3) -- (2,2) -- (1,2) -- (1,0) -- (0,0);
\node at (2,4.5) {$\mu_J$};
\draw (7,6) -- (7,0) -- (1,0);
\draw[magenta,opacity=.4,thick] (2,0) -- (2,-1) -- (1,-1);
\draw[magenta,opacity=.4,thick] (2,0) -- (2,1) -- (3,1) -- (3,2) -- (4,2) -- (4,3) -- (6,3) -- (6,5) -- (7,5);
\draw[magenta,opacity=.4,thick] (7,5) -- (8,5);
\draw[magenta,opacity=.4,thick] (2,-2) -- (3,-2) -- (3,0) -- (4,0) -- (4,1) -- (5,1) -- (5,2) -- (7,2) -- (7,4) -- (9,4);
\draw[magenta,opacity=.4,thick] (3,-3) -- (4,-3) -- (4,-1) -- (5,-1) -- (5,0) -- (6,0) -- (6,1) -- (8,1) -- (8,3) -- (10,3);
\draw[dashed,opacity=.5] (0,0) -- (3.5,-3.5);
\draw[dashed,opacity=.5] (7,6) -- (10.5,2.5);
\draw[opacity=.4,dashed] (5,6) -- (7,6);
\draw[opacity=.4,dashed] (5,5) -- (7,5);
\draw[opacity=.4,dashed] (5,4) -- (7,4);
\draw[opacity=.4,dashed] (3,3) -- (7,3);
\draw[opacity=.4,dashed] (2,2) -- (7,2);
\draw[opacity=.4,dashed] (1,1) -- (7,1);
\draw[opacity=.4,dashed] (6,6) -- (6,0);
\draw[opacity=.4,dashed] (5,5) -- (5,0);
\draw[opacity=.4,dashed] (4,4) -- (4,0);
\draw[opacity=.4,dashed] (3,3) -- (3,0);
\draw[opacity=.4,dashed] (2,2) -- (2,0);
\draw[opacity=.4,dashed] (1,1) -- (1,0); 
\node at (1.5,.5) {1};
\node at (1.5,1.5) {1};
\node at (2.5,.5) {2};
\node at (2.5,1.5) {1};
\node at (2.5,2.5) {1};
\node at (3.5,.5) {2};
\node at (3.5,1.5) {2};
\node at (3.5,2.5) {1};
\node at (3.5,3.5) {1};
\node at (4.5,.5) {3};
\node at (4.5,1.5) {2};
\node at (4.5,2.5) {2};
\node at (4.5,3.5) {1};
\node at (5.5,.5) {3};
\node at (5.5,1.5) {3};
\node at (5.5,2.5) {2};
\node at (5.5,3.5) {1};
\node at (5.5,4.5) {1};
\node at (5.5,5.5) {1};
\node at (6.5,.5) {4};
\node at (6.5,1.5) {3};
\node at (6.5,2.5) {2};
\node at (6.5,3.5) {2};
\node at (6.5,4.5) {2};
\node at (6.5,5.5) {1};

\node at (11.5,3) {${\longmapsto}$};
\node at (11.5,3.75) {\small $-\psi$};
 
\begin{scope}[xshift=14cm]
\draw (0,0) -- (0,6) -- (5,6) -- (5,4) -- (3,4) -- (3,3) -- (2,3) -- (2,2) -- (1,2) -- (1,0) -- (0,0);
\draw (5,6) -- (7,6) -- (7,0) -- (1,0);
\draw[opacity=.4,dashed] (0,6) -- (7,6);
\draw[opacity=.4,dashed] (0,5) -- (7,5);
\draw[opacity=.4,dashed] (0,4) -- (7,4);
\draw[opacity=.4,dashed] (0,3) -- (7,3);
\draw[opacity=.4,dashed] (0,2) -- (7,2);
\draw[opacity=.4,dashed] (0,1) -- (7,1);
\draw[opacity=.4,dashed] (6,6) -- (6,0);
\draw[opacity=.4,dashed] (5,6) -- (5,0);
\draw[opacity=.4,dashed] (4,6) -- (4,0);
\draw[opacity=.4,dashed] (3,6) -- (3,0);
\draw[opacity=.4,dashed] (2,6) -- (2,0);
\draw[opacity=.4,dashed] (1,6) -- (1,0); 
\node at (.5,.5) {$1$};
\node[opacity=.5] at (.5,1.5) {$0$};
\node[opacity=.5] at (.5,3.5) {$0$};
\node[opacity=.5] at (.5,4.5) {$0$};
\node[opacity=.5] at (.5,5.5) {$0$};
\node at (.5,2.5) {$-1$};
\node at (1.5,2.5) {$1$};
\node at (1.5,3.5) {$-1$};
\node[opacity=.5] at (1.5,.5) {$0$};
\node[opacity=.5] at (1.5,1.5) {$0$};
\node[opacity=.5] at (1.5,4.5) {$0$};
\node[opacity=.5] at (1.5,5.5) {$0$};
\node at (2.5,3.5) {$1$};
\node[opacity=.5] at (2.5,.5) {$0$};
\node[opacity=.5] at (2.5,1.5) {$0$};
\node[opacity=.5] at (2.5,2.5) {$0$};
\node[opacity=.5] at (2.5,5.5) {$0$};
\node at (2.5,4.5) {$-1$};
\node[opacity=.5] at (3.5,.5) {$0$};
\node[opacity=.5] at (3.5,1.5) {$0$};
\node[opacity=.5] at (3.5,2.5) {$0$};
\node[opacity=.5] at (3.5,3.5) {$0$};
\node[opacity=.5] at (3.5,4.5) {$0$};
\node[opacity=.5] at (3.5,5.5) {$0$};
\node at (4.5,4.5) {$1$};
\node[opacity=.5] at (4.5,.5) {$0$};
\node[opacity=.5] at (4.5,1.5) {$0$};
\node[opacity=.5] at (4.5,2.5) {$0$};
\node[opacity=.5] at (4.5,3.5) {$0$};
\node[opacity=.5] at (4.5,5.5) {$0$};
\node[opacity=.5] at (5.5,.5) {$0$};
\node[opacity=.5] at (5.5,1.5) {$0$};
\node[opacity=.5] at (5.5,2.5) {$0$};
\node[opacity=.5] at (5.5,3.5) {$0$};
\node[opacity=.5] at (5.5,4.5) {$0$};
\node[opacity=.5] at (5.5,5.5) {$0$};
\node at (6.5,.5) {$-1$};
\node[opacity=.5] at (6.5,1.5) {$0$};
\node[opacity=.5] at (6.5,2.5) {$0$};
\node[opacity=.5] at (6.5,3.5) {$0$};
\node[opacity=.5] at (6.5,4.5) {$0$};
\node[opacity=.5] at (6.5,5.5) {$0$};

\node[right] at (7.5,3) {$=\bar{\bf g}_{{\Yng(1)}}(p_{J})$};
\end{scope}
\end{tikzpicture}
  \caption{
  On the left: the pictorial representation of the GT tableau for $J=\{3,4,7,9,11,12\} \subset [13]$ (\thref{prop:val grec}).  The south steps of the path cutting out the Young diagram $\mu_J$ correspond to indices in $J$. 
  On the right, we depict its image under $-\psi$ which coincides with the ${\bf g}$-vector of $p_J$ up to homogenization, see \eqref{eq:g-vectors for Grec} and \eqref{eq:homogenized g vector op}.
 }
    \label{fig:val grec}
\end{figure}

\subsubsection{A combinatorial description of ${\bf g}$-vectors}\label{sec:g-vects}
In this subsection we consider the cluster variety $\cA^{\rm op}$ whose initial quiver is obtained by opposing the initial quiver for $\cA$. It is well known that $\cA$ and $\cA^{\rm op}$ are isomorphic (in general opposing the quiver gives rise to isomorphic cluster $ \cA$-varieties). 
We also have a partial minimal model $\cA^{\op} \hookrightarrow \UT_Y$.
We write $\seed_{\Yng(1)}^{\rm op}$ to denote the seed $\seed_{\Yng(1)}$ of equation \eqref{eq_seed} thought of as the initial seed for $\cA^{\rm op}$.
Notice that $-\psi$ determines a cluster ensemble map $p^{\rm op}:\cA^{\rm op} \to \cX^{\rm op}$
so that the Picard property holds for $(p^{\rm op},K)$ with respect to $\cA^{\rm op}\hookrightarrow \text{UT}_Y$.

In this setting an explicit combinatorial formula to compute {\bf g}-vectors of Pl\"ucker coordinates can be deduced from the categorification of the Grassmannian cluster algebra developed in \cite{JKS16,BKM16}.
We learned about it from Bernhard Keller in private email communication.
The below formula describes ${\bf g}$-vectors with respect to the seed $\seed_{\Yng(1)}^{\rm op}$ for the cluster variety $\cA^{\rm op}$ which we think of as another cluster structure on $\UT_{Y}$.

\begin{corollary}\thlabel{cor:gv Grec}
(Hook formula for {\bf g}-vectors)
Consider the seed $\seed_{\Yng(1)}^{\rm op}$ and $J \in \binom{[n]}{n-k}$. We let $i_1\times j_1,\dots ,i_s \times j_s $ be the rectangles corresponding to the turning points in the path $w_J$ that cuts out $\mu_J$ inside the $(n-k)\times k$-rectangle.
Then 
\begin{eqnarray}\label{eq:g-vectors for Grec}
{\bf g}_{\Yng(1)^{\rm op}}(p_J):={\bf g}^{\cA^{\rm op}}_{\seed_{\Yng(1)}^{\rm op}}(p_J)= \sum_{p=1}^{s}f_{i_{p}\times j_{p}}-f_{i_{p}\times j_{p+1}},
\end{eqnarray}
where we set $f_{i_s\times j_{s+1}}:=0$.
\end{corollary}

\begin{example}
The 
Consider $n-k=4$, $n=9$, and $J=\{2,4,6,7\}$. We have that $\mu_{J}=\Yng(4,3,2,2)$ and by \thref{cor:gv Grec}
\[
{\bf g}_{\Yng(1)^{\rm op}}\lrp{p_{\Yngs(4,3,2,2)}}= f_{\Yngs(4)}- f_{\Yngs(3)}+f_{\Yngs(3,3)}-f_{\Yngs(2,2)}+f_{\Yngs(2,2,2,2)}.
\]
\end{example}

\subsubsection{Equality of the Newton--Okounkov bodies}\label{sec:NO_Grass_equal}
 The aim of this section is to identify the Newton--Okounkov bodies of flow valuations with Newton--Okounkov bodies of {\bf g}-vector valuations for $\cA^{\rm op}$. 
We use a particular cluster ensemble lattice map for the identification and work in the initial seed $\seed^{\rm op}_{\Yng(1)}$ (whose quiver is opposite to the quiver depicted in Figure~\ref{fig:quivGrec59} for $n=9,k=5$).

We think of the open positroid variety inside $Y=\text{Gr}_{n-k}(\C^n)$ as the quotient of the cluster variety $\cA^{\rm op}$ by the torus $T_{K}$.
We choose a section 
\[
\sigma: N/K \to N \quad \text{with image} \quad N  \cap f_{\varnothing}^\perp;
\]
that is, a coset $n\mod K$ is sent to its unique representative satisfying $\langle n,f_{\varnothing}\rangle=0$.
It is not hard to see that $\sigma$ induces an isomorphism between the rings of rational functions $\mathbb C(T_{M/\langle f_{\varnothing}\rangle})$ and $\mathbb C(T_{K^\perp})$ that commutes with cluster $\mathcal X$ mutation.
We use $\sigma$ to realize $\Trop_{\Z}(({\cX_{\bf 1}^\vee})_{\seed^\vee})=\Trop_{\Z}((\cA^{\rm op}/T_K)_{\seed^{\rm op}})= N/K$ inside $ \Trop_\Z(\cA^{\rm op}_{\seed^{\rm op}})=N =\Trop_\Z(\cX^\vee_{\seed^\vee})$ for every seed.
Moreover, the dual of $\sigma$ induces an isomorphism of lattices
\[
\sigma^*:M/\langle f_{\varnothing}\rangle \to K^\perp.
\]
Notice that $T_{K^\perp}=\pi^{-1}(\bf 1)$ where $\pi:T_M\to T_{K^*}$ is the restriction of $\cX\to T_{K^*}$ to a cluster chart.
As alluded to above, we obtain an isomorphism of cluster $\cX$-varieties
\[
\sigma^*:\cX_{\setminus \varnothing}\to \cX_{\bf 1},
\]
where $\cX_{\setminus \varnothing}$ is the $\cX$-variety associated with the initial data obtained by deleting the index $\varnothing$ upon realizing $M/\langle f_{\varnothing}\rangle$ as $\langle f_{i\times j}\mid 1\le i\le n-k,1\le j\le k\rangle \subset M$.
Given a seed $\seed$ we denote the corresponding seed of $\cX_{\setminus\varnothing}$ by $\seed_{\setminus\varnothing}$.
In particular, we have
$\Trop_{\Z}((\mathcal X^\vee_{\setminus\varnothing})_{\seed^\vee_{\setminus \varnothing}})=N_{\seed}\cap f_{\varnothing}^\perp$.
The flow valuation is defined on ring of rational functions on the positroid variety which coincides with
\[
\mathbb C(\cX_{\setminus \varnothing}) \cong \mathbb C(x_{i\times j}:1\le i\le n-k,1\le j\le k).
\]

The next result follows from the preceding discussion and Corollary~\ref{cor:gv on midX}.

\begin{proposition}\label{prop:flow is gv for X}
For every choice of seed $\seed$ the diagram commutes:
\[
\xymatrix{
\mathbb C(\cX_{\bf 1})\setminus \{0\} \ar[d]_{(\sigma^*)^*}\ar[r]^{\cv_{\seed}} & N/K\ar[d]^{\sigma}\\
\mathbb C(\cX_{\setminus \varnothing})\setminus \{0\} \ar[r]_{\val_{\seed}} & N_{\seed}\cap f_{\varnothing}^\perp.
}
\]
\end{proposition}
The flow valuation is a ${\bf c}$-vector valuation for the variety $\cX_{\setminus \varnothing}$ as both are defined by picking the lowest degree exponent of a Laurent polynomial with respect to the same order. That is:
\[
\val_{\seed}= \cv^{\cX_{\setminus \varnothing}}_{\seed\setminus\varnothing}.
\]
Alternatively, in light of Proposition \ref{prop:flow is gv for X} we may think of the flow valuation as a ${\bf c}$-vector valuation for $\cX_{\bf 1}$.
Our aim now is to identify the images of $\val_\seed$ with 
those of a $\bf g$-vector valuation for $\cA^{\rm op}$, or more precisely a ${\bf g}$-vector valuation for $\cA^{\rm op}/T_{K}$.
To avoid confusion we introduce the following notation
\begin{eqnarray}\label{eq:homogenized g vector op}
\bar {\bf g}_{\Yng(1)^{\rm op}}:R\setminus \{0\}\longrightarrow K^\perp\cong M_{\seed}/\langle f_{(n-k)\times k}\rangle 
\end{eqnarray}
defined for a homogeneous element $h\in R_q\setminus \{0\}$ by
\[
h\longmapsto 
{\bf g}_{\Yng(1)^{\rm op}}\lrp{\frac{h}{
p_{(n-k)\times k}^q}},
\]
where ${\bf g}_{\Yng(1)^{\rm op}}(p_{(n-k)\times k})=f_{(n-k)\times k}$.
Notice that $\bar {\bf g}_{\Yng(1)^{\rm op}}$ is the restriction of ${\bf g}_{\Yng(1)^{\rm op}}: \mathbb C(Y)\setminus \{0\}\to M$ to the section ring $R\hookrightarrow \mathbb C(Y)$ where the embedding is defined by $R_q\ni h\mapsto h/p_{(n-k)\times k}^q$.

\begin{theorem}\thlabel{thm: val and gv}
\begin{samepage}
For every $J\in \binom{[n]}{n-k}$ we have
\begin{eqnarray}
-\psi(\val_{\Yng(1)}(p_J))=\bar {\bf g}_{\Yng(1)^{\rm op}}(p_{J}). 
\end{eqnarray}
In particular, 
the Newton--Okounkov bodies $\Delta_{\val_{\Yng(1)}}(D_{n-k},p_{(n-k)\times k})$ and $\Delta_{\bar{\bf g}_{\Yng(1)^{\rm op}}}(D_{n-k},p_{(n-k)\times k})$ are unimodularly equivalent with lattice isomorphism given by $-\psi$. 
\end{samepage}
\end{theorem}

\begin{proof}
We prove the claim in several steps. 
First, we need to describe $\val_{\Yng(1)^{\rm op}}(p_J)$. 
Fortunately, this is straightforward using \thref{prop:val grec}.
Let us analyze the image of the \emph{$i$-strip}, \emph{i.e.} the image of the elements of form $-ie_{a\times b}$ corresponding to a box in position $a\times b$ of the grid lying between the path $w_J^i$ and $w_J^{i-1}$. 
We deduce
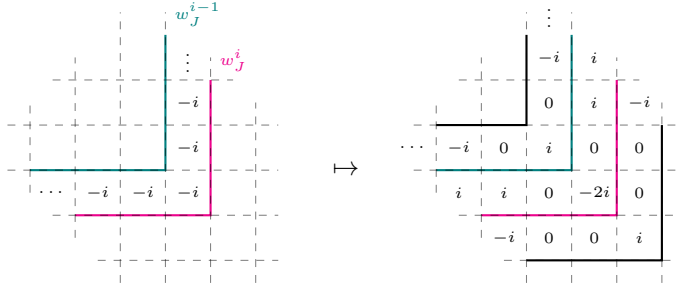
\begin{center}
	\begin{tikzpicture}[scale=.6]
\draw[thick,teal] (1,3) -- (4,3) -- (4,6);
\draw[thick,magenta] (2,2) -- (5,2) -- (5,5);
\node at (4.5,5.5) {\tiny $\vdots$};
\node at (1.5,2.5) {\tiny $\cdots$};
\node at (2.5,2.5) {\tiny $-i$};
\node at (3.5,2.5) {\tiny $-i$};
\node at (4.5,2.5) {\tiny $-i$};
\node at (4.5,3.5) {\tiny $-i$};
\node at (4.5,4.5) {\tiny $-i$};
\draw[dashed,opacity=.5] (1,2.5) -- (1,4.5);
\draw[dashed,opacity=.5] (2,1.5) -- (2,5.5);
\draw[dashed,opacity=.5] (3,.5) -- (3,6.5);
\draw[dashed,opacity=.5] (4,.5) -- (4,6.5);
\draw[dashed,opacity=.5] (5,.5) -- (5,5.5);
\draw[dashed,opacity=.5] (6,.5) -- (6,4.5);
\draw[dashed,opacity=.5] (1.5,5) -- (5.5,5);
\draw[dashed,opacity=.5] (.5,4) -- (6.5,4);
\draw[dashed,opacity=.5] (.5,3) -- (6.5,3);
\draw[dashed,opacity=.5] (1.5,2) -- (6.5,2);
\draw[dashed,opacity=.5] (2.5,1) -- (6.5,1);
\node[above right,magenta] at (5,5) {\tiny $w_J^i$};
\node[above right,teal] at (4,6) {\tiny $w_J^{i-1}$};

\node at (8,3) {$\mapsto$};

\begin{scope}[xshift=9 cm]
\node at (.5,3.5) {\tiny $\cdots$};
\node at (1.5,2.5) {\tiny $i$};
\node at (2.5,1.5) {\tiny $-i$};
\node at (3.5,6.5) {\tiny $\vdots$};
\node at (4.5,5.5) {\tiny $i$};
\node at (5.5,4.5) {\tiny $-i$};
\draw[thick] (1,4) -- (3,4) -- (3,6);
\draw[thick,teal] (1,3) -- (4,3) -- (4,6);
\draw[thick,magenta] (2,2) -- (5,2) -- (5,5);
\draw[thick] (3,1) -- (6,1) -- (6,4);
\node at (1.5,3.5) {\tiny $-i$};
\node at (2.5,3.5) {\tiny $0$};
\node at (3.5,3.5) {\tiny $i$};
\node at (3.5,4.5) {\tiny $0$};
\node at (3.5,5.5) {\tiny $-i$};
\node at (2.5,2.5) {\tiny $i$};
\node at (3.5,2.5) {\tiny $0$};
\node at (4.5,2.5) {\tiny $-2i$};
\node at (4.5,3.5) {\tiny $0$};
\node at (4.5,4.5) {\tiny $i$};
\node at (3.5,1.5) {\tiny $0$};
\node at (4.5,1.5) {\tiny $0$};
\node at (5.5,1.5) {\tiny $i$};
\node at (5.5,2.5) {\tiny $0$};
\node at (5.5,3.5) {\tiny $0$};
\draw[dashed,opacity=.5] (1,2.5) -- (1,4.5);
\draw[dashed,opacity=.5] (2,1.5) -- (2,5.5);
\draw[dashed,opacity=.5] (3,.5) -- (3,6.5);
\draw[dashed,opacity=.5] (4,.5) -- (4,6.5);
\draw[dashed,opacity=.5] (5,.5) -- (5,5.5);
\draw[dashed,opacity=.5] (6,.5) -- (6,4.5);
\draw[dashed,opacity=.5] (1.5,5) -- (5.5,5);
\draw[dashed,opacity=.5] (.5,4) -- (6.5,4);
\draw[dashed,opacity=.5] (.5,3) -- (6.5,3);
\draw[dashed,opacity=.5] (1.5,2) -- (6.5,2);
\draw[dashed,opacity=.5] (2.5,1) -- (6.5,1);
\end{scope}
\end{tikzpicture}
\end{center}
Notice that unless $i=1$ all non zero entries in the picture cancel with the images of the $(i-1)$- and the $(i+1)$-strips.
When $i=1$ however, the entry $i=1$ above the path  $w_J^{0}=w_J$ stays.
Hence, for every corner in $w_{J}$ corresponding to a south step followed by a west step $-\psi\left(\val_{\Yng(1)}(p_J)\right)$ has coefficient $1$ for $f_{a\times b}$ where $a\times b$ corresponds to the box whose south east corner coincides with this corner of $w_J$.
The case of a corner in $w_{J}$ corresponding to a west step followed by a south step is very similar, with the only difference that the signs change. 
In particular,
$-\psi\left(\val_{\Yng(1)}(p_J)\right)$ has coefficient $-1$ for $f_{a\times b}$ where $a\times b$ corresponds to the box whose south east corner is adjacent to this corner of $w_J$.

It is left to analyze the parts of $-\psi\left(\val_{\Yng(1)}(p_J)\right)$ corresponding to \emph{frozen} vertices. 
The arguments here are very similar, the only special case being the south east corner of the $(n-k)\times k$-rectangle.
Hence, we restrict our attention to this case and omit the others.

Consider the vertex in position $(n-k)\times k$ and assume in $\val_{\Yng(1)}(p_J)$ the corresponding entry is $i$.
Notice, that coefficient for the vertex $(n-k-1)\times (k-1)$ necessarily is $i-1$.
So, applying $-\psi$ we see that
\begin{center}
	\begin{tikzpicture}[scale=.6]
\draw (.5,0) -- (3,0) -- (3,2.5);
\node at (2.5,.5) {\tiny $-i$};
\node at (1.5,1.5) {\tiny $-i+1$};
\node at (1.5,.5) {\tiny $\cdots$};
\node at (2.5,1.5) {\tiny $\vdots$};

\draw[opacity=.5,dashed] (.5,1) -- (3,1);
\draw[opacity=.5,dashed] (.5,2) -- (3,2);
\draw[opacity=.5,dashed] (1,0) -- (1,2.5);
\draw[opacity=.5,dashed] (2,0) -- (2,2.5);

\node at (4,1.25) {$\mapsto$};

\begin{scope}[xshift=5cm]
\draw (.5,0) -- (3,0) -- (3,2.5);
\node at (2.5,.5) {\tiny $-1$};
\node at (1.5,1.5) {\tiny $-i$};
\node at (1.5,.5) {\tiny $1$};
\node at (2.5,1.5) {\tiny $1$};

\draw[opacity=.5,dashed] (.5,1) -- (3,1);
\draw[opacity=.5,dashed] (.5,2) -- (3,2);
\draw[opacity=.5,dashed] (1,0) -- (1,2.5);
\draw[opacity=.5,dashed] (2,0) -- (2,2.5);
\end{scope}
\end{tikzpicture} 
\end{center}
Observe that the entries $1$ and $-i$ cancel by similar arguments as above.
The only non-zero coefficient in this picture is $-1$ for $f_{(n-k)\times k}$.
Summarizing, we have
\begin{eqnarray*}
-\psi\left(\val_{\Yng(1)}(p_J)\right) &=& \sum_{\begin{smallmatrix}
  \text{south to west}\\
  \text{corners of }w_J
\end{smallmatrix} } f_{a'\times b'} -f_{(n-k)\times k}
-\sum_{\begin{smallmatrix}
  \text{west to south}\\
  \text{corners of }w_J
\end{smallmatrix} } f_{a\times b}\\
&\overset{\text{Equation~\eqref{eq:g-vectors for Grec}}}{=}& {\bf g}_{\Yng(1)^{\rm op}}(p_{J}) - f_{(n-k)\times k}=\bar{\bf g}_{\Yng(1)^{\rm op}}(p_J).
\end{eqnarray*}
This implies $-\psi(\Delta_{\val_{\Yng(1)}}(D_{n-k},p_{(n-k)\times k})=\Delta_{\bar{\bf g}_{\Yng(1)}^{\rm op}}(D_{n-k},p_{(n-k)\times k})$. 
\end{proof}

\begin{remark}\label{rmk:val and cval}
    The attentive reader might notice that the Theorem~\ref{thm: val and gv} and the discussion preceding it closely resemble Lemma~\ref{lem:cval_gval}.
    However, the difference in convention choices in \cite{RW} and the present paper yield the necessity of a non-trivial change of coordinates.
    To avoid lengthening the exposition even more we decided to give the result in a single seed but allude to the fact that Theorem~\ref{thm: val and gv} indeed is an instance of Lemma~\ref{lem:cval_gval} (after non-trivial changes of cluster coordinates).
    After making the appropriate change of coordinate one can show that the map $-\psi$ may be described as the tropicalization of a cluster ensemble map. In particular, Theorem~\ref{thm: val and gv} can be extended to all seeds.
\end{remark}

\subsection{The intrinsic Newton--Okounkov body for Grassmannians}
\label{sec:Grass_intrinsic}
As before, let $\lb_e$ be the bundle over $\text{Gr}_{n-k}(\C^n)$ obtained by pullback of $\mathcal{O}(1)$ under the Pl\"ucker embedding $\text{Gr}_{n-k}(\C^n)\hookrightarrow \mathbb{P}^{\binom{n}{k}-1}$. Recall the definition of the intrinsic Newton--Okounkov body from Definition~\ref{def:intrinsic_lb}.

\begin{corollary}\label{cor:intrinsicNO grassmannian}
Consider the partial minimal model $\cA^{\rm op}\hookrightarrow \UT_{Y} $, the minimal model $\cA^{\op}/T_K \hookrightarrow Y$ and the map ${\bf g}:\mathbb{B}_{\tf}(\cA^{\op})\to \Trop_{\Z}((\cA^{\op})^\vee)$ of \eqref{eq:nu_seed_free}. Then
\eqn{
\Delta_{\mathrm{BL}}(\lb_e) = \bconv\Bigg( \lrc{ \gv \lrp{p_J} \mid J \in \binom{[n]}{n-k}} \Bigg).
}
\end{corollary}

\begin{proof}
The Newton--Okounkov polytope for the flow valuation with respect to $\seed_{\Yng(1)}$ is the convex hull of the images of Pl\"ucker coordinates (see \cite[\S16.1]{RW}).
So by Theorem~\ref{thm: val and gv} the same is true for the Newton--Okounkov body $\Delta_{\bar{\gv}_{\Yng(1)^{\op}}}(D_{n-k},p_{(n-k)\times k})$.
By Theorem~\ref{NO_bodies_are_positive},  $\Delta_{\bar{\gv}_{\Yng(1)^{\op}}}(D_{n-k},p_{(n-k)\times k})$ is positive, hence it is broken line convex.
Therefore, the broken line convex hull of the set $\lrc{ \bar{\gv}_{\Yng(1)^{\op}} \lrp{p_J} \mid J \in \binom{[n]}{n-k}} $ in the lattice  $\Trop_\R({\cA^{\op}/T_{K}}^{\vee}_{\seed_{\Yng(1)^{\op}}})$  
coincides with its convex hull.
We take into account Remark~\ref{rem:comparing_NO_bodies} to get that $ \Delta_{\gv_{\Yng(1)^{\op}}}(\lb_e) = \Delta_{\bar{\gv}_{\Yng(1)^{\op}}}(D_{n-k},p_{(n-k)\times k}) + \gv_{\Yng(1)^{\op}}(p_{(n-k)\times k}) $.
This implies that $ \Delta_{\gv_{\Yng(1)^{\op}}}(\lb_e) $ is the broken line convex hull of the set $\lrc{ \gv_{\Yng(1)^{\op}} \lrp{p_J} \mid J \in \binom{[n]}{n-k}} $.
Being a broken line convex set is independent of the choice of seed, the result follows.
\end{proof}

\begin{remark}
In the proof of Corollary~\ref{cor:intrinsicNO grassmannian}, we implicitly use the $\cA$ cluster structure to view the intrinsic Newton--Okounkov body $\Delta_{\mathrm{BL}}(\lb_e)$ as the broken line convex hull of tropical points indexing Pl\"ucker coordinates.
We could alternatively use the $\cX$ cluster structure as Rietsch--Williams do, and define theta functions with the corresponding $\cX$ scattering diagram.
By identifying the Rietsch--Williams valuation with the $\cv$-vector valuation (Corollary~\ref{cor:gv on midX}) and noting that there is a cluster ensemble automorphism of the open positroid variety (see \cite[Theorem~7.1, Corollary~5.11]{MullSp}, \cite[Theorem~7.3, Proposition~7.4]{RW}),
we can apply Lemma~\ref{lem:cval_gval} to give a completely analogous statement to Corollary~\ref{cor:intrinsicNO grassmannian} which uses the Rietsch--Williams valuation rather than the $\gv$-vector valuation.
In fact, in \S\ref{sec:Gr36} we present an example of an explicit computation related to the intrinsic Newton--Okounkov body $\Delta_{\mathrm{BL}}(\lb_e)$ defined via the $\Xnet$ scattering diagram.
\end{remark}

\subsubsection{Example}\label{sec:Gr36}
In this subsection, we will give an example of the intrinsic Newton--Okounkov body for the case of $\Grass_3\lrp{\C^6}$ and compare this to a Newton--Okounkov body of \cite{RW}. 
In particular, in \cite[\S9]{RW}, Rietsch--Williams discuss a non-integral vertex appearing in the Newton--Okounkov body $\Delta_{\val_{G}}(D_{3},p_{123})$ associated to the plabic graph $G$ of Figure~\ref{fig:3-6}.
We illustrate how this non-integral vertex in the usual Newton--Okounkov body framework corresponds to a point in the interior of a broken line segment in $\Delta_{\mathrm{BL}}(\lb_e)$ and thus is not a genuine vertex from the intrinsic Newton--Okounkov body perspective.
Here, to facilitate comparison with \cite{RW}, we will view the open positroid variety as $\Xnet_{\vb{1}}$ (up to codimension 2).
So, the scattering diagram we use to define $\Delta_{\mathrm{BL}}(\lb_e)$ in the subsection will be $\scat^{\Xnet_{\mathbf{1}}}_{\text{in},\seed_G}$ for a particular choice of initial seed $\seed_G$.  
The choice of seed is encoded by the plabic graph illustrated in Figure~\ref{fig:3-6}.

\begin{figure}[ht]
    \centering
	
\tikzexternaldisable
\begin{tikzpicture}[scale=.5]
\tikzset{->-/.style={decoration={
  markings,
  mark=at position #1 with {\arrow{>}}},postaction={decorate}}}
  \tikzset{-<-/.style={decoration={
  markings,
  mark=at position #1 with {\arrow{<}}},postaction={decorate}}}
   
\draw (0,0) circle [radius=5];  
 
\draw[-<-=.5] (-1,2) -- (1,2);
\draw[->-=.5] (1,2) -- (2.25,0);
\draw[->-=.5] (2.25,0) -- (1,-2);
\draw[->-=.5] (1,-2)-- (-1,-2);
\draw[-<-=.5] (-1,-2) -- (-2.25,0);
\draw[-<-=.5] (-2.25,0)-- (-1,2);
\draw[-<-=.5] (-1,2)  -- (-1,3.5);
\draw[->-=.5] (-1,3.5)-- (1,3.5);
\draw[->-=.5] (1,3.5) -- (1,2);
\draw[-<-=.5] (2.25,0) -- (3.5,-1);
\draw[->-=.7] (3.5,-1) -- (2.375,-3);
\draw[-<-=.5] (2.375,-3) -- (1,-2);
\draw[->-=.5] (-2.25,0) -- (-3.5,-1);
\draw[->-=.5] (-3.5,-1) -- (-2.375,-3);
\draw[-<-=.5] (-2.375,-3) -- (-1,-2);
\draw[->-=.5] (-1,4.9) -- (-1,3.5);
\draw[->-=.5] (1,4.9) -- (1,3.5);
\draw[->-=.5] (4.7,-1.75) -- (3.5,-1);
\draw[->-=.5] (2.375,-3) -- (3.35,-3.75);
\draw[->-=.5] (-3.5,-1) -- (-4.7,-1.75);
\draw[->-=.5] (-2.375,-3) -- (-3.35,-3.75);

\draw[fill] (1,3.5) circle [radius=.175];  
\draw[fill] (-1,2) circle [radius=.175];  
\draw[fill] (2.25,0) circle [radius=.175];  
\draw[fill] (2.375,-3) circle [radius=.175];  
\draw[fill] (-1,-2) circle [radius=.175];  
\draw[fill] (-3.5,-1) circle [radius=.175];  

\draw[fill, white] (-1,3.5) circle [radius=.175];  
\draw (-1,3.5) circle [radius=.175];  
\draw[fill, white] (1,2) circle [radius=.175];  
\draw (1,2) circle [radius=.175];  
\draw[fill, white] (3.5,-1) circle [radius=.175];  
\draw (3.5,-1) circle [radius=.175];
\draw[fill, white] (1,-2) circle [radius=.175];  
\draw (1,-2) circle [radius=.175];  
\draw[fill, white] (-2.25,0) circle [radius=.175];  
\draw (-2.25,0) circle [radius=.175];  
\draw[fill, white] (-2.375,-3) circle [radius=.175];  
\draw (-2.375,-3) circle [radius=.175];  

\node[above] at (-1,4.9) {1};
\node[above] at (1,4.9) {2};
\node[right] at (4.7,-1.75) {3};
\node[right] at (3.35,-3.75) {4};
\node[left] at (-4.7,-1.75) {6};
\node[left] at (-3.35,-3.75) {5};

\node at (0,0) {\scalebox{.3}{ $\yng(2,1)$}};
\node at (0,2.75) {\scalebox{.3}{$\yng(2)$}};
\node at (0,4.25) {\scalebox{.3}{$\yng(3)$}};
\node at (3,1.75) {\scalebox{.3}{$\yng(3,3)$}};
\node at (2.3,-1.5) {\scalebox{.3}{$\yng(3,3,2)$}};
\node at (3.7,-2.25) {\scalebox{.3}{$\yng(3,3,3)$}};
\node at (0,-3.5) {\scalebox{.3}{$\yng(2,2,2)$}};
\node at (-2.3,-1.5) {\scalebox{.3}{$\yng(1,1)$}};
\node at (-3.7,-2.25) {\scalebox{.3}{$\yng(1,1,1)$}};
\node at (-3,1.75) {{$\varnothing$}};

\begin{scope}[xshift=14cm]

\def\op{.4}

\draw[opacity=\op, name path = boundary] (0,0) circle [radius=5];  
 
\node [circle, draw=black, fill=black, inner sep=0pt, minimum size=5pt, opacity=\op] (1b) at (-1,2) {}; 
\node [circle, draw=black, fill=black, inner sep=0pt, minimum size=5pt, opacity=\op] (2b) at (1,3.5) {}; 
\node [circle, draw=black, fill=black, inner sep=0pt, minimum size=5pt, opacity=\op] (3b) at (2.25,0) {}; 
\node [circle, draw=black, fill=black, inner sep=0pt, minimum size=5pt, opacity=\op] (4b) at (2.375,-3) {}; 
\node [circle, draw=black, fill=black, inner sep=0pt, minimum size=5pt, opacity=\op] (5b) at (-1,-2) {}; 
\node [circle, draw=black, fill=black, inner sep=0pt, minimum size=5pt, opacity=\op] (6b) at (-3.5,-1) {}; 

\node [circle, draw=black, fill=white, inner sep=0pt, minimum size=5pt, opacity=\op] (1w) at (-1,3.5) {};
\node [circle, draw=black, fill=white, inner sep=0pt, minimum size=5pt, opacity=\op] (2w) at (1,2) {};
\node [circle, draw=black, fill=white, inner sep=0pt, minimum size=5pt, opacity=\op] (3w) at (3.5,-1) {};
\node [circle, draw=black, fill=white, inner sep=0pt, minimum size=5pt, opacity=\op] (4w) at (1,-2) {};
\node [circle, draw=black, fill=white, inner sep=0pt, minimum size=5pt, opacity=\op] (5w) at (-2.375,-3) {};
\node [circle, draw=black, fill=white, inner sep=0pt, minimum size=5pt, opacity=\op] (6w) at (-2.25,0) {};

\draw[color=gray, opacity=\op] (1b) -- (2w) ;
\draw[color=gray, opacity=\op] (2w) -- (3b) ;
\draw[color=gray, opacity=\op] (3b) -- (4w) ;
\draw[color=gray, opacity=\op] (4w) -- (5b) ;
\draw[color=gray, opacity=\op] (5b) -- (6w) ;
\draw[color=gray, opacity=\op] (6w) -- (1b) ;

\draw[color=gray, opacity=\op] (1b) -- (1w) ;
\draw[color=gray, opacity=\op] (1w) -- (2b) ;
\draw[color=gray, opacity=\op] (2b) -- (2w) ;
\draw[color=gray, opacity=\op] (3b) -- (3w);
\draw[color=gray, opacity=\op] (3w) -- (4b);
\draw[color=gray, opacity=\op] (4b) -- (4w);
\draw[color=gray, opacity=\op] (6w) -- (6b);
\draw[color=gray, opacity=\op] (6b) -- (5w);
\draw[color=gray, opacity=\op] (5w) -- (5b);

\path [name path = 1p] (1w) -- (-1,5);
\path [name intersections={of=boundary and 1p, by = 1e}];
\path [name path = 2p] (2w) -- (1,5);
\path [name intersections={of=boundary and 2p, by = 2e}];
\path [name path = 3p] (3b) --++ (2.5,-2);
\path [name intersections={of=boundary and 3p, by = 3e}];
\path [name path = 4p] (4b) --++ (1.25,-1);
\path [name intersections={of=boundary and 4p, by = 4e}];
\path [name path = 5p] (5w) --++ (-1.25,-1);
\path [name intersections={of=boundary and 5p, by = 5e}];
\path [name path = 6p] (6b) --++ (-1.25,-1);
\path [name intersections={of=boundary and 6p, by = 6e}];

\draw[opacity=\op] (1w) -- (1e) node [pos=1, above,opacity=1] {$1$};
\draw[opacity=\op] (2b) -- (2e) node [pos=1, above,opacity=1] {$2$};
\draw[opacity=\op] (3w) -- (3e) node [pos=1, right,opacity=1] {$3$};
\draw[opacity=\op] (4b) -- (4e) node [pos=1, right,opacity=1] {$4$};
\draw[opacity=\op] (5w) -- (5e) node [pos=1, left,opacity=1] {$5$};
\draw[opacity=\op] (6b) -- (6e) node [pos=1, left,opacity=1] {$6$};

\node (246) at (0,0) {\footnotesize $246$};
\node (256) at (0,2.75) {\footnotesize $256$};
\node (156) at (0,4.25) {\footnotesize $156$};
\node (126) at (3,1.75) {\footnotesize $126$};
\node (124) at (2.3,-1.5) {\footnotesize $124$};
\node (123) at (3.8,-2.35) {\footnotesize $123$};
\node (234) at (0,-4) { \footnotesize $234$};
\node (346) at (-2.3,-1.5) {\footnotesize $346$};
\node (345) at (-3.8,-2.35) {\footnotesize $345$};
\node (456) at (-3,1.75) {\footnotesize $456$};

\draw [->] (246) -- (456);
\draw [->] (246) -- (126);
\draw [->] (246) -- (234);

\draw [->] (256) -- (246);
\draw [->] (124) -- (246);
\draw [->] (346) -- (246);

\draw [->] (256) -- (156);

\draw [->]  (2.5,-1.8)  -- (3.35,-2.25) ;

\draw [->] (-2.5,-1.8)  -- (-3.35,-2.25) ;

\draw [->] (456) -- (346);
\draw [->] (234) -- (346);
\draw [->] (456) -- (256);
\draw [->] (126) -- (256);
\draw [->] (126) -- (124);
\draw [->] (234) -- (124);

\end{scope}

\end{tikzpicture}

\tikzexternalenable

    \caption{A plabic graph $G$ for $\Grass_3\lrp{\C^6}$ for which $\Delta_{\val_{G}}(D_{3},p_{123})$ has non-integral vertex (see \cite[\S9]{RW}). On the left, the labels are in terms of Young diagrams. On the right, we display the quiver and label faces by Pl\"ucker coordinates.}
    \label{fig:3-6}
\end{figure}
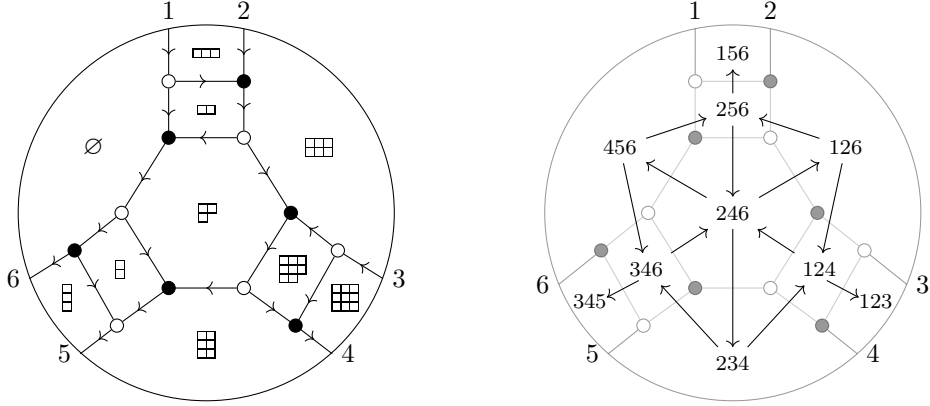

Recall that the Young diagrams in Figure~\ref{fig:3-6} label the network parameters used in flow polynomial expressions (see \cite[Equation (6.3)]{RW}).
The $\cA$-cluster determined by trips in the plabic graph $G$ consists of the Pl\"ucker coordinates whose indices are given in Figure~\ref{fig:3-6} (see \cite[Definition 3.5]{RW}).

According to \cite[\S9]{RW}, a non-integral vertex in the Newton--Okounkov polytope comes from half the valuation of the flow polynomial for the element $ f = (p_{124} p_{356} - p_{123} p_{456}) / {p_{123}^2}$.
They compute $\frac{1}{2}\val_G(f)$ and express its entries in tabular form (see \cite[Table~3]{RW}) as we have reproduced in Table~\ref{table}.
The function $p_{123}^2\, f$ is one of the two $\cA$ cluster variables that are not Pl\"ucker coordinates, see {\it e.g.} \cite[Eq.(4), p.42]{Sco06}. 
It is obtained by mutation at $\Yng(2,1)$.

\begin{table}
\begin{center}
\captionsetup{type=table}
\begin{tabular}{|C|C|C|C|C|C|C|C|C|} 
\hline
  \vphantom{\Yng(1,1,1,1)} \Yng(3,3,3)  & \Yng(3,3,2)  & \Yng(2,2,2)  & \Yng(1,1,1)  & \Yng(3,3)  & \Yng(2,1)  & \Yng(1,1)  & \Yng(3)  & \Yng(2)\\
\hline
     \vphantom{\Yng(1,1,1)_{\Yng(1,1,1)}}\frac{3}{2} &     \frac{3}{2} & 1 &   \frac{1}{2} & 1 & \frac{1}{2} &\frac{1}{2} &\frac{1}{2} &\frac{1}{2} \\
\hline
\end{tabular} 
\captionof{table}{The rational vertex $\frac{1}{2}\val_G(f)$\label{table}}
\end{center}
\end{table}

Note we can re-interpret the expression for $f$ as the expansion of a product of theta functions.
All Pl\"ucker coordinates are $\cA$ cluster variables, and all $\cA$ cluster monomials are theta functions.
Then
\eqn{p_{124}\, p_{356} = 
p_{123}^2\, f + p_{123}\, p_{456},}
and the right hand side is a sum of two theta functions. 
This means there are only two balanced pairs of broken lines contributing to the product.
We will see that the pair with no bending corresponds to the summand $ p_{123}\, p_{456}$,
while the other involves a maximal bend at an initial wall.
(Since the bend is at an initial wall, we are able to see the relevant broken line segment without constructing the consistent scattering diagram.)

We interpret the Rietsch--Williams valuation as being valued in $ \Trop_{\Z}((\cX_{\mathbf{1}})^{\vee})$ (see Proposition~\ref{prop:flow is gv for X}) and consider broken lines in the associated $\cX$ cluster scattering diagram.
The choice of seed identifies $\Trop_{\Z}((\cX_{\mathbf{1}})^\vee)$ with $ N^\vee / {\lra{(1,1,\dots, 1)}}$ and we draw the scattering in $\lrp{ N^\vee / {\lra{(1,1,\dots, 1)}}} \otimes \R$.

We will use Figure \ref{fig:3-6} to define the fixed data and the seed data for the cluster structure. 
The initial scattering diagram for the $\Xnet$ variety is 
\[ \mathfrak{D}^{\cX_{}}_{\text{in},\seed_G}= \lrc{   \left( (v_{\mu})^{\perp}, \  1+ z^{e_{\mu}}\right) \mid \mu \in \{ \Yng(2), \Yng(2,1), \Yng(3,3,2), \Yng(1,1) \}}.\]
To get initial scattering diagram for the fibre over $\mathbf{1}$ we take the quotient of the support $\mathfrak{D}^{\Xnet}_{\text{in},\seed_G}$ by $\left(\R \cdot (1,1,\dots, 1) \right)$. (Observe that $(v_{\mu})^{\perp}$ is invariant under translations by $\R \cdot (1,1,\dots, 1)$.) 

\[ \mathfrak{D}^{\cX_{\mathbf{1}}}_{\text{in},\seed_G}= \lrc{   \left( (v_{\mu})^{\perp}/ \left(\R \cdot (1,1,\dots, 1) \right) , \  1+ z^{e_{\mu}}\right) \mid \mu \in \{ \Yng(2), \Yng(2,1), \Yng(3,3,2), \Yng(1,1) \}}\]

All pertinent valuations may be found in \cite[Table 3]{bossinger2019full}.
We record them here using the ordering of Table~\ref{table}. We choose the representative whose coefficient of $e_\varnothing$ is $0$ and do not record this entry.
\eqn{\val_{G}({p}_{124}) = (1,0,0,0,0,0,0,0,0)}
\eqn{\val_{G}({p}_{356}) = (2,2,2,1,2,1,1,1,1)}
\eqn{\val_{G}({p}_{123}) = (0,0,0,0,0,0,0,0,0) }
\eqn{\val_{G}({p}_{456}) = (3,2,2,1,2,1,1,1,1) }
So, we have $\val_{G}({p}_{124}) + \val_G( p_{356} ) = \val_{G}( p_{123} \, p_{456}  )$.
The summand $p_{123} \, p_{456} $ in the product $p_{124} \, p_{356}$ corresponds to the straight broken line segment from $ \val_{G}({p}_{356})$ to $\val_G( p_{124} )$, whose midpoint is $\frac{1}{2}  \val_{G}( p_{123} \, p_{456}  )$. 

The bending wall for the other broken line segment is $\lrp{(v_{\Yngs(3,3,2)})^{\perp}, 1+z^{e_{\Yngs(3,3,2)}}}$.
Note that $v_{\Yngs(3,3,2)}= f_{\Yngs(3,3,3)} + f_{\Yngs(2,1)}-f_{\Yngs(3,3)} - f_{\Yngs(2,2,2)}$, and $\frac{1}{2}\val(f) = \frac{1}{2}\val(p_{123}^2\, f)$ is perpendicular to this vector.
So, $\frac{1}{2}\val(p_{123}^2\, f)$ lies in the support of this wall.
We will see that there is a broken line segment from $\val_{G}(p_{356})$ to $\val_{G}(p_{124})$ passing through $\frac{1}{2}\val(f)$ and bending maximally here, as depicted in Figure~\ref{fig:maxbend}.

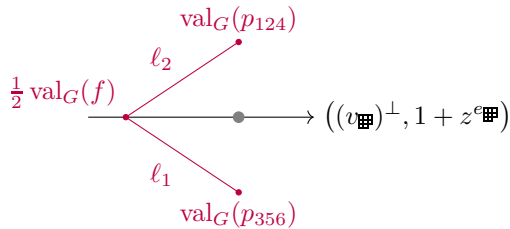
\begin{figure}[ht]
 \centering
    \begin{tikzpicture}
    \draw[->] (-2,0) -- (1,0) node[anchor=west]{$\lrp{(v_{\Yngs(3,3,2)})^{\perp}, 1+z^{e_{\Yngs(3,3,2)}}}$};
    \filldraw [gray] (0,0) circle (2pt);
    \filldraw [purple] (-1.5,0) circle (1pt) node[anchor=south east]{$\frac{1}{2}\val_G(f)$};
    \draw[purple] (-1.5,0) -- node[anchor=south east]{$ \ell_2$} (0,1);
    \filldraw [purple]  (0,1) circle (1pt) node[anchor=south]{$\val_G(p_{124})$};
    \draw[purple] (-1.5,0) --node[anchor=north east]{$\ell_1$} (0,-1);
    \filldraw [purple]  (0,-1) circle (1pt) node[anchor=north]{$\val_G(p_{356})$};
    \end{tikzpicture}
 \caption{Rational point obtained from broken line bending maximally at an initial wall. In this broken line segment, $\ell_1$ and $\ell_2$ will take equal time, corresponding to the summand $p_{123}^2\, f$ in the product $p_{124} \, p_{356}$.
\label{fig:maxbend}} 
\end{figure}

Recall that the exponent vector of the decoration monomial along $\ell_i$ is the negative of the velocity vector there. 
Traveling along $\ell_1$, this velocity vector is positively proportional to 
\[\frac{1}{2}\val_G(f) - \val_G(p_{356}) = 
-\left(\frac{1}{2}, \frac{1}{2}, 1, \frac{1}{2},  1, \frac{1}{2} , \frac{1}{2}, \frac{1}{2}, \frac{1}{2}\right).\] 
We can take a broken line with exponent vector $ v_1=(1,1,2,1,2,1,1,1,1) $ along $\ell_1$.
The possible bendings after crossing the wall correspond to summands of
\eqn{z^{v_1}\lrp{1+z^{e_{\Yngs(3,3,2)}}}^{-\lra{v_1,v_{\Yngs(3,3,2)}}}=z^{v_1}\lrp{1+z^{e_{\Yngs(3,3,2)}}}^{2}.}
The maximal bending corresponds to the summand $z^{v_1+ 2 e_{\Yngs(3,3,2)}} = z^{(1,3,2,1,2,1,1,1,1)}$. Let us call $v_2 := (1,3,2,1,2,1,1,1,1)$.
Then observe that $\frac{1}{2}\val_G(f) - \frac{1}{2}v_2 = \val_G(p_{124})$. 
So, we have a broken line segment $\gamma$ traveling from $\val_G(p_{356})$ to $\frac{1}{2}\val_G(f)$ with decoration monomial $z^{v_1}$, bending maximally and continuing to $\val_G(p_{124})$ with decoration monomial $z^{v_2}$. Precisely $\frac{1}{2}$ a unit of time is spent in each straight segment.
From this perspective, $\frac{1}{2}\val_G(f)$ is not a genuine vertex; $\frac{1}{2}\val_G(f)$ is in the relative interior of the support of $\gamma$, and the endpoints of $\gamma$ are in the Newton--Okounkov body.

\end{document}